\documentclass[12pt,twoside,reqno]{amsart}
\usepackage{amssymb,amsmath,amstext,amsthm,amsfonts}
\usepackage[ansinew]{inputenc} 
\usepackage{graphicx}
\usepackage[mathscr]{eucal}
\usepackage{hyperref}

\newcommand{\blob}{\rule[.2ex]{.8ex}{.8ex}}

\newcommand{\R}{\mathbb{R}}
\newcommand{\C}{\mathbb{C}}
\newcommand{\N}{\mathbb{N}}
\newcommand{\Z}{\mathbb{Z}}
\newcommand{\Q}{\mathbb{Q}}
\newcommand{\T}{\mathbb{T}}
\newcommand{\SL}{{\rm SL}}
\newcommand{\GL}{{\rm GL}}

\newcommand{\Mat}{{\rm Mat}}

\newcommand{\tops}{\mathscr{X}}

\newcommand{\strip}{\mathbb{S}}

\newcommand{\sabs}[1]{\left| #1 \right|} 
\newcommand{\abs}[1]{\bigl| #1 \bigr|} 
\newcommand{\norm}[1]{\lVert#1\rVert} 
\newcommand{\normr}[1]{\lVert#1\rVert_r} 
\newcommand{\intpart}[1]{\lfloor #1 \rfloor} 

\newcommand{\transl}{{\rm T}} 

\newcommand{\less}{\lesssim}

\newcommand{\ep}{\epsilon} 
 \newcommand{\ka}{\kappa} 
\newcommand{\la}{\lambda}
\newcommand{\ga}{\gamma}
\newcommand{\La}{\Lambda}
\newcommand{\om}{\omega}

\newcommand{\vpsi}{\vec{\psi}}
\newcommand{\GLmR}{\GL(m, \R)}
\newcommand{\cocycles}{C^{\om}_{r} (\T, \GL(m, \R))}
\newcommand{\cocyclesTd}{C^{\om}_{r} (\T^d, \GL(m, \R))}
\newcommand{\cocycle}[3]{C^{\om}_{r} ({#1}, \GL({#2}, {#3}))}

\newcommand{\gabar}{\bar{\gamma}}
\newcommand{\etabar}{\bar{\eta}}
\newcommand{\deltabar}{\bar{\delta}}

\newcommand{\scale}{\mathscr{N}}  

\newcommand{\An}[1]{A^{({#1})}}  
\newcommand{\Bn}[1]{B^{({#1})}}  

\newcommand{\Larn}[1]{\Lambda^{({#1})}_\rho}  
\newcommand{\Lapn}[1]{\Lambda^{({#1})}_\pi}   
\newcommand{\Lasn}[1]{\Lambda^{({#1})}_s}   
\newcommand{\Lajn}[1]{\Lambda^{({#1})}_{p_j}}   
\newcommand{\Lamn}[1]{\Lambda^{({#1})}_{p_m}}   

\newcommand{\usn}[1]{u^{({#1})}_s}   
\newcommand{\upjn}[1]{u^{({#1})}_{p_j}}   

\newcommand{\Lar}{\Lambda_\rho}  
\newcommand{\Lap}{\Lambda_\pi}   
\newcommand{\Las}{\Lambda_s}   
\newcommand{\Laj}{\Lambda_{p_j}}   
\newcommand{\Lam}{\Lambda_{p_m}}   

\newcommand{\Gr}{{\rm Gr}}
\newcommand{\Pp}{\mathbb{P}}
\newcommand{\FF}{\mathscr{F}}

\newcommand{\hatv}{\hat{v}}
\newcommand{\hatV}{\hat{V}}

\newcommand{\tausvp}{\tau\text{ - s.v.p.}} 
\newcommand{\tausvr}{\tau\text{ - s.v.r.}}  
\newcommand{\taugp}{\tau\text{-gap pattern}}  
\newcommand{\taublockL}{\text{Lyapunov spectrum } \tau\text{-block}}  
\newcommand{\taugapL}{\text{Lyapunov spectrum } \tau\text{-gap}}  

\newcommand{\nzero}{n_{00}}

\newcommand{\B}{\mathscr{B}}
\newcommand{\Bt}{\tilde{\mathscr{B}}}
\newcommand{\Bbar}{\bar{\mathscr{B}}}
\newcommand{\allsvf}{\mathscr{S}}
\newcommand{\E}{\mathscr{E}}

\newcommand{\dist}{{\rm dist}}

\newcommand{\filt}{\underline{F}}

\newcommand{\filtn}[1]{\filt^{({#1})}}

\newcommand\restr[2]{{
  \left.\kern-\nulldelimiterspace 
  #1 
  \vphantom{\big|} 
  \right|_{#2} 
  }}

\theoremstyle{plain}
\newtheorem{theorem}{Theorem}[section]
\newtheorem{proposition}{Proposition}[section]
\newtheorem{corollary}[proposition]{Corollary}
\newtheorem{lemma}[proposition]{Lemma}

\newtheorem{remark}{Remark}[section]

\numberwithin{equation}{section}

\title[Continuity of the Lyapunov exponents]{Continuity of the Lyapunov exponents for quasiperiodic cocycles}

\date{}

\begin{document}

\author[P. Duarte]{Pedro Duarte}
\address{Departamento de Matem\'atica and CMAF\\
Faculdade de Ci\^encias\\
Universidade de Lisboa\\
Portugal 
}
\email{pduarte@ptmat.fc.ul.pt}

\author[S. Klein]{Silvius Klein}
\address{CMAF\\ Faculdade de Ci\^encias\\
Universidade de Lisboa\\
Portugal\\ 
and IMAR, Bucharest, Romania }
\email{silviusaklein@gmail.com}

\begin{abstract} Consider the Banach manifold of real analytic linear cocycles with values in the general linear group of any dimension and base dynamics given by a Diophantine translation on the circle. 
We prove a precise higher dimensional Avalanche Principle and use it in an inductive scheme to show that the Lyapunov spectrum blocks associated to a gap pattern in the Lyapunov spectrum of such a cocycle are locally H\"{o}lder continuous. Moreover, we show that all Lyapunov exponents are continuous everywhere in this Banach manifold, irrespective of any gap pattern in their spectra. \\
These results also hold for Diophantine translations on higher dimensional tori,  albeit with a loss in  the modulus of continuity of the Lyapunov spectrum blocks.
\end{abstract}

\maketitle

\section{Introduction and main statements}\label{main-thms_section}


A linear cocycle is a  dynamical system on a vector bundle 
such that the  action on the base is fiber independent,  and the action on each fiber is linear. 
In particular, a linear cocycle determines
a dynamical system on the base space, usually referred to as its {\em base dynamics}. 
To simplify matters, the bundle is usually assumed to be trivial, i.e.
of the form $B=X\times\R^m$, in which case
the cocycle acts on fibers  through linear transformations in  the group $\GL(m,\R)$.
Given a subgroup $G$ of $\GL(m,\R)$, we call a $G$-valued cocycle  one that
acts on fibers through linear transformations in   $G$.
The {\em Lyapunov exponents} of a linear cocycle {\em measure} the growth rate of fiber vectors
along the cocycle dynamics. Precise definitions will be given below.

Given a discrete Schr\"odinger operator, solving formally the
corresponding finite differences equation  gives rise to a 
linear cocycle, which is usually called  a
{\em Schr\"odinger cocycle}.
When applied to the particular case of Schr\"odinger cocycles, properties such as positivity and continuity 
of the Lyapunov exponents have important implications to 
spectral problems for the corresponding operator.
Such implications, regarding the nature of the spectrum and the modulus of continuity of the integrated density of states of the operator  follow from  results of 
Ishii, Pastur, Kotani and Thouless.

Two  classes of general linear cocycles have
been extensively studied so far, 
reflecting  two paradigmatic ergodic behaviors
of the base dynamics:
the class of {\em random cocycles}, where the base dynamics is  a Bernoulli shift,
and the class of {\em quasiperiodic  cocycles}, where the base dynamics is
a torus translation.

The main  purpose of this work is to establish the continuity
of all Lyapunov exponents for quasiperiodic cocycles.

The first continuity result (in fact  H\"older  continuity) for general random cocycles   is due to Émile Le Page. This result requires strong irreducibility and contraction assumptions on the cocycle (see \cite{lePage} and  references therein).

The case of {\em random}, $\GL (2, \C)$ valued cocycles (with no additional assumptions) has been studied by C. Bocker-Neto and M. Viana in \cite{BV}. They prove continuity  of the Lyapunov exponents $L_1$, $L_2$ and of the invariant Oseledets subspaces $E_1 (x)$, $E_2 (x)$ as functions of the input data (cocycle and probability distribution). We are not aware of any higher dimensional version of this result, or of any quantitative (e.g. H\"older) description of the continuity in this setting.

The problem of continuity of the Lyapunov exponents for {\em quasiperiodic, two dimensional} cocycles has been widely studied. 

Sharp results are available especially for {\em Schr\"odinger}, $\SL (2, \R)$ - valued cocycles, where the Lyapunov exponent is  seen as a function of energy (see \cite{GS-Holder}, on which many ideas of this paper are based), or jointly as a function of frequency and energy (see \cite{BJ-jointcont}, \cite{B-contpos-Td}). 

Joint continuity in frequency and cocycle has been obtained even for cocycles with singularities, i.e. $\Mat (2, \C)$-valued cocycles (see \cite{JMarx-CMP12}).

Similar problems for {\em higher dimensional} quasiperiodic cocycles have been studied more recently.   
Our paper was originally motivated by \cite{Schlag}, where H\"older continuity is proven for Schr\"odinger - like cocycles, under the assumption that the Lyapunov spectrum is simple. 

A result on joint continuity in frequency and cocycle of all Lyapunov exponents has been announced in 2012 by A. \'{A}vila, J. Jitomirskaya and C. Sadel (see \cite{AJS}). Their method does not require invertibility of the cocycle. 

The results in this paper differ from ~\cite{AJS} in that while we require invertibility of the cocycle at all points, and we only consider the problem of continuity of the Lyapunov exponents as functions of the cocycle, we obtain a {\em quantitative} (i.e. H\"{o}lder) description of the modulus of continuity. Moreover, our method and results also apply to translations on higher dimensional tori.

\smallskip
 

Given a probability space $(X, \mu)$ and given an ergodic transformation $\transl \colon X \to X$, any measurable function 
  $A: X \to \GLmR$ determines a skew-product map
$F = (\transl, A) : X \times\R^m \to  X \times \R^m$ defined by $F (x,v)=(\transl x, A(x)\,v)$.
The dynamical  system underlying such a map is called a {\em linear cocycle}
over the transformation $T$.
When $\transl$ is fixed, the measurable function $A$ is also referred to as a linear cocycle.
The iterates of $F$ are given by  $(\transl^n \, x , \An{n} (x)\,v)$,  where
$$\An{n} (x)= A(\transl^{n-1} \, x) \cdot  \ldots \cdot A(\transl x) \cdot A(x)$$
The cocycle is called {\em integrable }\, when
$\int_X \log^+ \norm{A(x)^{\pm 1}}\, d\mu(x) <\infty$,
where \, $\log^+(x) = \max\{\log x, 0\}$.
In 1965 Oseledets proved his famous Multiplicative Ergodic Theorem, 
which when applied to cocycles  $(A, \transl)$ with $\transl$ invertible,  says that if 
$A$ is integrable  then there are: numbers $L_1 (A) \geq L_2 (A) \geq \ldots \geq L_m (A) $,
 an $F$-invariant measurable decomposition
$\R^m=\oplus_{j=1}^{n} E_{j} (A) \, (x)$,
and a non decreasing surjective map $k:\{1,\ldots, m\}\to \{1,\ldots, n\}$ such that
for $\mu$-almost every $x\in X$, every $1\leq i\leq m$ and every $v\in E_{k_i} (A) \, (x) \setminus \{0\}$,
$$L_i (A) = \lim_{n\to\pm \infty}\frac{1}{n}\,\log \norm{\An{n} (x)\, v}\;$$ 
Moreover,   $L_i (A) =L_{i+1} (A)$ \, if and only if $k_{i}=k_{i+1}$,
and the subspace  $E_{j} (A) \,  (x)$ has dimension equal to $\#k^{-1}(j)$.
The numbers $L_i (A) $ are called the {\em Lyapunov exponents}  of $F$ (or of $A$).

If we denote by $s_1 (g) \ge s_2 (g) \ge \ldots \ge s_m (g)$  the singular values of a matrix $g \in \GLmR$, 
then  it is easy to verify that  the Lyapunov exponents of $A$ are
\begin{equation} \label{formula-lyap-i}
 L_i (A) = \lim_{n\to\infty}\frac{1}{n}\,\log s_i (\An{n} (x)) \quad \text{ for }\; \mu \text{-a.e.}\, x\in X
\quad (1\leq i\leq m)
\end{equation}
In particular,  the largest Lyapunov exponent is
\begin{align}
L_1 (A)  &= \lim_{n\to\infty}\frac{1}{n}\,\log \norm{ \An{n} (x) }\;\;
\text{ for }\; \mu \text{-a.e.}\, x\in X  \nonumber \\
&= \lim_{n\to\infty}\frac{1}{n}\,\int_X \log \norm{ \An{n} (x) }\,d\mu(x)\label{formula-lyap}
\end{align}

\smallskip


Our paper is concerned with certain quasiperiodic cocycles defined as follows. 

Let the phase space $X$ be the additive 
group $\T=\R/\Z$.  Given a frequency  $\omega\in\R\setminus\Q$,
consider the translation $\transl = \transl_\om \colon \T\to\T$, $\transl x := x+\om \, { \rm mod} \, \Z$,
which is an invertible and ergodic transformation with respect to the Haar measure $\mu = d x $ on $\T$.
In fact, we assume more on the frequency, namely that it satisfies a (strong) Diophantine condition:
\begin{equation}\label{DCt}
\norm{k \cdot \om} \ge \frac{t}{\abs{k} \cdot (\log \abs{k})^2 } \quad \text{for all } k \ge 2
\end{equation}
where $\norm{k \cdot \om}$ denotes the distance from $k \cdot \om$ to the nearest integer and $t > 0$ is some small constant. We denote by $\rm{D C}_t$ the set of all such frequencies. Clearly $\mu( \rm{D C}_t ) = o (t)$. We note that all estimates in this paper that depend on the frequency $\om \in \rm{D C}_t$ in fact depend only on the fixed parameter $t$. 

This quantitative description of the irrationality of $\om$ implies a quantitative Birkhoff ergodic theorem for the corresponding translation and for subharmonic sample functions (see Section~\ref{LDT_section}). 

We also consider the case of the multifrequency translation $\transl \, \underline{x} = \transl_{\underline{\om}} \, \underline{x} := \underline{x} + \underline{\om} $
on the torus $\T^d$ of dimension $ d \ge 1$. In this case we only make a standard Diophantine assumption on the multifrequency $ \underline{\om} $, as a stronger condition will have no additional benefit. The same comments and general strategy apply to this multivariable case.

\smallskip

Returning to the Lyapunov exponents of a cocycle $A$, if for some $1 \le i  < m$, $L_i (A) > L_{i+1} (A)$, we say that the Lyapunov spectrum of $A$ has a {\em gap} at dimension $i$. If  the Lyapunov spectrum of $A$ has gaps at dimensions $1 \le \tau_1 < \tau_2 < \ldots < \tau_k < m$, then calling the sequence $\tau = (\tau_1, \tau_2, \ldots, \tau_k)$ a {\em signature}, we say that the Lyapunov spectrum of $A$ has a {\em gap pattern} encoded by the signature $\tau$, or in short, a $\taugp$.

In particular, the case of {\em simple} Lyapunov spectrum (i.e. all distinct Lyapunov exponents) is encoded by the (full) signature $\tau = (1, 2, \ldots, m-1)$.

\smallskip

In a forthcoming paper we will provide sufficient conditions for certain types of quasiperiodic cocycles to have simple spectra, or more generally, any kind of gap pattern. 

\smallskip

This paper is concerned with proving {\em continuity} of the Lyapunov exponents as functions of quasiperiodic, {\em analytic} cocycles.  Moreover, given a cocycle that satisfies a gap pattern, we show that every sum of Lyapunov exponents between two consecutive gaps is a {\em H\"older} continuous function near the given cocycle.

More precisely, let $\cocyclesTd$ be the set of all real analytic cocycles $A \colon \T^d \to \GLmR$ which have a continuous  extension 
to the strip $\strip_r := \{ z \in \C^d \colon \abs{ \Im z } \le r \}$ that is holomorphic in the interior of $\strip_r$. 
Endowed with the the norm 
$$\normr{A} := \sup_{z \in \strip_r} \, \norm{ A (z) }$$
$\cocyclesTd$ becomes a Banach manifold.

Let $A \in \cocyclesTd$ and let  $\tau = (\tau_1, \tau_2, \ldots, \tau_k)$ be a signature. We define the $\taublockL$s of $A$ as:
\begin{align*}
\La_{\pi, 1} (A) & := L_1 (A) + \ldots + L_{\tau_1} (A) \\
\La_{\pi, 2} (A) & := L_{\tau_1 + 1}  (A) + \ldots + L_{\tau_2} (A)
\end{align*}
and in general, with the convention that $\tau_0 = 0$,
$$\La_{\pi, j} (A) := L_{\tau_j + 1}  (A) + \ldots + L_{\tau_{j+1}} (A) \quad \text{for all } 0 \le j \le k-1$$

\medskip

We are now ready to formulate the main results of this paper. 

\begin{theorem} \label{main-thm1}
Let $A \in \cocycles$, let   $\transl = \transl_\om$ be a translation on $\T$ where $\om \in \T$ satisfies the (strong) Diophantine condition \eqref{DCt} and  let $\tau$ be a signature. 

If $A$ has a $\taugp$, then the corresponding Lyapunov spectrum $\tau$-blocks are H\"older continuous functions in a neighborhood of $A$.

In particular, if $A$ has simple Lyapunov spectrum, then all Lyapunov exponents of $A$ are H\"older continuous in a neighborhood of $A$. 
\end{theorem}

\begin{theorem} \label{main-multifreq}
Let $A \in \cocycle{\T^d}{m}{\R}$, let   $\transl = \transl_{\underline{\om}}$ be a translation on $\T^d$, $d \ge 1$, where $\underline{\om} \in \T^d$ satisfies a standard Diophantine condition and  let $\tau$ be a signature. 

If the cocycle $A$ has a $\taugp$, then the corresponding Lyapunov spectrum $\tau$-blocks are log-H\"older continuous functions in a neighborhood of $A$.

In particular, if $A$ has simple Lyapunov spectrum, then all Lyapunov exponents of $A$ are log-H\"older continuous in a neighborhood of $A$. 
\end{theorem}

The H\"older (and log-H\"older respectively) constant and the size of the neighborhood in each of the two  theorems above depend on some measurements of $A$ and on $\om$.

\begin{theorem} \label{main-thm2}
Let $A \in \cocycle{\T^d}{m}{\R}$ and let   $\transl = \transl_{\underline{\om}}$ be a translation on $\T^d$ ($d \ge 1$) by a frequency satisfying a standard Diophantine condition.
Then all Lyapunov exponents are continuous functions on $ \cocycle{\T^d}{m}{\R}$.
\end{theorem}

This theorem holds irrespectively of any gap pattern. 
  
\medskip

Given a natural number $1\leq k\leq m$, let ${\rm L}^1(\T^d,\Gr_k^m)$ denote  the space
of measurable functions $E:\T^d\to \Gr^m_k$, that to each $x\in\T^d$ associate a $k$-dimensional
linear subspace $E(x)\subset \R^m$, i.e. an element of the
Grassmann  manifold $\Gr^m_k$. Equipped with the metric
$$ d(E,E')=\int_{\T^d} d(E(x),E'(x))\, dx \;,$$
the space   ${\rm L}^1(\T^d,\Gr_k^m)$  is a Banach manifold.
Consider a cocycle $(\transl,A)$, where  $A\in \cocyclesTd$
and $\transl:\T^d\to\T^d$ is an ergodic translation. Given a signature
$\tau=(\tau_1,\ldots, \tau_k)$, if $A$ has a $\tau$-gap pattern,
by Oseledets theorem there are $A$-invariant sub-bundles 
$E^{A}_j\in {\rm L}^1(\T^d,\Gr^m_{k_j})$, with $k_j=\tau_j-\tau_{j-1}$ for $1\leq j\leq k$, such that
on $E^{A}_j$ the cocycle $A$ has Lyapunov exponents in the range $[L_{\tau_j}(A),L_{\tau_{j-1}+1}(A)]$.
Furthermore, these sub-bundles are such that $\R^m=\oplus_{j=1}^k E^{A}_j(x)$, for almost every $x\in\T^d$.
Regard this family of sub-bundles $E^A$ 
as an element of the product Banach manifold
$\prod_{j=1}^k {\rm L}^1(\T^d, \Gr^m_{k_j})$, equipped with the component
max distance. We refer to $E^A$  
as the {\em Oseledets $\tau$-decomposition} of  $A$.
Define $\underline{F}^A(x)$ to be the sequence of  subspaces 
$F^A_1(x)\subset \ldots \subset F^A_k(x)$, where for each $j=1,\ldots, k$, $F^A_j(x)$ is the linear span of $E^A_1(x)\cup \ldots \cup E^A_j(x)$.
The sequence $\underline{F}^A(x)$ is a $\tau$-flag, an element in the flag manifold $\mathscr{F}^m_\tau$.
 The function $\underline{F}^A:\T^d\to\mathscr{F}^m_\tau$ is a point in the Banach manifold
 ${\rm L}^1(\T^d,\mathscr{F}^m_\tau)$,  equipped with the metric
 $$   d(\underline{F},\underline{F}')=\int_{\T^d} d(\underline{F}(x),\underline{F}'(x))\, dx  \;. $$
 We refer to this function as the 
{\em Oseledets $\tau$-filtration} of  $A$.

\begin{theorem}\label{oseledets:continuity} Given  $A\in\cocycle{\T^d}{m}{\R}$ with a $\tau$-gap pattern,
there is a neighbourhood $\mathscr{U}$ of $A$ where every
cocycle has the same  $\tau$-gap pattern. Moreover, the functions
$B \mapsto E^B\in \prod_{j=1}^k {\rm L}^1(\T^d, \Gr^m_{k_j})$,
resp. $B\mapsto \underline{F}^B\in {\rm L}^1(\T^d,\mathscr{F}^m_\tau)$,
that associate  the Oseledets $\tau$-decomposition, resp. the Oseledets $\tau$-filtration
to each  cocycle $B\in \mathscr{U}$, are continuous.
\end{theorem}

\medskip

The paper is organized as follows. In Section~\ref{defs-nots_section} we introduce some definitions and notations, and describe some uniform measurements on the size of the cocycle. In Section~\ref{AP_section} we establish a precise higher dimensional version of the Avalanche Principle (AP) of M. Goldstein and W. Schlag, using a geometrical approach. In Section~\ref{LDT_section} we derive a large deviation theorem (LDT) for logarithmic averages of quantities related to singular values of the iterates of the cocycle. In Section~\ref{ind-step-thm_section} we present a general inductive scheme, based on the AP and on the LDT, which will be used repeatedly throughout the paper. Section~\ref{proofs_section} contains the proofs of our main statements, while Section~\ref{consequences_section} presents some consequences and extensions of the main statements. 

\smallskip

\section{Definitions and notations}\label{defs-nots_section}

\subsection*{Singular values}
By definition, a {\em singular value} of a matrix $g\in\GLmR$
is an eigenvalue of one of the conjugate matrices
$\sqrt{g^\ast\, g}$ and $\sqrt{g\,g^\ast}$. We order the 
singular values $s_j (g)$, $1 \le j \le m$ of the matrix $g\in\GLmR$ from the largest to the smallest:
$$s_1(g)\geq s_2(g) \geq \ldots \geq  s_m(g) >0$$ We have
$s_1(g)=\norm{g}$ and $s_m(g)= \norm{g^{-1}}^{-1} = m (g)$ (i.e. the minimum expansion of $g$). 

We call  a {\em singular value formula}, abbreviated s.v.f., any expression obtained by taking products and ratios of singular values of a matrix $g \in \GLmR$.  We will not use all s.v.formulas, but only those obtained as products of singular values with distinct indices, namely expressions of the form 
$$s (g) = s_{j_1} (g) \cdot \ldots \cdot s_{j_l} (g) \quad \text{where} \quad 1 \le j_1 < \ldots < j_l \le m$$
as well as ratios of two \textit{consecutive} singular values, namely expressions of the form $$s (g) = \frac{s_j (g)}{s_{j+1}(g)} \ \text{ or } \ s (g) = \frac{s_{j+1} (g)}{s_{j}(g)} \quad  \text{where} \quad 1\le j < m$$
We denote the set of all these particular s.v.formulas by $\allsvf$. Note that $\allsvf$ is a finite set, whose cardinality depends only on the dimension $m$.

\smallskip

We define some special types of s.v.formulas $s \in \allsvf$.

\smallskip

For $1 \leq j \leq m$, let 
$$p_j (g) := s_1 (g) \cdot \ldots \cdot s_j (g) = \norm{\wedge_j (g)}$$ 
where $\wedge_j (g)$ is the $j$th exterior power of $g$.

Note that every individual singular value can be described in terms of these s.v.formulas:
$$s_j(g)=\frac{p_j (g)}{p_{j-1}(g)}$$
Therefore, any s.v.f. $s\in \allsvf$ can be described in terms of products and ratios of some s.v.f. of the form $p_j$. This will be important later, as  $p_j (g)$ is the {\em norm} of a matrix related to $g$.

\smallskip


\subsection*{Gap patterns}
We call a  {\em signature} any sequence of integers $\tau=(\tau_1,\ldots, \tau_k)$ such that
$1\leq \tau_1 <\tau_2 <\ldots <\tau_k<m$. 

Signatures are used to describe certain flag manifolds and to encode (known) gap patterns in the singular spectrum of a matrix $g \in \GLmR$ or in the Lyapunov spectrum of a cocycle $A \in \cocycles$. 

More precisely, given a matrix $g \in \GLmR$ and a signature $\tau=(\tau_1,\ldots, \tau_k)$, we say that the singular spectrum of $g$ has a {\em $\tau$-gap pattern} if for every $1 \le j \le k$,  we have:
$$s_{\tau_j} (g) > s_{\tau_j + 1} (g)$$
The other pairs of consecutive singular values may (or may not) be equal.    

Similarly, given a cocycle $A \in \cocycles$, we say that the Lyapunov spectrum of $A$ has a {\em $\tau$-gap pattern} if for every $1 \le j \le k$,  we have:
$$L_{\tau_j} (A) > L_{\tau_j + 1} (A) $$
Again, the other pairs of consecutive Lyapunov exponents may (or may not) be equal. However, when we do know (which will be the case in the proof of Theorem~\ref{main-thm2}) that the other pairs of consecutive Lyapunov exponents are equal, we say that the gap pattern of the Lyapunov spectrum of $A$ is {\em precisely} described by the signature $\tau$.

\smallskip

Given a signature $\tau=(\tau_1,\ldots, \tau_k)$,
we define two kinds of {\em $\tau$-singular value formulas}, which will be elements of the set $\allsvf$:  $\tau$-singular value products, abbreviated $\tausvp$, and $\tau$-singular value ratios, abbreviated $\tausvr$.

\smallskip

A $\tausvp$ is either a product
$$p_{\tau_j} (g) = s_1 (g) \cdot \ldots  \cdot s_{\tau_j} (g)$$
or a block product of the form 
$$
\pi_{\tau,j}(g):=  s_{\tau_{j-1}+1}(g) \cdot \ldots \cdot s_{\tau_j}(g)
$$
with $1\leq j\leq k$ and the convention that $\tau_0=0$.

Throughout this paper, $\pi (g)$ will refer to any such $\tausvp$

Note that for every $1 \le j \le k$ we have:
$$\pi_{\tau,j}(g) = \frac{p_{\tau_j} (g)}{p_{\tau_{j-1}} (g)}$$
and
$$p_{\tau_j} (g) = \pi_{\tau,1}(g) \cdot \ldots \cdot \pi_{\tau,j}(g)$$

\smallskip

A $\tausvr$ is either a ratio of the form
$$\rho_{\tau_j}(g)  := \frac{s_{\tau_j}(g)}{s_{\tau_j + 1} (g)} \quad  \ge 1$$
or one of the form
$$\sigma_{\tau_j}(g)  := \frac{s_{\tau_j + 1}(g)}{s_{\tau_j} (g)} \quad \le 1$$
 where $1 \le j \le k$.
 
 Throughout this paper, $\rho (g)$ will refer to any $\tausvr$ of the form $\rho_{\tau_j}(g)$, while $\sigma (g)$ will refer to any $\tausvr$ of the form $\sigma_{\tau_j}(g)$.
 
 \smallskip
 
If the matrix $g \in \GLmR$ has a $\taugp$  then 
 $$\rho_\tau (g) := \min_{1\le j \le k} \ \rho_{\tau_j}(g) > 1
 \quad
 \text{  and  } 
 \quad
  \sigma_\tau (g) := \max_{1\le j \le k} \ \sigma_{\tau_j}(g) < 1$$
 
 \smallskip
 
 A note on these notations: the $\tausvr$ $\rho$ and $\sigma$ are of course inverse to each other. We will use $\sigma$ (which is less than $1$ and  correlated with the rate of {\em contraction} of $g$'s projective action)  in the initial formulation and proof of the Avalanche Principle.  However, throughout the rest of the paper, we will only use $\rho$ (which is greater than $1$
 and correlated with the size of gaps in terms of the Lyapunov spectrum).

\smallskip


\subsection*{Lyapunov spectrum blocks and gaps}
Let $A \in \cocycles$ be a cocycle, let $s \in \allsvf$ be a singular value formula and let $n \in \N$ be an integer. We define the following quantities:
\begin{equation*} 
\Lasn{n} (A) := \int_{\T} \frac{1}{n} \, \log s (\An{n} (x)) \, dx
\end{equation*}
and
\begin{equation} \label{LasA}
\Las (A) := \lim_{n \to \infty} \, \Lasn{n} (A) 
\end{equation}

In particular, if $s = p_j$, where $1 \le j \le m$, we have:
$$\Lajn{n} (A) = \int_{\T} \, \frac{1}{n} \, \log \norm{\wedge_j \An{n} (x) } \, dx = \int_{\T} \, \frac{1}{n} \, \log \norm{(\wedge_j A)^{(n)} (x) } \, dx$$

Due to the sub-multiplicativity  of the norm, the sequence $\Lajn{n} (A)$ is sub-additive, so it converges to
$$\inf_{n \ge 1} \, \Lajn{n} (A) =: \Laj (A)$$

Since any s.v.f. $s \in \allsvf$ is a product or ratio of s.v.formulas of the form $p_j$, $1 \le j \le m$, and since the logarithm turns these expressions into sums and differences, we conclude that the limits in \eqref{LasA} exist, and we get the following formulas.

If for some $1 \le j \le m$, $s = s_j$, i.e. $s$ is the $j$th singular value, then
$$\La_{s_j} (A) = L_j (A) =  \text{the } j\text{th Lyapunov exponent of } A$$ 

Then if  for some $1 \le j \le m$, $s = p_j$, we get
$$\Laj (A) = L_1 (A) + \ldots + L_j (A)$$

\smallskip

Let  $\tau=(\tau_1,\ldots, \tau_k)$ be a signature and let $1 \le j \le k$. 

If $s = \pi_{\tau,j}$, i.e. the $j$th (block) $\tausvp$, then: 
\begin{equation*} 
\La_{\pi_{\tau,j}} (A) = L_{\tau_{j-1}+1} (A) + \ldots +  L_{\tau_{j}} (A)
\end{equation*}
and we call this quantity the $j$th  {\em Lyapunov spectrum $\tau$-block}.

\smallskip

If $s = \rho_{\tau_j}$ i.e. the $j$th $\tausvr$, then: 
\begin{equation*} 
\La_{\rho_{\tau_j}} (A) = L_{\tau_{j}} (A) - L_{\tau_{j}+1} (A)
\end{equation*}
and we call this quantity the $j$th  {\em Lyapunov spectrum $\tau$-gap}.

\medskip


\subsection*{Example}
Here is an example to make these notations easier to understand. 
Let $\tau=(1, 3)$, and  assume that the matrix $g \in \GLmR$ has the $\taugp$. This means that:
$$s_1(g) > s_2(g) \ge s_3 (g) > s_4 (g) \ge \ldots \ge s_m (g)$$
In other words, there are (known) ``gaps" (i.e. different consecutive singular values) precisely at positions $\tau_1 =1$, so $s_1(g) > s_2(g)$ and $\tau_2 =3$, so $s_3(g) > s_4(g)$. There could, of course, be other gaps elsewhere.

The $\tausvp$ are:
\begin{align*}
\ \pi_{\tau,1}  (g) & =  s_1 (g)\\
\ \pi_{\tau,2}  (g) & =  s_2 (g) \cdot s_3 (g)\\
p_3 (g)  & =  [s_1 (g)] \cdot [s_2 (g) \cdot s_3 (g)]
\end{align*}
The block $s_4 (g) \cdot \ldots \cdot s_m (g)$ could also be thought of as a $\tausvp$, but that would be redundant. 

The $\tausvr$ are:
\begin{align*}
\rho_1 (g)   = \frac{s_1 (g)}{s_2 (g)} \quad \text{ and }
\quad 
\rho_3 (g)   = \frac{s_3 (g)}{s_4 (g)} 
\end{align*}
and their inverses $\sigma_1 (g)$ and  $\sigma_3 (g)$.

\smallskip

Now assume that the cocycle $A \in \cocycles$ has the $\taugp$, where $\tau=(1, 3)$. This means that:
$$L_1 (A) > L_2 (A) \ge L_3 (A) > L_4 (A) \ge \ldots  \ge L_m (A)$$

The $\taublockL$s are
$$\La_{\pi_{\tau,1}} (A)  =  L_1 (A) \; \text{ and } \; 
 \La_{\pi_{\tau,2}} (A)   =  L_2 (A) + L_3 (A)\,, $$
while the $\taugapL$s are 
$$\La_{\rho_{\tau_1}} (A)  = L_1 (A) - L_2 (A) \; \text{ and } \; \La_{\rho_{\tau_2}} (A) = L_3 (A) - L_4 (A)\,.$$

\smallskip


\subsection*{A uniform measurement of the cocycle}
Given a cocycle $A \in \cocycles$, we define a scaling constant $C (A)$ which will dominate any multiplicative constant (dependent on the size of $A$) that will appear in various estimates throughout this paper.

We first derive some simple bounds on the size of the cocycle, on its exterior powers and on its iterates.

We have:
$$\norm{A(x) } \ge \norm{A(x)^{-1}}^{-1} \ge \normr{A^{-1}}^{-1}$$
so
\begin{equation}\label{scalingct-eq1}
 \normr{A^{-1}}^{-1} \le \norm{A(x)} \le \normr{A} \quad \text{for all } x \in \T
\end{equation}

For all $1 \le j \le m$, clearly
$$\norm{\wedge_j A (x) } \le \norm{A(x)}^j \le \normr{A}^j$$
and
$$\norm{[ \wedge_j A ]^{-1} (x) } = \norm{\wedge_j A^{-1} (x)} \le \normr{A^{-1}}^j$$

Then applying \eqref{scalingct-eq1} to $\wedge_j A (x)$ we get:
\begin{equation}\label{scalingct-eq2}
 \normr{A^{-1}}^{-j} \le \norm{\wedge_j A(x)} \le \normr{A}^j \quad \text{for all } x \in \T
\end{equation}

Moreover, for all $n \ge 1$ we have
$$\norm{\An{n} (x)} \le  \prod_{i=n-1}^{0} \, \norm{A (\transl^i x)} \le \normr{A}^n$$
and similarly
\begin{equation*} 
\norm{[\An{n} (x)]^{-1}} \le  \normr{A^{-1}}^n
\end{equation*}

Then applying  \eqref{scalingct-eq1} to $\An{n} (x)$ we obtain:
\begin{equation}\label{scalingct-eq3}
 \normr{A^{-1}}^{-n} \le \norm{\An{n} (x)} \le \normr{A}^n \quad \text{for all } x \in \T
\end{equation}

Combining \eqref{scalingct-eq2} and \eqref{scalingct-eq3}, for all $1 \le j \le m$ and $n \ge 1$ we get:
\begin{equation}\label{scalingct-eq4}
 \normr{A^{-1}}^{-n \, j} \le \norm{\wedge_j \An{n} (x)} \le \normr{A}^{n \, j}  \quad \text{for all } x \in \T
\end{equation}

From \eqref{scalingct-eq3}  we have:
$$\abs{ \frac{1}{n} \log \norm{\An{n} (x)} } \le \max \{  \abs{\log \normr{A} }, \abs{ \log \normr{A^{-1}}}   \}$$

In fact, from \eqref{scalingct-eq4}, for all $1 \le j \le m$,
\begin{align*}
\abs{ \frac{1}{n} \log p_j (\An{n} (x)) }   &= \abs{ \frac{1}{n} \log \norm{\wedge_j \An{n} (x)} }\\
&\le j \,  \max \{  \abs{\log \normr{A} }, \abs{ \log \normr{A^{-1}}}   \}
\end{align*}

Since any $s \in \allsvf$ is obtained by taking products and ratios of s.v.f. of the form  $p_j$, for $1 \le j \le m$, we then conclude that for all $x \in \T$ we have:
\begin{equation} \label{scalingct-eq5}
\abs{ \frac{1}{n} \log s (\An{n} (x)) }   
\less  \max \{  \abs{\log \normr{A} }, \abs{ \log \normr{A^{-1}}}   \}
\end{equation}
and in particular,
\begin{equation} \label{scalingct-eq6}
\abs{ \Lasn{n} (A) }   
\less  \max \{  \abs{\log \normr{A} }, \abs{ \log \normr{A^{-1}}}   \}
\end{equation}
The inherent constant in \eqref{scalingct-eq5}, \eqref{scalingct-eq6} depends only on the dimension $m$, and in fact this constant  is not more than $m \, (m+1)$.

\smallskip

To prove the Large Deviation Theorem~\ref{LDT-thm} in Section~\ref{LDT_section}, we will use a {\em quantitative} Birkhoff ergodic theorem for subharmonic sample functions, namely Theorem 3.8. in \cite{GS-Holder}. The scaling factor there depends on the sup norm of the given subharmonic function, and on the width $r$ of the strip on which it is defined. Applied to our case, the scaling factor will be of order 
$$r^{-1}\cdot  \max \{  \abs{\log \normr{A} }, \abs{ \log \normr{A^{-1}}} \}  \;.$$ 

Accounting for other estimates that will appear in this paper (namely in Lemma~\ref{scales-divide-lemma} and in Lemma~\ref{cont-finite-lemma}) we then define:
\begin{equation} \label{scalingct-star}
C (A) := \frac{m (m+1)}{r} \ [  \,  \abs{\log \normr{A} } +  \abs{ \log \normr{A^{-1}}} + \log (1 + \normr{A}) \, ]
\end{equation} 
We also denote (mind the letter case) \,
 $c (A) :=  C(A)^{-2}$.

The width $r$ of the domain and the dimension $m$ are of course fixed. For simplicity of notations, from now on we will write $a \less b$ whenever $a \le \text{constant} \cdot b$, where the multiplicative constant is either universal or depends only on $m$, $r$. 

It is clear that the scaling constant $C (A)$ defined in \eqref{scalingct-star} depends continuously on $A \in \cocycles$, so if $C (A) < C$ and $B \in \cocycles$ such that $\normr{A - B} \ll 1$, then $C (B) < C$. This will ensure that any estimate that depends only on the size $C (A)$ of $A$ will hold uniformly for any cocycle $B$ in a neighborhood of $A$.

In fact, it is easy to verify that if $B \in \cocycles$ such that
$$\normr{A - B} \le \frac{\normr{A^{-1}}^{-1}}{4} \ (< \frac{\normr{A}}{4})$$
then
\begin{align*}
\frac{1}{2} \, \normr{A^{-1}} &\le \normr{B^{-1}} \le 2 \normr{A^{-1}}\\
\frac{3}{4} \, \normr{A} &\le \normr{B} \le \frac{5}{4} \, \normr{A}
\end{align*}
so
$$C (B) < C(A) + 2$$

Hence if $A \in \cocycles$ is {\em fixed}, we can take $C := C (A) + 2$, $c := \frac{1}{C^2}$ and all estimates depending on the size of $A$ will hold with these multiplicative scaling factors $C, \, c$  for all cocycles $B$ in a ball of {\em fixed} radius $ \frac{\normr{A^{-1}}^{-1}}{4}$ around $A$.

\smallskip

\section{A higher dimensional avalanche principle}\label{AP_section}
\subsection*{Shadowing Lemma}
The following is an abstract lemma on chains 
of continuous mappings  $g_i:X \to X$ ($0\leq i\leq n $) 
on a  compact metric space $X$. Each mapping $g_i$ is assumed to be contractive away from a critical set $\Sigma_i$. The lemma
states that any `pseudo orbit' of this chain 
that stays away from the critical sets $\Sigma_i$ is shadowed by
an orbit of the mapping chain. In particular, any closed pseudo-orbit is shadowed
by a periodic orbit of the mapping chain.

\begin{center}
\includegraphics*[scale={1.2}]{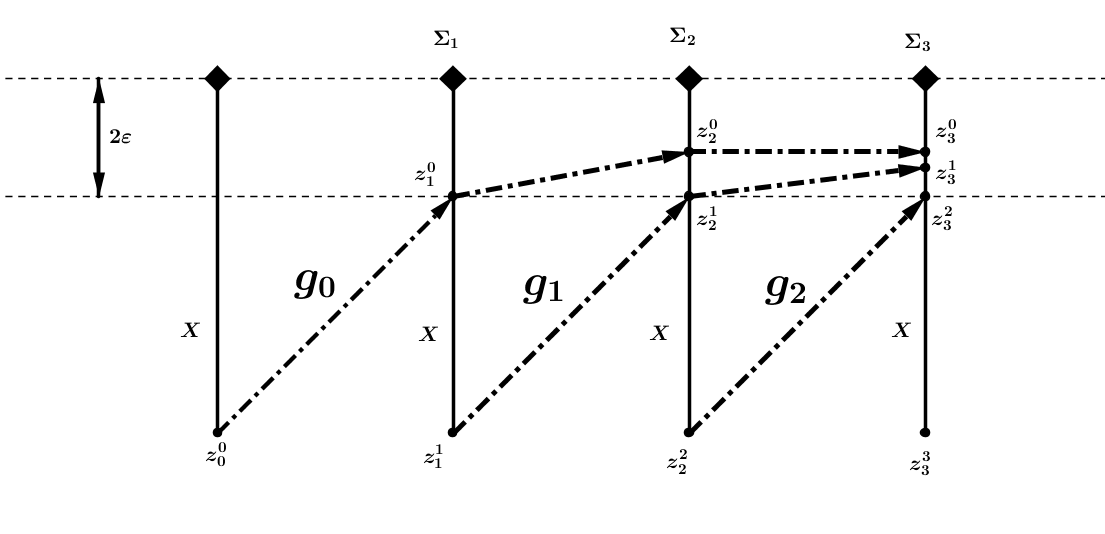} 
\end{center}

We denote by $B_\varepsilon(\Sigma)$ the $\varepsilon$-neighbourhood of a subset $\Sigma\subset X$,
$$ B_\varepsilon(\Sigma)=\{\, x\in X\,:\, d(x, \Sigma)<\varepsilon\,\}\;.$$

\begin{lemma}\label{shadow:lemma}
Let $\varepsilon, \delta>0$ and $0<\kappa<1$ such that
$\delta/(1-\kappa)<\varepsilon<1/2$ and $\delta<\kappa$. 
Given a compact metric space $X$ with diameter $1$, 
a chain of continuous
mappings  $g_0, g_1,\ldots, g_{n-1}:X \to X$,
some closed subsets $\Sigma_0,\ldots, \Sigma_{n-1}\subset X$,
and some pairs of points 
$(x_0, y_0),(x_1, y_1),\ldots, (x_{n-1}, y_{n-1})$ in $X\times X$ such that for every $0\leq i\leq n-1$:
\begin{enumerate}
\item[(a)] $g_i(x_i)=y_i$,
\item[(b)] $d(x_i,\Sigma_i)=1$ and $d(y_i,\Sigma_{i+1})\geq 2\,\varepsilon$,
\item[(c)] $g_i\vert_{X\setminus B_\varepsilon(\Sigma_i)}:
X\setminus B_\varepsilon(\Sigma_i)\to X$ has Lipschitz constant
$\leq \kappa$, 
\item[(d)] $g_i(X\setminus B_\varepsilon(\Sigma_i))\subset B_\delta(y_i)$.
\end{enumerate}
Then setting $g^{(n)} :=g_{n-1}\circ \ldots\circ g_1\circ g_0$, the following hold:
\begin{enumerate}
\item[(1)] 
$d\left(\, y_{n-1}, \, g^{(n)}(x_0)\,\right)\leq 
\frac{\delta}{1-\kappa}$,
\item[(2)] if $x_0=y_{n-1}$ then there is 
a unique point $x^\ast\in X$
such that $g^{(n)}(x^\ast)=x^\ast$ and
$d\left(  x_0, x_\ast \right)\leq 
\frac{\delta}{(1-\kappa)(1-\kappa^n)}$.
\end{enumerate}
\end{lemma}

\begin{proof}
Define $z^i_j:= (g_{j-1}\circ \ldots \circ g_{i+1}\circ g_i)(x_i)$
for $i\leq j\leq n$. Note that
$z^i_i=x_i$ and $z^i_{i+1}=y_i$.
Clearly
$d(z_j^{j-1},z_j^{j})=d(y_{j-1},x_j)\leq  {\rm diam}(X)=1$.
We claim that for $1\leq j\leq n$ and $0\leq i\leq j-2$,\,
$d(z_j^{i},z_j^{i+1})\leq \kappa^{j-i-2}\,\delta$.
For $j=1$ there is nothing to prove.
For $j=2$  we only need to see that item (d) implies
$d(z_2^{0},z_2^{1})=d(g_1(y_0),g_1(x_1))\leq \delta$.
Assuming these inequalities hold for some $j$,
we have $d(z_j^j, \Sigma_j)=d(x_j,\Sigma_j)=1$, while
for $i<j$
\begin{align*}
d(z^i_j, \Sigma_j) &\geq d(y_{j-1},\Sigma_j)- 
\sum_{\ell=i}^{j-2} d( z^\ell_j, z^{\ell+1}_j) \\
&\geq 2\varepsilon -(\delta+\kappa\,\delta+\ldots + \kappa^{j-i-2}\,\delta)\\
&\geq 2\varepsilon -\frac{\delta}{1-\kappa} > \varepsilon\;
\end{align*}
Hence, because of (c), the map $g_j$ acts as a Lipschitz $\kappa$-contraction on the 
points $z_j^i$ with $0\leq i\leq j$.
Since $z_{j+1}^i=g_j(z_j^i)$,
we have for $0\leq i\leq j-2$
\begin{align*}
d(z^i_{j+1}, z^{i+1}_{j+1}) &= 
d(g_j(z_j^i), g_j(z_j^{i+1}))
\leq \kappa\, d( z_j^i, z_j^{i+1})
\leq \kappa^{j-i-1}\,\delta\;
\end{align*}
On the other hand, for $i=j-1$, because of item (d) we have
$$ d(z^{j-1}_{j+1}, z^{j}_{j+1}) 
= d(g_j(z_j^{j-1}), g_j(z_j^{j}))
= d(g_j(y_{j-1}), g_j(x_{j}))
 \leq \delta\;$$
and this establishes the claim.

Thus 
\begin{align*}
 d(y_{n-1}, g^{(n)}(x_0)) &=
 d(z_n^{n-1}, z^0_n)
 \leq \sum_{\ell=1}^{n-1} d(z_n^{\ell}, z^{\ell-1}_n)\\
 &\leq \delta + \kappa\, \delta + \ldots +  \kappa^{n-2}\,\delta
 \leq \frac{\delta}{1-\kappa}\;
\end{align*}
which proves (1).
 
Let us now assume that $y_{n-1}=x_0$.
Then from item (1) we get that 
$d(x_0, g^{(n)}(x_0))<\delta/(1-\kappa)$.
It is now enough to see that
$g^{(n)}(\overline{B_\varepsilon(x_0)})\subset B_\varepsilon(x_0)$
and $g^{(n)}$ acts as a $\kappa^n$-contraction on the closed ball $\overline{B_\varepsilon(x_0)}$. 
We claim that 
$(g_i\circ\ldots\circ g_0)(\overline{B_\varepsilon(x_0)})\cap B_\varepsilon(\Sigma_{i+1})=\emptyset$, for every $i=0,1,\ldots, n-2$.
Let us assume the claim and finish the proof.
It follows from the claim, together with  assumption (c), that 
$g^{(n)}$ acts as a $\kappa^n$-contraction on the ball $\overline{B_\varepsilon(x_0)}$.
Given $x\in  \overline{B_\varepsilon(x_0)}$,
\begin{align*}
 d( g^{(n)}(x), x_0) &\leq
 d(g^{(n)}(x), g^{(n)}(x_0)) + d(g^{(n)}(x_0), x_0)\\
 &\leq \kappa^{n-1}\, d(g(x),g(x_0)) + \delta + \kappa\,\delta  + \ldots + \kappa^{n-2}\,\delta\\
  &\leq \delta + \kappa\,\delta  + \ldots + \kappa^{n-1}\,\delta
  \leq \frac{\delta}{1-\kappa} < \varepsilon\;.
\end{align*}
Thus $g^{(n)}(x)\in B_\varepsilon(x_0)$, which proves that
$g^{(n)}(B_\varepsilon(x_0))\subset B_\varepsilon(x_0)$.

We now prove the claim above  by induction in $i=0,1,\ldots, n-2$.
Consider first the case $i=0$. 
Given $x\in \overline{B_\varepsilon(x_0)}$, \,
$d(x,\Sigma_0)\geq d(x_0,\Sigma_0)-d(x,x_0)\geq 1-\varepsilon>\varepsilon$.
This implies that $d(g_0(x),g_0(x_0))\leq\kappa\,d(x,x_0)$.
Thus
\begin{align*}
 d( g_0(x), \Sigma_1) &\geq
 d(  y_0, \Sigma_1) - d(y_0, g_0(x))\\
 &\geq 2\,\varepsilon -  d(g_0(x_0), g_0(x))\\
 &\geq 2\,\varepsilon -  \kappa\, d(x_0,x)\geq 
 2\,\varepsilon -  \kappa\, \varepsilon >\varepsilon\;
\end{align*}
which proves that $g_0(\overline{B_\varepsilon(x_0)})\cap B_\varepsilon(\Sigma_1)=\emptyset$.

Assume now that for every $\ell\leq i-1$,
$(g_\ell\circ \ldots\circ g_0)(\overline{B_\varepsilon(x_0)})\cap B_\varepsilon(\Sigma_{\ell+1})$ is empty.
Then $g^{(i)}=g_{i-1}\circ \ldots \circ g_0$ acts
as a $\kappa^i$ contraction on $\overline{B_\varepsilon(x_0)}$
and $g^{(i)}(\overline{B_\varepsilon(x_0)})\cap B_\varepsilon(\Sigma_{i})=\emptyset$.
Thus for every $x\in\overline{B_\varepsilon(x_0)}$,
\begin{align*}
 d( g^{(i+1)}(x), \Sigma_{i+1}) &\geq
 d(  y_i, \Sigma_{i+1}) - d(y_i, g^{(i+1)}(x))\\
 &\geq 2\,\varepsilon -  d(z_{i+1}^0, z_{i+1}^i) -
 d( z_{i+1}^0,  g^{(i+1)}(x))\\
  &\geq 2\,\varepsilon -  \sum_{\ell=0}^{i-1} d(z_{i+1}^\ell, z_{i+1}^{\ell+1}) -
 d( g^{(i+1)}(x_0),  g^{(i+1)}(x))\\
 &\geq 2\,\varepsilon -   (\delta+\kappa\, \delta +\ldots+\kappa^{i-1}\, \delta)
  -  \kappa^{i}\, d(g_0(x_0),g_0(x))  \\
 &\geq 2\,\varepsilon -  (\delta+ \kappa\, \delta +  \ldots+\kappa^{i}\, \delta)
 \geq 2\,\varepsilon-\frac{\delta}{1-\kappa}>\varepsilon\;
\end{align*}
which proves that $g^{(i+1)}(\overline{B_\varepsilon(x_0)})\cap B_\varepsilon(\Sigma_{i+1})=\emptyset$.
\end{proof}

\medskip

\subsection*{Flag Manifolds}
Given a signature $\tau=(\tau_1,\ldots, \tau_k)$, any sequence $\underline{F}=(F_1,\ldots, F_{k})$ of vector subspaces 
$F_1\subset F_2\subset \ldots \subset F_{k}\subset \R^m$
such that $\dim F_j=\tau_j$, for each $j=1,\ldots, k$,
will be called a {\em $\tau$-flag}.
Denote by $\FF^m_\tau$ the manifold of all  $\tau$-flags.
If $\tau=(1)$ the flag manifold $\FF^m_\tau$ is the projective
space $\Pp^{m-1}$. Similarly, if $\tau=(k)$,\, $\FF^m_\tau$ coincides with the Grassmannian manifold  $\Gr^m_k$.

The general linear group $\GL(m,\R)$ acts on flags.
Each linear automorphism $g$ induces the diffeomorphism
$\varphi_g  :\FF^m_\tau\to\FF^m_\tau$ defined by $\varphi_g\underline{F}=(g F_1,\ldots, g F_{k})$.

The Grassmannian $\Gr^m_k$ can be identified with 
a submanifold of the projective space
$\Pp(\wedge_k\R^m)$, by identifying each vector subspace
$V\subset \R^m$ of dimension $k$ with the exterior
product $v_1\wedge v_2\wedge \ldots \wedge v_k$, where
$\{v_1,v_2,\ldots, v_k\}$ is any orthonormal basis of $V$.
With the canonical Riemannian metric, projective spaces
have diameter $\pi/2$. We normalize the distance so that
the Grassmannian manifolds $\Gr^m_k$ have diameter one,
and consider $\FF^m_\tau$ as a submanifold of
$\Gr^m_{\tau_1}\times \Gr^m_{\tau_2}\times \ldots
\times \Gr^m_{\tau_k}$ which with the induced (max) product distance,
has also diameter $1$.
With this metric, each flag manifold $\FF_\tau^m$ is a compact metric space with diameter $1$.


\subsection*{Singular Eigenspaces}

Assuming $\norm{g}=s_1(g)>s_2(g)$, denote by $\hatv_\pm (g)\in\Pp^{m-1}$ the most expanding singular directions of $g$.
Letting $v_\pm(g)\in\R^m$ be a 
unit vector in the line $\hatv_\pm (g)$,
we have
\begin{align*}
g\,v_-(g) &=\norm{g}\,v_+(g) \;\\
g^\ast v_+(g) &=\norm{g}\,v_-(g) \;
\end{align*}
Suppose that $g$ has a $\tau$-gap pattern and
let $v_\pm^1(g),\ldots, v_\pm^m(g)$ be orthonormal singular eigenbases
of $g$. These bases are characterized by the relations
$$ g\, v_-^ j(g)= s_j(g)\, v_+^ j(g),\quad 1\leq j \leq m\;$$
We define the {\em most expanding $\tau$-flags} of $g$ as
$$ \hatv_{\tau,\pm}(g):=(\hatV_\pm^{\tau_1}(g),\ldots, \hatV_\pm^{\tau_k}(g))\in \FF^m_\tau $$
where
$\hatV_\pm^{\tau_j}(g)$ is the linear span $\langle v_\pm^ 1(g),\ldots, v_\pm^ {\tau_j}(g)\rangle$
for $1\leq j \leq k$.
The subspaces $\hatV_\pm^{\tau_j}(g)$ do not depend on the choice
of the singular eigen-basis, precisely because $g$ has a $\tau$-gap pattern.
Identifying each vector subspace $\hatV_\pm^{\tau_j}(g)\subset \R^m$ with 
a simple $\tau_j$-vector  $V_\pm^{\tau_j}(g) \in \wedge_{\tau_j}\R^m$,
the following relation holds
$$ (\wedge_{\tau_j} g)\, V_-^{\tau_j}(g)= p_{\tau_j}(g)\, V_+^{\tau_j} (g),\;
\text{ where }\;
p_{\tau_j}(g) = \norm{\wedge_{\tau_j} g} \,.$$ 
We also have 
$$ \varphi_g\, \hatv_{\tau,-}(g)= \hatv_{\tau,+}(g)$$

\bigskip

\subsection*{Angles and Orthogonality}
Let us call the {\em correlation  of} $u$ and $v\in\Pp^{m-1}$
the number
$$\alpha(u,v):= \abs{\langle u,v\rangle}=\cos \angle(u,v)\;$$
Because we have normalized the projective metric,
$\angle(u,v)=\frac{\pi}{2}\,d(u,v)$, and hence
$$\alpha(u,v) = \cos\left(\frac{\pi}{2}\, d(u,v)\right)\;$$
It is also clear that for $u,v\in\Pp^{m-1}$,
$$ \alpha(u,v)=0\; \Leftrightarrow \;
d(u,v)=1 \; \Leftrightarrow \; u\perp v$$
We generalize these  concepts to Grassmannian
manifolds. Consider two vector subspaces $U,V\subset \R^m$ of dimension $k\in\{1,\ldots, m\}$ and fix two orthonormal
bases $\{u_1,\ldots, u_k\}$ of $U$ and $\{v_1,\ldots, v_k\}$ of $V$.
Define the {\em correlation of} $U$ and $V$ to be
\begin{align*}
\alpha(U,V) & :=
  \abs{ \langle \, u_1\wedge  \ldots \wedge u_k, \,
v_1\wedge   \ldots \wedge v_k\, \rangle}\\
&=\abs{\det \left( \begin{array}{cccc}
\langle u_1,v_1\rangle & \langle u_1,v_2\rangle & \ldots & \langle u_1,v_k\rangle\\
\langle u_2,v_1\rangle & \langle u_2,v_2\rangle & \ldots & \langle u_2,v_k\rangle\\
\vdots & \vdots & \ddots & \vdots \\
\langle u_k,v_1\rangle & \langle u_k,v_2\rangle & \ldots & \langle u_k,v_k\rangle 
\end{array}\right)}\\
&= \abs{\det (\pi_{U,V})}
\;, 
\end{align*}
where $\pi_{U,V}:U\to V$ is the restriction to $U$ of the
orthogonal projection to $V$. 
Define, as above, $\angle(U,V)$ to be the angle between the vectors
$u_1\wedge \ldots \wedge u_k$ and
$v_1\wedge \ldots \wedge v_k$ in $\wedge_k\R^m$.
By the  normalization of the Grassmannian  metric, we have 
$\angle(U,V)=\frac{\pi}{2}\,d(U,V)$, and hence
$$\alpha(U,V) = \cos\left(\frac{\pi}{2}\, d(U,V)\right)$$
We say that the  vector subspaces $U,V\subset \R^m$ are
{\em orthogonal}, and write $U\perp V$,
if there is some pair of unit vectors
$(u,v)$ such that $u\in U\cap V^\perp$ and $v\in V\cap U^\perp$.
It is not difficult  to check that for $U,V\in\Gr^m_k$,
$$ \alpha(U,V)=0\; \Leftrightarrow \;
d(U,V)=1 \; \Leftrightarrow \; U\perp V$$
Finally, let us consider two flags $\underline{F},\underline{G}\in 
 \FF^m_\tau$, and define
\begin{align*}
d(\underline{F},\underline{G})&:=
\max_{1\leq j\leq k} d(F_j,G_j)\\ 
\alpha(\underline{F},\underline{G})&:=
\min_{1\leq j\leq k} \alpha(F_j,G_j)\\
\angle(\underline{F},\underline{G})&:=
\max_{1\leq j\leq k} \angle(F_j,G_j)\\ 
\underline{F}\perp \underline{G} &\; :\Leftrightarrow\;
F_j\perp G_j,\, \text{ for some }\, j=1,\ldots, k.
\end{align*}
 
 With these definitions,
 $$\alpha(\underline{F},\underline{G}) = 
 \cos\left( \angle(\underline{F},\underline{G})\right) = \cos\left(\frac{\pi}{2}\, d(\underline{F},\underline{G})\right)$$
 and
$$ \alpha(\underline{F},\underline{G})=0\; \Leftrightarrow \;
d(\underline{F},\underline{G})=1 \; \Leftrightarrow \; \underline{F}\perp \underline{G}$$

\bigskip

For each flag  $\underline{F}\in\FF^{m}_\tau$ we define the {\em orthogonal
flag hyperplane}:
$$\Sigma(\underline{F})=\{\, \underline{G}\in\FF^{m}_\tau\,:\, 
 \alpha(\underline{G},\underline{F})=0\,\}
 = \{\, \underline{G}\in\FF^{m}_\tau\,:\, 
  \underline{G}\perp \underline{F} \,\}$$

\begin{lemma} \label{alpha:dist}
For any flags $\underline{F},\underline{G}\in\FF^m_\tau$,
\begin{enumerate}
\item[(a)] $\displaystyle  \alpha(\underline{F},\underline{G})=\sin\left( \frac{\pi}{2}\,d(\underline{G},\Sigma(\underline{F}))\right)$,
\item[(b)] $\displaystyle d(\underline{G},\Sigma(\underline{F}))\leq \alpha(\underline{F},\underline{G})\leq \frac{\pi}{2}\,d(\underline{G},\Sigma(\underline{F}))\;. $
\end{enumerate}
\end{lemma}

\begin{proof}
Item (b) follows from (a).
Let us prove item (a). If $\FF^m_{(1)}=\Pp^ {m-1}$
then $\frac{\pi}{2}\,d(u,\Sigma(v))$ is complementary
to $\angle(u,v)$, and hence
 $$ \alpha(u,v)=\cos \angle(u,v) = \sin\left( \frac{\pi}{2}\,d(u,\Sigma(v))\right)$$
In case $\FF^m_{(k)}=\Gr^ {m}_k$, the angle
$\frac{\pi}{2}\,d(U,\Sigma(V))$ is complementary
 to $\angle(U,V)$, and again
$$ \alpha(U,V)=\sin\left( \frac{\pi}{2}\,d(U,\Sigma(V))\right)$$
In general, since\,
$$\Sigma(\underline {F})=\bigcup_{1\leq j\leq k}
\{\, \underline {G}\in\FF^m_\tau\,:\,\alpha(F_j,G_j)=0\, \}$$
we get that
$$d(\underline {G},\Sigma(\underline {F}))= \min_{1\leq j\leq k} d(G_j, \Sigma(F_j))$$
Hence
\begin{align*}
\alpha(\underline {F},\underline {G}) &= 
\min_{1\leq j\leq k} \alpha(F_j, G_j)\\
&= 
\min_{1\leq j\leq k} \sin\left(\frac{\pi}{2}\,
d\left(G_j,\Sigma(F_j) \right) \right)\\
&= 
\sin\left(\frac{\pi}{2}\,\min_{1\leq j\leq k} 
d\left(G_j,\Sigma(F_j) \right) \right)\\
&= 
\sin\left(\frac{\pi}{2}\, 
d\left(\underline {G},\Sigma(\underline {F}) \right) \right)\;
\end{align*}
\end{proof}

\smallskip

\subsection*{Projective Contraction}
Consider the flag manifold  $\FF^m_\tau$ with signature  $\tau=(\tau_1,\ldots, \tau_k)$.
Given $g\in\GL(m,\R)$   and $\varepsilon>0$  define
the {\em singular critical sets}
$$\Sigma_{\tau}^{\pm }(g):=\{\, \underline{F}\in\FF^m_\tau\,:\,
  \alpha_{\tau}(\underline{F},\hatv_{\tau,-}(g))=0\,\}\;
$$
as well as its $\varepsilon$-neighborhood ($\varepsilon>0$)
$$\Sigma_{\tau}^{-,\varepsilon}(g):=\{\, \underline{F}\in\FF^m_\tau\,:\,
  \alpha_{\tau}(\underline{F},\hatv_{\tau,-}(g))<\varepsilon\,\}\;
$$

\begin{proposition} \label{proj:contr} Given  $\varepsilon>0$ and $\kappa>0$, for any $g\in\GL(m,\R)$ 
with $\sigma_\tau(g)\leq \kappa$,
the restriction mapping $\varphi_g:\FF^m_\tau\setminus \Sigma_{\tau}^{-,\varepsilon}(g)
\to \FF^m_\tau$ has Lipschitz constant
$\kappa\,( 1 + \varepsilon )/\varepsilon^2$ and
$\varphi_g\left( \FF^m_\tau\setminus \Sigma_{\tau}^{-,\varepsilon}(g) \right)\subset B_{\kappa/\varepsilon}(\hatv_{\tau,+}(g))$.
\end{proposition}

Given a  unit vector $v\in\R^m$, $\norm{v}=1$,
denote by $\pi_v,\pi_v^\perp:\R^m\to\R^m$
the orthogonal projections $\pi_v(x)=\langle v,x\rangle\, v$
resp. $\pi_v^ \perp(x)=x -\langle v,x\rangle\, v$.

\begin{lemma}\label{diff:proj}
Given two non-colinear unit vectors $u,v\in\R^m$, i.e., $\norm{u}=\norm{v}=1$,
denote by $P$ the plane spanned by $u$ and $v$. Then
\begin{enumerate}
\item[(a)] is $\pi_v-\pi_u$ a self-adjoint endomorphism,
\item[(b)]  ${\rm Ker}(\pi_v-\pi_u)= P^ \perp$,
\item[(c)] the restriction $\pi_v-\pi_u:P\to P$ is anti-conformal
with similarity factor $\abs{\sin\angle (u,v)}$,
\item[(d)] $\norm{\pi^\perp_v-\pi^\perp_u}=\norm{\pi_v-\pi_u}\leq \norm{v-u}$.
\end{enumerate}
\end{lemma}

\begin{proof}
Items (a) and (b) are obvious. For (c) we may clearly assume that
$d=2$, and consider $u=(u_1,u_2)$, $v=(v_1,v_2)$, with $u_1^ 2+u_2^ 2=v_1^ 2+v_2^ 2=1$.
The projections $\pi_u$ and $\pi_v$ are represented by the matrices
$$ U= \left(\begin{array}{cc}
u_1^2 & u_1 u_2 \\
u_1 u_2 & u_2^ 2
\end{array} \right)\quad \text{ and }\quad
V = \left(\begin{array}{cc}
v_1^2 & v_1 v_2 \\
v_1 v_2 & v_2^ 2
\end{array} \right)$$
w.r.t. the canonical basis. Hence  $\pi_v-\pi_u$ is given by
$$ V-U= \left(\begin{array}{cc}
v_1^2-u_1^2 & v_1 v_2 - u_1 u_2 \\
v_1 v_2 - u_1 u_2 & v_2^ 2 - u_2^ 2
\end{array} \right) = \left(\begin{array}{cc}
\beta & \alpha \\
\alpha & -\beta 
\end{array} \right)$$
where $\alpha=v_1 v_2 - u_1 u_2$ and $\beta= v_1^2-u_1^2 = -(v_2^ 2 - u_2^ 2)$.
This proves that the restriction of $\pi_v-\pi_u$ to the plane $P$ is
anti-conformal. The similarity factor of this map is
$$ \norm{\pi_v-\pi_u} = \norm{\pi_v(u)-u}=\norm{\pi_v^\perp(u) }
=\abs{\sin \angle(u,v)}$$
Finally, since $u - \langle v,u\rangle\,v\perp v$,
\begin{align*}
  \norm{\pi^\perp_v-\pi^\perp_u}^2 &=\norm{\pi_v-\pi_u}^2 = \norm{\pi_v^\perp(u) }^2\\
&= \norm{ u - \langle v,u\rangle\,v}^2\\
&= \norm{u-v}^2- \norm{ v - \langle v,u\rangle\,v}^2\\
&\leq  \norm{u-v}^2\;. 
\end{align*}
\end{proof}

\begin{lemma}
Given $g\in\GL(m,\R)$, the derivative of the transformation
$\varphi_g:\Pp^{m-1}\to \Pp^{m-1}$,
$\varphi_g(x)=\frac{g\,x}{\norm{g\,x}}$, is given by 
$$ (D\varphi_g)_x v = 
\frac{g\,v-\langle \varphi_g(x), \, g\,v\rangle\, \varphi_g(x)}{\norm{g\,x}} =
 \frac{1}{\norm{g\,x}}\, \pi_{\varphi_g(x)}^ \perp (g\,v) $$
\end{lemma}

\begin{proof}
Simple calculation.
\end{proof}

\begin{proof}[Proof of Proposition~ \ref{proj:contr}]
Let us first consider the case $k=1$, where $\FF^m_\tau=\Pp^ {m-1}$.
Let $\hatv_\pm \in\Pp^{m-1}$ be the  top singular directions  of $g$, 
and let $v_\pm$ be their representative unit vectors, which satisfy
$g\, v_- =\norm{g}\, v_+$. Given $\hat{x}\in\Pp^{m-1}\setminus \Sigma^{-,\varepsilon}$, represented by a
unit vector $x\in\R^m$, we can write $x= a\,v_- + w$ with $\abs{a}\geq\varepsilon$,
 $w\perp v_-$ and $\norm{w}^2=1-\abs{a}^2$. Then $g\,x= \norm{g}\,a\, v_+  + g\,w$ with $g\,w\perp v_+$,
and in particular $\norm{g\,x}\geq \varepsilon\,\norm{g}$.
The projective (non-normalized)  distance $d(\varphi_g(x), v_+)$ is the angle
between $g\,x$ and $v_+$.
Hence
\begin{align*}
\norm{ \varphi_g(x) - v_+ }&\leq d(\varphi_g(x), v_+) 
= \arctan\left(\frac{\norm{g\,w}}{\norm{g}\,\abs{a}}\right)\\
& \leq \frac{\norm{g\,w}}{\norm{g}\,\abs{a}}
\leq \frac{s_2(g)\,\norm{w}}{\norm{g}\,\abs{a}} \\
&\leq \frac{s_2(g)\,\sqrt{1-\varepsilon^ 2}}{s_1(g)\,\varepsilon} \leq   \frac{\kappa \,\sqrt{1-\varepsilon^ 2}}{ \varepsilon} 
\leq \frac{\kappa}{ \varepsilon} 
\;. 
\end{align*}
In particular this proves that
$\varphi_g\left( \Pp^{m-1}\setminus \Sigma^{-,\varepsilon}\right)\subset B_{\kappa/\varepsilon}(\hatv_+)$.
By definition of $v_+$ one has $\norm{\pi^\perp_{v_+}\circ g}\leq s_2(g)\leq \kappa\,\norm{g}$.
Thus because
$$ (D\varphi_g)_x\, v=\frac{1}{\norm{g x}}\,
\pi_{v_+}^\perp(g\,v) + \frac{1}{\norm{g x}}\left( \pi_{\varphi_g(x)}^\perp - \pi_{v_+}^\perp\right)(g\,v)\;,
$$
by lemma~\ref{diff:proj} we have
\begin{align*}
\norm{ (D\varphi_g)_x } &\leq
\frac{\kappa\,\norm{g}}{\norm{g x}}
 + \frac{\norm{ \varphi_g(x) - v_+ }\,\norm{g}}{\norm{g x}} \\
 &\leq
\frac{\kappa}{\varepsilon}
 + \frac{\kappa\,\sqrt{1-\varepsilon^2} }{\varepsilon^2} 
 = \frac{\kappa\,(\varepsilon+\sqrt{1-\varepsilon^2}) }{\varepsilon^2} \leq\frac{\kappa\,(1+ \varepsilon)}{\varepsilon^2}  \;.
\end{align*}
Since $\Pp^{m-1}\setminus \Sigma^{-,\varepsilon}$ is a Riemannian convex disk,
by the mean value theorem it follows that
$\varphi_g: \Pp^{m-1}\setminus \Sigma^{-,\varepsilon}(g)\to \Pp^{m-1}$ is a Lipschitz contraction,
provided $\kappa\ll \varepsilon^ 2$.

The case $\FF^m_{(k)}=\Gr^m_k$, with $1<k\leq m$, reduces to $k=1$ by taking exterior powers. Notice that
$$ \sigma_{k}(g)=\frac{ s_{k+1}(g) }{ s_{k}(g) }
= \frac{ s_{2}(\wedge_{k}g) }{ s_{1}(\wedge_{k}g) } 
= \sigma_1(\wedge_{k} g) $$
Also, if $\{u_1,\ldots, u_d\}$ is an orthonormal basis
then the linear span $U=\langle u_1,\ldots, u_{k}\rangle$
and the subspace $V_-^k(g) \in \Gr^m_k$ satisfy
$$ \alpha_{k} (U,V_-^k(g)) = \alpha_1(u_1\wedge \ldots \wedge
u_{k}, v_-(\wedge_{k} g) ) $$

Having in mind that $d\left(\underline{F}, \underline{G}\right)=\max_{1\leq i \leq k} d(F_j,G_j)$
and
$\alpha_\tau\left(\underline{F}, \underline{G}\right)=\min_{1\leq i \leq k} \alpha(F_j,G_j)$,
the  case of a general signature reduces to the previous one.
\end{proof}

\smallskip

\subsection*{Exotic Operation} 
The following algebraic operation on the set $[0,1]$ will play an important role in next section.
 $$ a\oplus b := a+b -a\,b$$

Clearly
$([0,1],\oplus)$ is a commutative semigroup isomorphic to $([0,1],\cdot)$. In fact, the map $\Phi:([0,1],\oplus)\to ([0,1],\cdot)$, $\Phi(x):= 1-x$,
is a semigroup isomorphism. We summarize some properties of this
operation.

\begin{proposition} \label{oplus:prop}
For any $a,b,c\in [0,1]$,
\begin{enumerate}
\item $0\oplus a = a$,
\item $1\oplus a = 1$,
\item $a\oplus b = (1-b)\,a+b = (1-a)\,b+a $,
\item $a\oplus b <1$\; $\Leftrightarrow$\; $a<1$ and $b<1$,
\item $a\leq b$ \; $\Rightarrow$\; $a\oplus c\leq b\oplus c$,
\item $b>0$\; $\Rightarrow$\; 
$({a} {b}^{-1}\oplus c)\,b\leq a\oplus c$,
\item $a\,c + b\,\sqrt{1-a^2}\,\,\sqrt{1-c^ 2}    \leq \sqrt{a^2 \oplus b^2}$.
\end{enumerate}
\end{proposition}

\begin{proof}
For the last item consider the function 
$f:[0,1]\to [0,1]$
defined by $f(c):= a\,c + b\,\sqrt{1-a^2}\,\,\sqrt{1-c^ 2}$.
A simple computation shows that
$$ f'(c)=a-\frac{b\,c\,\sqrt{1-a^2}}{\sqrt{1-c^2}}$$
The derivative $f'$  has a zero at $c= a/\sqrt{a\oplus b}$,
and one can check that this zero is a global maximum of $f$.
Since
$f( a/\sqrt{a\oplus b} )=\sqrt{a^2\oplus b^2}$,
item (7) follows.
\end{proof}

\smallskip

\subsection*{Expansivity Factors}
Assuming that $g$ and $g'$ have $\tau$-gap patterns, define 
\begin{equation}\label{alpha:def}
\alpha_\tau(g,g'):=\alpha_\tau(\hatv_{+}(g),\hatv_{-}(g'))\;, 
\end{equation}
where $\hatv_{\pm}(g)=\hatv_{\tau, \pm}(g)$ are the most expanding $\tau$-flags of $g$, and 
\begin{equation*} 
\beta_\tau(g,g'):=\sqrt{ \sigma_\tau(g)^ 2 \oplus \alpha_\tau(g,g')^ 2 \oplus \sigma_\tau(g')^ 2}\;. 
\end{equation*}
These numbers will be, respectively, the lower and the upper multiplicative factors
for the composition $g'\,g$ of two matrices $g$ and $g'$ (see Proposition~\ref{prod:2:lemma}).
If $\tau=(k)$, which corresponds to $\FF^m_\tau=\Gr^m_k$,
we write $\alpha_k(g,g')$ instead of $\alpha_\tau(g,g')$,
and if $\tau=(1)$, which corresponds to $\FF^m_\tau=\Pp^{m-1}$, 
we simply write $\alpha(g,g')$ instead of $\alpha_1(g,g')$ or
$\alpha_\tau(g,g')$. Analogous conventions are adopted for $\beta_\tau$.

\begin{lemma} \label{alpha:beta:bound}
Given matrices $g,g'\in\GL(m,\R)$ with $\tau$-gap patterns,
$$ 1\leq \frac{\beta_\tau(g,g')}{\alpha_\tau(g,g')}\leq \sqrt{ 1+
\frac{\sigma(g)^2\oplus \sigma(g')^2}{ \alpha_\tau(g,g')^2} }$$
\end{lemma}

\begin{proof} Simple computation.
\end{proof}

\begin{proposition}\label{prod:2:lemma}
Given $g, g' \in\GL(m,\R)$ with a $(1)$-gap pattern,
$$ \alpha(g,g') \leq \frac{\norm{g'\,g}}{\norm{g'}\,\norm{g}}
\leq \beta(g,g')   $$
\end{proposition}

\begin{proof}
Let $v\in\R^m$ be a unit vector,
$v=\alpha_\ast\, v_-(g)+w$ with  $w\perp v_-(g)$
and $\norm{w} =\sqrt{1-\alpha_\ast^2}$.
Then $g\,v= \alpha_\ast\,\norm{g}\, v_+(g)+g\,w$ with $g\,w\perp v_+(g)$,
 $\norm{g\,w}=\sigma_\ast\,\norm{g}\, \sqrt{1-\alpha_\ast^2}$ and
 $\sigma_\ast\leq \sigma(g)$. Hence
 $\norm{g\,v}=\norm{g} \sqrt{\alpha_\ast^2\oplus\sigma_\ast^2}$ and
 $$ \frac{g\,v}{\norm{g\,v}}=
 \frac{\alpha_\ast}{\sqrt{\alpha_\ast^2\oplus\sigma_\ast^2}}\,v_+(g) + \frac{1 }{\sqrt{\alpha_\ast^2\oplus\sigma_\ast^2}}\,\frac{g\,w}{\norm{g}}$$
 Write $\alpha=\alpha(g,g')=\abs{ \langle v_+(g), v_-(g') \rangle}$.
Let us define and estimate
 \begin{align*}
 \beta_\ast  := \abs{\langle \frac{g\,v}{\norm{g\,v}}, v_-(g') \rangle }
 &=\abs{  \frac{\alpha\,\alpha_\ast}{\sqrt{\alpha_\ast^2\oplus\sigma_\ast^2}} 
  + \frac{1 }{\sqrt{\alpha_\ast^2\oplus\sigma_\ast^2}}\,\frac{1}{\norm{g}}\,\langle g\,w,  v_-(g')  \rangle
   }\\
 &\leq \frac{\alpha\,\alpha_\ast}{\sqrt{\alpha_\ast^2\oplus\sigma_\ast^2}} 
 + \frac{\sigma_\ast \, \sqrt{1-\alpha_\ast^2}}{\sqrt{\alpha_\ast^2\oplus\sigma_\ast^2}}\,
 \abs{ \langle \frac{g\,w}{\norm{g\,w}},  v_-(g')  \rangle }\\
 &\leq \frac{\alpha\,\alpha_\ast}{\sqrt{\alpha_\ast^2\oplus\sigma_\ast^2}} 
 + \frac{\sigma_\ast \, \sqrt{1-\alpha_\ast^2}\, \sqrt{1-\alpha^2} }{\sqrt{\alpha_\ast^2\oplus\sigma_\ast^2}} \leq \frac{\sqrt{\alpha^2\oplus\sigma_\ast^2}}{\sqrt{\alpha_\ast^2\oplus\sigma_\ast^2}}
  \;.
 \end{align*}
On the last inequality use Lemma~\ref{oplus:prop} (7).
For the preceding inequality notice that
 $v_-(g')$ can be written as
 $v_-(g')=	\alpha\, v_+(g) + u$ with $u\perp v_+(g)$ and $\norm{u}=\sqrt{1-\alpha^2}$.
 Thus
  \begin{align*}
  \abs{\langle \frac{g\,w}{\norm{g\,w}},  v_-(g')  \rangle } &=
  \abs{\langle \frac{g\,w}{\norm{g\,w}},  u  \rangle }\leq \norm{u} \leq \sqrt{1-\alpha^2}\;.
  \end{align*}
 
 Repeating the initial argument, now with $g'$ and $g v/\norm{g v}$, we get
\begin{align*}
\norm{ g'\,g\,v} &\leq \norm{g'}\,\sqrt{\beta_\ast^2\oplus (\sigma')^2}\,\norm{g\,v}\\
&\leq \norm{g'}\,\norm{g}\,\sqrt{\beta_\ast^2\oplus (\sigma')^2}\,\sqrt{\alpha_\ast^2\oplus (\sigma_\ast)^2}\\
&\leq \norm{g'}\,\norm{g}\,\sqrt{\sigma_\ast^2\oplus \alpha^2\oplus (\sigma')^2 }
= \beta\,\norm{g'}\,\norm{g} \;,
\end{align*}
where $\sigma'=\sigma(g')$. On the last inequality use Lemma~\ref{oplus:prop} (6).
On the other hand, taking $\alpha_\ast=1$ we have $v=v_-(g)$,
$g\,v= \norm{g}\, v_+(g)$, and  
$$ \norm{g'\,g}\geq \norm{g'\,g\,v}\geq \alpha\,\norm{g'}\norm{g\,v} = \alpha\,\norm{g'}\norm{g}\;.$$
\end{proof}

\begin{corollary} \label{multip:norms:lemma}
Given $g_0, g_1,\ldots, g_{n-1}\in\GL(m,\R)$ with a $(1)$-gap pattern,
writing $g^{(i)}= g_{i-1}\ldots g_0$, we have
$$ \prod_{i=1}^{n-1}
\alpha(g^{(i)},g_i)
\leq \frac{\norm{ g_{n-1}\ldots g_1 g_0}}{
\norm{g_{n-1}}\ldots \norm{g_{1}}
\norm{g_{0}}}\leq \prod_{i=1}^{n-1}
\beta(g^{(i)},g_i) $$
\end{corollary}

\begin{proof}
Follows from Proposition~\ref{prod:2:lemma} by induction.
\end{proof}

\begin{corollary} \label{svp:lemma}
Given
$g_0, g_1,\ldots, g_{n-1}\in\GL(m,\R)$ having a $\tau$-gap pattern, and the $\tau$-s.v.p. function $\pi = \pi_{\tau,j}$,
writing $g^{(i)}= g_{i-1}\ldots g_0$, 
$$ \prod_{i=1}^{n-1}
\frac{\alpha_{\tau_{j}}(g^{(i)},g_i)}{\beta_{\tau_{j-1}}(g^{(i)},g_i)}
\leq \frac{\pi( g_{n-1}\ldots g_1 g_0) }{
\pi(g_{n-1})\ldots  \pi(g_{1})
\pi(g_{0})} \leq \prod_{i=1}^{n-1}
\frac{\beta_{\tau_{j}}(g^{(i)},g_i)}{\alpha_{\tau_{j-1}}(g^{(i)},g_i)} $$
\end{corollary}

\begin{proof} We have  $\pi(g)=\norm{\wedge_ {\tau_{j}} g}/\norm{\wedge_ {\tau_{j-1}} g}$.
By corollary~\ref{multip:norms:lemma} we get
$$ \prod_{i=1}^{n-1}
\alpha_{\tau_{j}}(g^{(i)},g_i)
\leq \frac{\norm{ \wedge_{\tau_{j}}( g_{n-1}\ldots g_0) }}{
\norm{\wedge_{\tau_{j}} g_{n-1}}\ldots \norm{\wedge_{\tau_{j}}  g_{0}}}
\leq \prod_{i=1}^{n-1}
\beta_{\tau_{j}}(g^{(i)},g_i)  $$
and
$$ \prod_{i=1}^{n-1}
\beta_{\tau_{j-1}}(g^{(i)},g_i)^{-1}
\leq \frac{\norm{\wedge_{\tau_{j-1}} g_{n-1}}\ldots \norm{\wedge_{\tau_{j-1}}  g_{0}}}
{\norm{ \wedge_{\tau_{j-1}}( g_{n-1}\ldots g_0) }}
\leq \prod_{i=1}^{n-1}
\alpha_{\tau_{j-1}}(g^{(i)},g_i)^{-1} $$
To finish we just need to multiply these inequality chains term wise. 
\end{proof}

\smallskip

\subsection*{The Avalanche Principle}
We now state a generalization of the {\em Avalanche Principle} of M. Goldstein and W. Schlag (see \cite{GS-Holder}).
This theorem says that given a chain of matrices $g_0,g_1,\ldots, g_{n-1}$ in $\GL(m,\R)$, 
with some quantified $\tau$-gap pattern, and with minimum `angles' between the most expanding
singular flags for pairs of consecutive matrices, then 
their product keeps the same pattern.
Given $0\leq j <i <n$, let us write $g^{(i)}:= g_{i-1}\ldots g_{1} g_0$.

\begin{theorem} \label{Theorem:AP} There is a smooth function $\kappa_0(\varepsilon)=o(\varepsilon^2)$ as $\varepsilon\to 0$, 
such that given $\varepsilon>0$, $\kappa\leq \kappa_0(\varepsilon)$,
and  \,  $g_0, g_1,\ldots, g_{n-1}\in\GL(m,\R)$, \,
 if 
\begin{enumerate}
\item[(a)] $\sigma_\tau(g_i)\leq \kappa$,\, 
for $0\leq i\leq n-1$, and 
\item[(b)] $\alpha_\tau(g_{i-1}, g_{i})\geq \varepsilon$,\; 
for $1\leq i\leq n-1$,
\end{enumerate}
then 
\begin{enumerate}
\item[(1)]  $d(\hatv_+(g^{(n)}), \hatv_+(g_{n-1})) 
  \lesssim   \kappa\,\varepsilon^{-1}$   
\item[(2)]  $d(\hatv_-(g^{(n)}), \hatv_-(g_{0})) 
 \lesssim    \kappa\,\varepsilon^{-1}$  
 \item[(3)]
$  \sigma_\tau(g^{(n)})  \leq   \left(\frac{\kappa\,(1+\varepsilon)}{\varepsilon^{2}}\right)^n$

\item[(4)]  for any $\tau$-s.v.p. function $\pi$,
\begin{align}
\label{snd:s}
&\abs{\log \pi(g^{(n)}) +
\sum_{j=1}^{n-2} \log \pi(g_i) -
\sum_{\ell=1}^{n-1}
\log  \pi(g_\ell\,g_{\ell-1}) } \lesssim  n\,\frac{\kappa}{\varepsilon^2}
\;
\end{align}
\end{enumerate}
\end{theorem}

\begin{proof} 
Given $\varepsilon>0$ and $\kappa\ll \varepsilon^2$, set $\varepsilon'=\varepsilon/\pi$.
Applying Proposition~ \ref{proj:contr}
with $\kappa$ and $\varepsilon'$, for
 any $g\in\GL(m,\R)$ with $\sigma_\tau(g)\leq \kappa$, 
 the map $\varphi_g$
 has Lispchitz constant 
 $\kappa'= \kappa\,\varepsilon'^{-2}\,(\varepsilon'+\sqrt{1-\varepsilon'^2})
 \lesssim  \kappa\, \varepsilon^{-2}$
 over $\FF^m_\tau\setminus \Sigma_\tau^{-,\varepsilon'}(g)$.
Also $\varphi_g\left(\FF^m_\tau\setminus \Sigma_\tau^{-,\varepsilon'}(g)\right)\subset B_\delta(\hatv_{\tau,+}(g))$ 
where $\delta = \kappa/\varepsilon'$.

We are going to apply  Lemma~\ref{shadow:lemma} twice, 
with the given constants $\varepsilon'$, $\kappa'$ and $\delta = \kappa/\varepsilon'$.
Consider the compact metric space $X=\FF^m_\tau$,
the sequence of mappings
$\varphi_{g_0}$, $\varphi_{g_1}$, $\ldots$, $\varphi_{g_{n-1}}$, $\varphi_{g_{n-1}^\ast}$,
$\ldots$,  $\varphi_{g_1^\ast}$,  $\varphi_{g_0^\ast}$, the sequence of argument-value pairs
\begin{tabbing}
\quad \= $\hatv_-(g_0)\stackrel{g_0}{\mapsto} \hatv_+(g_0)$\;, \qquad\qquad
\= $\ldots$\;\;,  \qquad  \=$\hatv_-(g_{n-1})\stackrel{g_{n-1} }{\mapsto} \hatv_+(g_{n-1})$,\\
\>$\hatv_+(g_{n-1})\stackrel{g_{n-1}^\ast}{\mapsto} \hatv_-(g_{n-1})$\; , \>$\ldots$ \;\;,  
  \>$\hatv_+(g_0)\stackrel{g_0^ \ast}{\mapsto}  \hatv_-(g_0)$ \;,
\end{tabbing} 
and the sequence of singular critical sets
$$\Sigma^-_\tau(g_0), \Sigma^-_\tau(g_1),\ldots, \Sigma^-_\tau(g_{n-1}), \Sigma^+_\tau(g_{n-1}), \ldots , \Sigma^+_\tau(g_1), \Sigma^+_\tau(g_0)\; .$$
Let us check the assumptions of Lemma~\ref{shadow:lemma}.
Assumption (a) is obvious.
The first part of (b) follows from Lemma~\ref{alpha:dist} (a) and the fact that $\alpha_\tau(\underline{F},\underline{F})=1$
for any flag $\underline{F}\in\FF^m_\tau$.
For the second part of (b) we use the hypothesis 
$\alpha_\tau(v_+(g_{i-1}), v_-(g_i))= \alpha_\tau(g_{i-1}, g_{i})\geq \varepsilon$,
together with Lemma~\ref{alpha:dist} (b) to conclude that 
$$  d(v_+(g_{i-1}),\Sigma_\tau^-(g_{i}))\geq \frac{2}{\pi}\, \alpha_\tau(v_+(g_{i-1}), v_-(g_i))\geq 2\,\varepsilon'$$
Because the second half of our chain of mappings consists of the
adjoints of the mappings in the first half, 
the singular vector geometry is
`replicated', and whence 
$$  d(v_-(g_{i}),\Sigma_\tau^+(g_{i-1})) =
d(v_+(g_{i}^\ast),\Sigma_\tau^-(g_{i-1}^\ast)) \geq  2\,\varepsilon'$$
By Lemma~\ref{alpha:dist} (b),
$\FF^m_\tau\setminus B_{\varepsilon'}(\Sigma_\tau^-(g_i))\subset \FF^m_\tau\setminus \Sigma_\tau^{-,\varepsilon'}(g_i)$,
and the  assumption (c) of Lemma~ \ref{shadow:lemma}
is satisfied. Finally, by Proposition~ \ref{proj:contr}
we have $\varphi_{g_i}\left( \FF^m_\tau\setminus \Sigma_\tau^{-,\varepsilon'}(g_i)\right)\subset B_\delta(\hatv_{\tau,+}(g_i))$,
which shows that assumption (d) of Lemma~ \ref{shadow:lemma}
is satisfied.

The most expanding  flag $\hatv_-(g^{(n)})\in\FF_\tau^m$ of $g^{(n)}$ is a fixed point of the mapping
$\varphi_{g_0^\ast\ldots g_{n-1}^ \ast g_{n-1}\ldots g_0}$. Hence, by (2) of Lemma~\ref{shadow:lemma},
$$ d(\hatv_-(g_0), \hatv_-(g^{(n)}))\leq \frac{\delta}{(1-\kappa')(1-{\kappa'}^{2n})}\lesssim \kappa\,\varepsilon^{-1}\;.$$
 Applying the same argument to the chain of maps
$\varphi_{g_0^\ast}$, $\varphi_{g_1^\ast}$,  $\ldots$,  $\varphi_{g_{n-1}^\ast}$,
 $\varphi_{g_{n-1}}$, $\ldots$, $\varphi_{g_1}$,  $\varphi_{g_0}$,
 we conclude that 
 $d(\hatv_+(g_{n-1}), \hatv_+(g^{(n)})) \lesssim \kappa\,\varepsilon^{-1}$.
 
Let us prove item (3). Writing $s_i=s_i(g^{(n)})$,
we have 
$$s_1(\wedge_{\tau_j} g^{(n)})=s_1\ldots s_{\tau_{j}-1} s_{\tau_j}\quad\text{ and } \quad
s_2(\wedge_{\tau_j} g^{(n)})=s_1\ldots s_{\tau_{j}-1}s_{\tau_{j}+1}$$
The projective map $\varphi_{\wedge_{\tau_j}g^{(n)}}$ sends 
$\hatv_-(g^{(n)})$ to $\hatv_+(g^{(n)})$.
The largest singular value of the Jacobian of
$\varphi_{\wedge_{\tau_j}g^{(n)}}$ at $\hatv_-(g^{(n)})$ is
$s_2(\wedge_{\tau_j} g^{(n)})/s_1(\wedge_{\tau_j} g^{(n)})$.
Hence, denoting by ${\rm Lip}(\varphi_{g^{(n)}})$ the Lipschitz constant
of the map $\varphi_{\wedge_{\tau_j}g^{(n)}}$, in some appropriate domain,
 \begin{align*}
 \frac{s_{\tau_j+1}(g^{(n)})}{s_{\tau_j}(g^{(n)})} &=
  \frac{s_{2}(\wedge_{\tau_j} g^{(n)})}{s_{1}(\wedge_{\tau_j} g^{(n)})} 
  = \norm{ (D\varphi_{\wedge_{\tau_j}g^{(n)}})_{\hatv_-(g^{(n)})} }\\
  &\leq {\rm Lip}(\varphi_{g^{(n)}})\leq\left( \frac{\kappa\,(1+\varepsilon)}{\varepsilon^2}
  \right)^n \;.
 \end{align*}
   
 Before proving inequality~(\ref{snd:s}), we claim that
\begin{equation}\label{claim}
 \abs{\alpha_{\tau_j}(g^{(i)},g_i)-\alpha_{\tau_j}(g_{i-1},g_i)} \lesssim \kappa\,\varepsilon^{-1} \;.
\end{equation} 
Consider the following subspaces in $\Gr^m_{\tau_j}$:
 $\hatv=\hatv_+(g^{(i)})$, $\hatv_+=\hatv_+(g_{i-1})$ and $\hatv_-=\hatv_+(g_i)$,
 so that $\alpha_{\tau_j}(\hatv,\hatv_-)= \alpha_{\tau_j}(g^{(i)},g_i)$
 and  $\alpha_{\tau_j}(\hatv_+,\hatv_-)= \alpha_{\tau_j}(g_{i-1},g_i)$.
 Applying Lemma~\ref{shadow:lemma} to the chain of maps
 $\varphi_{g_{i-1}^\ast},\ldots, \varphi_{g_{0}^\ast}$, 
 $\varphi_{g_{0}},\ldots, \varphi_{g_{i-1}}$ we get 
 $$ d(\hatv,\hatv_+)\leq \frac{\delta}{(1-\kappa')(1-(\kappa')^{2i})} \lesssim \kappa\,\varepsilon^{-1} $$
Hence, setting $\Sigma_i^-:=\Sigma_{\tau_j}^-(g_i)$,
\begin{align*}
\abs{ \alpha_{\tau_j}(\hatv,\hatv_-) - \alpha_{\tau_j}(\hatv_+,\hatv_-) } &=
\abs{  \sin\left( \frac{\pi}{2}\, d(\hatv,\Sigma_i^-)\right) - \sin\left( \frac{\pi}{2}\, d(\hatv_+,\Sigma_i^-)\right) }\\
&\leq
\frac{\pi}{2}\, \abs{ d(\hatv,\Sigma_i^-) - d(\hatv_+,\Sigma_i^-) }\\
& \leq \frac{\pi}{2}\, d(\hatv,\hatv_+)
 \lesssim \kappa\,\varepsilon^{-1} \;.
\end{align*}
Notice in general the $\alpha$ and $\beta$ factors of corollary~\ref{svp:lemma},
are not comparable, but we claim their ratios are close to one,
at least for iso-dimensional couples.
More precisely, to prove inequality~(\ref{snd:s})
we use the following: for any $j$ and $i$, the logarithm of any ratio between the four factors
$\alpha_{\tau_j}( g^{(i)}, g_i)$, $\beta_{\tau_j}( g^{(i)}, g_i)$,
$\alpha_{\tau_j}( g_{i-1}, g_i)$ and  $\beta_{\tau_j}( g_{i-1}, g_i)$
 is of order $\kappa\,\varepsilon^{-2}$.
In fact, from the previous claim~(\ref{claim}) we have
\begin{align*}
\abs{\log\frac{ \alpha_{\tau_j}(g^{(i)},g_i) }{ \alpha_{\tau_j}(g_{i-1},g_i) } }
&\leq \frac{1}{\varepsilon}\,\abs{ \alpha_{\tau_j}(g^{(i)},g_i) -\alpha_{\tau_j}(g_{i-1},g_i) }
\lesssim \kappa\,\varepsilon^{-2} \;
\end{align*}
By Lemma~\ref{alpha:beta:bound}, 
since $\sigma_{\tau_j}(g_i)\leq \kappa$, 
the other ratios are of the same magnitude as this one, and the fact follows.
Thus we can assume that each of these four ratios is inside the interval
$[e^{-C\,\kappa\,\varepsilon^ {-2}},e^{C\,\kappa\,\varepsilon^ {-2}}]$,
for some universal constant $C>0$.
Now, given any $\tau$-s.v.p.   $\pi$,
applying corollary~\ref{svp:lemma}
we obtain lower and upper bounds for the following ratio
\begin{align*}
&\frac{ \pi(g_{n-1}\ldots g_1 g_0) \,\pi(g_{n-2}) \pi(g_{n-2})\ldots \pi(g_1)} 
{\pi(g_{n-1}  g_{n-2}) \ldots \pi(g_1\, g_0) } =\\
&\frac{ \pi(g_{n-1}\ldots g_1 g_0) }{ \pi(g_{n-1}) \ldots \pi(g_1) \pi(g_0)}
\, 
\frac{ \pi(g_{n-1}) \pi(g_{n-2}) }{ \pi(g_{n-1}  g_{n-2})}\,
\ldots \,
\frac{ \pi(g_{1}) \pi(g_{0}) }{ \pi(g_1\, g_0)}\;.
\end{align*}
These bounds are products of $4\,(n-1)$ $\alpha$ and $\beta$ factors.
Half of them appear as numerators, and the other half as denominators.
We can pair them in couples of iso-dimensional  factors.
Hence the bounds become products of $2\,(n-1)$  ratios close $1$. More precisely, we get
$$ e^{-2\,C\,n\,\kappa\,\varepsilon^ {-2}} \leq 
\frac{ \pi(g_{n-1}\ldots g_1 g_0) \,\pi(g_{n-2}) \pi(g_{n-2})\ldots \pi(g_1)} 
{\pi(g_{n-1}  g_{n-2}) \ldots \pi(g_1\, g_0) }\leq e^{2\,C\, n\,\kappa\,\varepsilon^ {-2}} \;.
$$
Taking logarithms~(\ref{snd:s}) follows.
\end{proof}

\begin{remark}\rm{
By Proposition~\ref{prod:2:lemma} and Lemma~\ref{alpha:beta:bound}, if assumption (a) of  Theorem~\ref{Theorem:AP} holds, then assumption (b) can be replaced by:
$$\frac{\norm{\wedge_{\tau_j} (g_i\,g_{i-1})}}{\norm{\wedge_{\tau_j} g_i}\,\norm{\wedge_{\tau_j} g_{i-1}}}\geq \varepsilon
\quad \text{ for }\; 1\leq i\leq n-1,
\; 1\leq j\leq k$$}
\end{remark}

\begin{remark}\rm{
For $\SL(2,\R)$ matrices we have $\sigma(g)=1/\norm{g}^2$, and letting $\kappa=1/\mu^2$ and $\varepsilon=1/\sqrt{\mu}$,
assumptions (a) and (b) in  Theorem~\ref{Theorem:AP} become:
\begin{enumerate}
\item[(a)] $\norm{g_i}\geq \mu$, \; for every $0\leq i\leq n-1$,
\item[(b)] $ \norm{g_i\,g_{i-1}}\geq \frac{1}{\sqrt{\mu}}\, \norm{g_{i}}\,\norm{g_{i-1}}$, \;
for every $1\leq i\leq n-1$.
\end{enumerate}
These are exactly the hypotheses of the Avalanche Principle in \cite{GS-Holder}, while \eqref{snd:s} is  the same as the conclusion in \cite{GS-Holder}. }
\end{remark}

We rephrase the Avalanche Principle in a form that will be used throughout the rest of the paper. 

\begin{proposition} \label{AP-practical}

Let $g_0, g_1, \ldots, g_{n-1} \in \GL (m, \R)$ be such that:
\begin{align*}
\rm{(gaps)} \  & \rho (g_i) >  \frac{1}{\ka} &  \text{for all } \tausvr \,  \rho
,  & \ \  0 \le i \le n-1  
\\
\rm{(angles)} \  & \frac{\pi (g_i \cdot g_{i-1})}{\pi ({g_i}) \cdot \pi (g_{i-1})}  >  \ep  & \
 \text{for all } \tausvp  \, \pi,  & \  \ 1 \le i \le n-1  
\end{align*}
where the positive constants $\ep, \ka$ satisfy
\begin{equation*}
\ka \ll \ep^2
\end{equation*}
Then we have:
\begin{equation*} 
\max\left\{ \, d(\hatv_+(g^{(n)}), \hatv_+(g_{n-1})),\,
d(\hatv_-(g^{(n)}), \hatv_-(g_{0})) \, \right\}
 \lesssim    \kappa\,\ep^{-2} 
\end{equation*}
\begin{equation*} 
 \rho (g^{(n)}) >  \left(\frac{\ep^2}{\ka}\right)^n \quad  \text{for all } \tausvr  \, \rho
\end{equation*}
\begin{equation*} 
\sabs{ \log \pi (g^{(n)}) + \sum_{i=1}^{n-2} \log \pi (g_i) -  \sum_{i=1}^{n-1} \log \pi (g_i \cdot g_{i-1}) } \less n \cdot \frac{\ka}{\ep^2} 
\end{equation*}
for all $\tausvp  \, \pi$
 \end{proposition}

\begin{proof}
Assume first that $\rho(g)>\ka^{-1}$, $\rho(g')>\ka^{-1}$
and $\frac{\norm{g' g}}{\norm{g} \norm{g'}}>\ep$. We claim that
$\alpha(g,g')> \ep\,\sqrt{1-2 \ep^2}$.
Combining Proposition~\ref{prod:2:lemma}
with Lemma~\ref{alpha:beta:bound}
\begin{align*}
\ep^2  & < \frac{\norm{g' g}^ 2}{\norm{g}^2 \norm{g'}^2}
\leq \beta(g,g')^ 2  \leq \left( 1+ \frac{2\,\ka^2}{\alpha(g,g')^2} \right)\,\alpha(g,g')^2\\
&\leq \alpha(g,g')^2 + 2\,\ka^2\;,
\end{align*}
which implies that $\alpha(g,g')^2\geq \ep^2-2\ep^4$,
thus proving the claim.
Consider now $\pi=p_{\tau_j}$ so that $\pi(g)=\norm{\wedge_{\tau_j} g}$.
From the previous abstract argument applied to the pairs of matrices
$\wedge_{\tau_j} g_{i-1}$ and $\wedge_{\tau_j} g_{i}$, we get that
$$\alpha_{\tau_j} (g_{i-1}, g_i)\geq \ep\,\sqrt{1-2\ep^2}
\quad \text{ for }\; 1\leq i\leq n-1,
\; 1\leq j\leq k$$
All statements follow if we apply Theorem~\ref{Theorem:AP} with parameters 
$\ka$ and $\varepsilon :=\ep\,\sqrt{1-2\ep^2}$.
\end{proof}


\section{The large deviation theorem}\label{LDT_section}
The following result shows that for any s.v.f. $s \in \allsvf$, the quantities $\frac{1}{n} \log s (\An{n} (x)) $ do not deviate much from their space averages $\Lasn{n} (A) $, for most space variables $x \in \T$ and for large enough $n$. This is similar to other large deviation theorems (LDT) obtained previously for quasiperiodic $\SL (2, \R)$ Schr\"{o}dinger cocycles (see for instance \cite{B-book}, \cite{GS-Holder} or \cite{Schlag} for a higher dimensional version). However, this result applies to logarithmic 
averages of {\em any} singular values (in fact of any singular value formulas) and not just to logarithmic averages of the norm (i.e. the largest singular value) of the iterates of the cocycle.

\begin{theorem}\label{LDT-thm}
Let $A \in \cocycles$ and let $C > 0$ such that $C(A) < C$.

Then for any s.v. formula  $s \in \allsvf$ and for any $\delta > 0$ we have:
\begin{equation}\label{LDT1}
\abs{ \{ x \in \T \colon \abs{\frac{1}{n} \log s (\An{n} (x))  - \Lasn{n} (A) } > \delta \} } < e^{- c \delta^3 n}
\end{equation}
where $c = C^{-2}$ and provided $n \ge n_{00} (C, \om, \delta)$.

In particular, if we set $\delta := n^{-1/6}$, we have:
\begin{equation}\label{LDT2}
\abs{ \{ x \in \T \colon \abs{\frac{1}{n} \log s (\An{n} (x))  - \Lasn{n} (A) } > n^{-1/6} \} } < e^{- c n^{1/2}}
\end{equation}
provided $n \ge n_{00} (C, \om)$.
\end{theorem}

\begin{proof}
The proof combines two ingredients: an abstract, fundamental result on averages of (Diophantine) translations of a subharmonic function (i.e.  a quantitative Birkhoff ergodic theorem) and the observation that the quantities  $\frac{1}{n} \log s (\An{n} (x))$ are almost invariant under the base dynamics.

\smallskip

The sharpest result on averages of translations of a subharmonic function for (strongly) Diophantine translations is due to M. Goldstein and W. Schlag (see Theorem 3.8. in \cite{GS-Holder}). It says that given a $1$ - periodic subharmonic function $ u (z)$ on the strip $\strip_r$ such that $\sup_{z \in \strip_r} \, \abs{u (z)} \le M$, if $\transl = \transl_\om$ is the translation on $\T$ by a frequency $\om \in \rm{D C}_t$, then: 
\begin{equation}\label{quantB}   
\abs{ \{ x \in \T \colon \abs{ \frac{1}{R} \, \sum_{j=0}^{R-1} u (\transl^j x) - \int_{\T} u (x) dx } > \delta \}  }  < e^{-c \, \delta^2 R}
\end{equation}
for all $\delta > 0$ and for all $R \ge R_0 (M, r, \om, \delta)$, where the constant $ c \asymp \frac{r}{M}$.

\smallskip

For any s.v.f. $s \in \allsvf$, let
\begin{equation*} 
\usn{n} (x) := \frac{1}{n} \log s (\An{n} (x)) 
\end{equation*}

Since $A (\cdot)$ has a holomorphic extension to the strip $\strip_r$, so do any of its exterior powers $\wedge_j A (\cdot)$, $1 \le j \le m$, and so do any of the iterates $\wedge_j \An{n} (\cdot)$. Therefore, the maps
$$\upjn{n} (x) = \frac{1}{n} \log p_j (\An{n} (x)) = \frac{1}{n} \log \norm{ \wedge_j \An{n} (x) } $$
have subharmonic extensions to the strip $\strip_r$.

Moreover, we have
$$ \sup_{z \in \strip_r} \, \abs{ \upjn{n} (z)} \less j \, \max \{  \abs{\log \normr{A} }, \abs{ \log \normr{A^{-1}}} \}  \le C(A) < C$$

We conclude that for all $1 \le j \le m$, and for all $n \ge 1$, the functions $\upjn{n} (\cdot)$ have subharmonic extensions to the {\em same} strip $\strip_r$, and on $\strip_r$ they are bounded by the {\em same} constant $C$. Hence \eqref{quantB} applies to them uniformly, meaning with the same constants $c$ and $R_0$. In fact, it would apply with the same constants if instead of $A$ we considered any other cocycle $B \in \cocycles$ such that $C (B) < C$. 

Moreover, since any s.v.f.  $s \in \allsvf$ can be expressed in terms of products and ratios of some $p_j$, $1 \le j \le m$, it follows that $\usn{n} (\cdot)$ is a {\em linear combination} of subharmonic functions $\upjn{n} (\cdot)$. Then by the linearity of the integral and the triangle inequality, \eqref{quantB} also applies to $\usn{n} (\cdot)$ and we get:
\begin{equation}\label{quantB-usn}   
\abs{ \{ x \in \T \colon \abs{ \frac{1}{R} \, \sum_{j=0}^{R-1} \usn{n} (\transl^j x) - \int_{\T} \usn{n} (x) dx } > \delta \}  }  < e^{-\bar{c}  \, \delta^2 R}
\end{equation}
for all $s \in \allsvf$, for all $\delta > 0$, for all $R \ge R_0 (C, \om, \delta)$ and for $$\bar{c} \asymp \frac{r}{ \max \{  \abs{\log \normr{A} }, \abs{ \log \normr{A^{-1}}} \} }$$
where the inherent constant depends only on $\# \allsvf$, hence only on the dimension $m$.

\smallskip

We now prove the almost invariance property under the base dynamics. We have:
\begin{align*}
\frac{1}{n} \log \norm{\An{n} (\transl \, x)} -  \frac{1}{n} \log \norm{\An{n} (x)}     = \\
 \frac{1}{n} \log \ \frac{\norm{ A (\transl^n \,  x)  \cdot   [ A (\transl^{n-1} \, x)  \cdot \ldots \cdot A (\transl \, x)  \cdot A (x) ] \cdot A(x)^{-1} }}{ \norm{ A (\transl^{n-1} \, x)  \cdot \ldots \cdot A (\transl \, x)  \cdot A (x) } }  \le \\
 \frac{1}{n} \log \, [ \norm{ A (\transl^n \, x) } \cdot \norm{ A (x)^{-1}} ]  \le \frac{\log [ \normr{A} \cdot \normr{A^{-1}} ]}{n} <  \frac{C}{n}
\end{align*}
Doing another similar calculation, we conclude that
$$\abs{ \frac{1}{n} \log \norm{\An{n} (\transl \, x)} -  \frac{1}{n} \log \norm{\An{n} (x)}  } < \frac{C}{n}$$
If instead of $A$ we consider exterior powers $\wedge_j A$, from the inequality above we obtain:
$$\abs{ \upjn{n} (\transl \, x) - \upjn{n} (x) } < \frac{C}{n}$$
and again, since for fixed $n$ the functions $\usn{n} (\cdot)$ are linear combinations of $\upjn{n} (\cdot)$, we conclude that
\begin{equation} \label{almost-inv-usn}
\abs{ \usn{n} (\transl \, x) - \usn{n} (x) } \le \frac{C}{n} \quad \text{for all } x \in \T, \ \text{ for all } s \in \allsvf
\end{equation} 
Applying \eqref{almost-inv-usn} $R$ times and averaging in $R$ we have:
\begin{equation}\label{almost-inv-usn-R}
 \abs{ \frac{1}{R} \, \sum_{j=0}^{R-1} \usn{n} (\transl^j x) -  \usn{n} (x) } \le \frac{C \, R}{n} \quad \text{for all } x \in \T
\end{equation}

Choose $n$ large enough depending on $R$ (hence on $C$, $\om$, $\delta$) and let 
$R := \intpart{ (\delta \, n) / (2 C )}$ so \eqref{almost-inv-usn-R} becomes:
\begin{equation}\label{almost-inv-delta}
 \abs{ \frac{1}{R} \, \sum_{j=0}^{R-1} \usn{n} (\transl^j x) -  \usn{n} (x) } \le \frac{\delta}{2} \quad \text{for all } x \in \T
\end{equation}

Combining \eqref{quantB-usn} and \eqref{almost-inv-delta} we conclude that if $n \ge \nzero (C, \om, \delta)$ then:
\begin{equation*}
\abs{  \{ x \in \T \colon  \abs{  \usn{n} (x) - \int_\T \usn{n} (x) d x     }   > \delta   \}   } < e^{ - \bar{c} \, (\delta / 2)^2 \, \delta n / (2 C)   } 
= e^{ - c \delta^3 \, n  }
\end{equation*}
where $c \asymp C(A)^{-2} > C^{-2}$.

\end{proof}

\begin{remark}\rm{The constants $c$ and $\nzero$ do not depend on the cocycle $A$ per se, but on the bound $C$ of its size  $C (A)$ and on the frequency $\om$. Therefore, \eqref{LDT1}, \eqref{LDT2} hold with the same constants $c$ and $\nzero$ for any other cocycle $B$ such that $C (B) < C$.
}
\end{remark}


\section{The inductive step theorem and other technicalities}\label{ind-step-thm_section}

The following lemma is an easy application of the LDT Theorem~\ref{LDT-thm}, and it shows how space average gaps and angles are related to their pointwise counterparts.   

\begin{lemma}\label{average-pointwise_gaps-angles}
Let $A \in \cocycles$,  $C>0$, $\delta >0$ and $\nzero (C, \om, \delta)$ such that $C(A) < C$, and the LDT \eqref{LDT1} applies for $n \ge \nzero$. Let $\tau$ be a signature, and let $\ga >0$.

(i) If $\Larn{n} (A) > \ga$, where $\rho$ is a $\tausvr$, then 
\begin{equation}\label{lemma1i}
\rho (\An{n} (x))  > e^{(\ga - \delta) n} \quad \text{for all } x \notin \B_n 
\end{equation}
where $\abs{\B_n} < e^{-c \delta^3 n}$

(ii) If for some set $\Bbar_n$ we have
$\rho (\An{n} (x))  > e^{\gabar n}$ for all $x \notin \Bbar_n$
then 
 \begin{equation}\label{lemma1ii}
 \Larn{n} (A) > \gabar - C \abs{\Bbar_n}
 \end{equation}
 
(iii) If $\abs{\Lapn{n} (A) - \Lapn{2n} (A)} < \eta$, where $\pi$ is a $\tausvp$, then
\begin{equation}\label{lemma1iii}
\frac{\pi (\An{2n} (x))}{\pi (\An{n} (x)) \cdot \pi (\An{n} (\transl^n x))} > \frac{1}{e^{(2\eta+4\delta)n}}
\end{equation}
for all $x \notin \Bt_n$, where $\abs{\Bt_n} \less e^{-c \delta^3 n}$
\end{lemma}

\begin{proof}
(i) Using the LDT~\eqref{LDT1} for $s = \rho$, a $\tausvr$, we get:
$$\abs{ \{ x \in \T \colon \abs{\frac{1}{n} \log \rho (\An{n} (x))  - \Larn{n} (A) } > \delta \} } < e^{- c \delta^3 n}$$
Denoting by $\B_n$ the deviation set above, we conclude that for $x \notin \B_n$ we have:
$$\frac{1}{n} \log \rho (\An{n} (x)) >  \Larn{n} (A) - \delta > \ga - \delta$$
and \eqref{lemma1i} follows.

\smallskip

(ii) Since $\rho$ is a $\tausvr$, $\rho (g) \ge 1$ for any matrix $g \in \GLmR$. Therefore,
\begin{align*}
 \Larn{n} (A) &= \int_{\T} \, \frac{1}{n} \log \rho (\An{n} (x)) \, dx \ge  \int_{\T \setminus \Bbar_n} \, \frac{1}{n} \log \rho (\An{n} (x)) \, dx \\
 &\ge \gabar (1 - \abs{\Bbar_n}) > \gabar - C \abs{\Bbar_n}
\end{align*}
where the last inequality is due to the fact that if $x \notin \Bbar_n$, then
$$\gabar \le \frac{1}{n} \log \rho (\An{n} (x))  \le C(A) < C$$

\smallskip

(iii) Using the LDT~\eqref{LDT1} for $s = \pi$, a $\tausvr$, we get:
$$\abs{ \{ x \in \T \colon \abs{\frac{1}{n} \log \pi (\An{n} (x))  - \Lapn{n} (A) } > \delta \} } < e^{- c \delta^3 n}$$
Denoting by $\B_n$ the deviation set above, we conclude that:
\begin{align}
 \frac{1}{n} \log \pi (\An{n} (x))  <  \Lapn{n} (A) + \delta \ & \text{ if } x \notin \B_n \label{lemma1iii-eq1}\\
 \frac{1}{n} \log \pi (\An{n} (\transl^n x))  <  \Lapn{n} (A) + \delta \ & \text{ if } x \notin \transl^{-n} \B_n
 \nonumber 
\end{align}
Using the LDT~\eqref{LDT1} at scale $2n$ as well, we get:
\begin{align}
 \frac{1}{2n} \log \pi (\An{2n} (x))  >  \Lapn{2n} (A) - \delta \ & \text{ if } x \notin \B_{2n} \label{lemma1iii-eq3}
 \end{align}
 Denoting $\Bt_n := \B_n \cup \transl^{-n} \B_n \cup \B_{2n}$, clearly $\abs{\Bt_n} \less e^{- c \delta^3 n}$, and if $x \notin \Bt_n$, then using \eqref{lemma1iii-eq1} - \eqref{lemma1iii-eq3} we obtain:
\begin{align*}
 \frac{1}{n} \log \,& \frac{\pi (\An{2n} (x))}{\pi (\An{n} (x)) \cdot \pi (\An{n} (\transl^n x))}   \\
&> 2 [\Lapn{2n} (A) - \delta] - [\Lapn{n} (A) + \delta] -  [\Lapn{n} (A) + \delta]  \\
&= 2 [\Lapn{2n} (A)  - \Lapn{n} (A)] - 4 \delta > - 2 \eta - 4 \delta
\end{align*}
so \eqref{lemma1iii} follows.
\end{proof}

\bigskip


All estimates in this paper on quantities like say, $\abs{\Lasn{n_1} (A) - \Lasn{n_0} (A)}$ where $n_1 \gg n_0$, will be of order $\frac{n_0}{n_1}$ at best. The following lemma and its corollary show that we may assume from now on that whenever we have two such scales, $n_1$ is a {\em multiple} of $n_0$, otherwise we incur another term of order  $\frac{n_0}{n_1}$.

\begin{lemma}\label{scales-divide-lemma}
Given $A \in \cocycles$, $C>0$ such that $C(A) < C$ and some integers satisfying
$$n_1 = n \cdot n_0 + q$$
then for any s.v.f.  $s \in \allsvf$ and for all $x \in \T$ we have:
\begin{equation}\label{scales-divide-x-eq}
\abs{  \frac{1}{n_1} \log s (\An{n_1} (x)) -  \frac{1}{n \cdot n_0} \log s (\An{n \cdot n_0} (x)) } \less C \frac{q}{n_1}
\end{equation}

In particular,
\begin{equation*} 
\abs{\Lasn{n_1} (A) - \Lasn{n \cdot n_0} (A)} \less C \frac{q}{n_1}
\end{equation*}
\end{lemma}

\begin{proof}
As before, when proving the LDT~\ref{LDT-thm}, since any s.v.f. $s \in \allsvf$ is obtained by taking products and ratios of s.v.f. of the form $p_j$, it is enough to establish \eqref{scales-divide-x-eq} for all $s = p_j$, where $1 \le j \le m$:
\begin{equation}\label{scales-divide-pj}
\abs{  \frac{1}{n_1} \log p_j (\An{n_1} (x)) -  \frac{1}{n \cdot n_0} \log p_j (\An{n \cdot n_0} (x)) } < C \frac{q}{n_1}
\end{equation}
But $$p_j (\An{n} (x)) = \norm{\wedge_j \An{n} (x)} = \norm{(\wedge_j \, A )^{(n)} (x) }$$
and due to \eqref{scalingct-eq2},  it is clearly enough to prove that
\begin{equation*} 
\abs{  \frac{1}{n_1} \log \norm{ (\An{n_1} (x))} -  \frac{1}{n \cdot n_0} \log \norm{ (\An{n \cdot n_0} (x))} } < C \frac{q}{n_1}
\end{equation*}
where $A$ is a cocycle of {\em any} dimension (which will imply estimate \eqref{scales-divide-pj} for exterior powers of $A$). 

This will be accomplished through some straightforward calculations.

We can write 
$$\An{n_1} (x) = M (x) \cdot \An{n \cdot n_0} (x)$$
where
$$M (x) = \prod_{i=n \cdot n_0+q-1}^{n \cdot n_0} \, A(\transl^i x) = \An{q} (\transl^{n \cdot n_0} x)$$

From  \eqref{scalingct-eq3}, we have:
$$\norm{M (x)} \le \normr{A}^q \quad \text{and}  \quad \norm{M (x)^{-1}} \le \normr{A^{-1}}^q$$

Then
\begin{align*}
 \frac{1}{n_1} \log \norm{ \An{n_1} (x)} &-  \frac{1}{n \cdot n_0} \log \norm{ \An{n \cdot n_0} (x)} 
=  \frac{1}{n \cdot n_0} \log \frac{ \norm{\An{n_1} (x) }^{\frac{n \cdot n_0}{n_1}}}{\norm{\An{n \cdot n_0} (x)}} \\
&\le \frac{1}{n \cdot n_0} \log  \frac{ \norm{M (x) }^{\frac{n \cdot n_0}{n_1}} \cdot  \norm{\An{n \cdot n_0} (x) }^{\frac{n \cdot n_0}{n_1}} }{\norm{\An{n \cdot n_0} (x)}} \\
&=  \frac{1}{n \cdot n_0} \log \frac{ \norm{M (x)}^{\frac{n \cdot n_0}{n_1}}}{ \norm{\An{n \cdot n_0} (x)}^{\frac{q}{n_1}} } \\
&\le \frac{1}{n \cdot n_0} \log  \frac{ \normr{A}^{q \cdot \frac{n \cdot n_0}{n_1}}}{ \normr{A^{-1}}^{- n \cdot n_0 \cdot \frac{q}{n_1}} } \\
&= \log [ \normr{A} \cdot \normr{A^{-1}} ] \cdot \frac{q}{n_1} < C \cdot \frac{q}{n_1}
\end{align*}

Similarly, we have:

\begin{align*}
 \frac{1}{n \cdot n_0} \log \norm{ \An{n \cdot n_0} (x)} &-  \frac{1}{n_1} \log \norm{ \An{n_1} (x)} =
\frac{1}{n \cdot n_0} \log \frac{ \norm{\An{n \cdot n_0} (x) }}{\norm{\An{n_1} (x)}^{\frac{n \cdot n_0}{n_1}}}  \\
&=  \frac{1}{n \cdot n_0} \log [  \bigl(\frac{ \norm{\An{n \cdot n_0} (x) }}{\norm{\An{n_1} (x)}}\bigr)^{\frac{n \cdot n_0}{n_1}}   \cdot \norm{\An{n \cdot n_0} (x)}^{\frac{q}{n_1}} ] \\
&\le  \frac{1}{n \cdot n_0} \log [  \norm{M (x)^{-1}}^{\frac{n \cdot n_0}{n_1}}   \cdot \norm{\An{n \cdot n_0} (x)}^{\frac{q}{n_1}} ] \\
&\le  \frac{1}{n \cdot n_0} \log [ \normr{A^{-1}}^{q \cdot \frac{n \cdot n_0}{n_1}} \cdot \normr{A}^{n \cdot n_0 \cdot \frac{q}{n_1}} ] 
\\
&= \log [ \normr{A^{-1}} \cdot \normr{A} ] \cdot \frac{q}{n_1} < C \cdot \frac{q}{n_1}
\end{align*}
so the lemma is proven.
\end{proof}

\medskip


The following lemma shows that at finite, fixed scale, the quantities $\frac{1}{n} \, \log s (\An{n} (x))$ and their space averages are H\"{o}lder continuous as functions of the cocycle.

\begin{lemma}\label{cont-finite-lemma}
Let $A \in \cocycles$ and $C>0$ such that $C(A) < C$. Let $n_0$ be any integer, and let $\ep > 0$ such that 
$$\ep \le e^{- 2C n_0} $$

If $B \in \cocycles$ such that $\normr{A - B} \le \ep$, then for any s.v.f. $s\in\allsvf$, we have:
\begin{equation*} 
\abs{ \frac{1}{n_0} \, \log s (\An{n_0} (x)) - \frac{1}{n_0} \, \log s (\Bn{n_0} (x)) } \less \sqrt{\ep} \quad \text{for all } x \in \T
\end{equation*}
In particular,
\begin{equation}\label{cont-finite-scale-eq}
\abs{ \Lasn{n_0} (A) - \Lasn{n_0} (B) } \less \sqrt{\ep}
\end{equation}
\end{lemma}

\begin{proof}
As before, it is enough to prove that for a cocycle $A$ of any dimension, 
\begin{equation*} 
\abs{ \frac{1}{n_0} \, \log \norm{\An{n_0} (x)} - \frac{1}{n_0} \, \log \norm{\Bn{n_0} (x)} } < \sqrt{\ep} \quad \text{for all } x \in \T
\end{equation*}

We have:
\begin{align*}
\abs{ \frac{1}{n_0} \, \log \norm{\An{n_0} (x)} - \frac{1}{n_0} \, \log \norm{\Bn{n_0} (x)} } =
\frac{1}{n_0}  \abs{ \log \frac{\norm{\Bn{n_0} (x)} }{\norm{\An{n_0} (x)} }  } \\
\le \frac{1}{n_0} \log [ 1 + \frac{ \norm{\Bn{n_0} (x) - \An{n_0} (x)} }{\norm{\An{n_0} (x)} }] <
\frac{1}{n_0}  \frac{ \norm{\Bn{n_0} (x) - \An{n_0} (x)} }{\norm{\An{n_0} (x)} }
\end{align*}

Using Trotter's formula, i.e. telescoping sums for the difference of products $\Bn{n_0} (x) - \An{n_0} (x)$, we have that for all $ x \in \T$,
\begin{align*}
\norm{\Bn{n_0} (x) - \An{n_0} (x)}  \le \ep \sum_{i=1}^{n_0-1} \, \normr{A}^i \cdot (\normr{A} + \ep)^{n_0-1-i} < \ep n_0 \, (\normr{A} + 1)^{n_0}
\end{align*}

Since from \eqref{scalingct-eq3} we have  $\norm{\An{n_0} (x)} \ge \normr{A^{-1}}^{- n_0}$, we conclude that
\begin{align*}
\abs{ \frac{1}{n_0} \, \log \norm{\An{n_0} (x)} - \frac{1}{n_0} \, \log \norm{\Bn{n_0} (x)} }  < \\
< \ep \cdot [ (\normr{A} + 1) \cdot \normr{A^{-1}} ]^{n_0} 
<  \ep \, e^{C n_0} \le \sqrt{\ep}
\end{align*}
\end{proof}

\medskip


Let $\tau$ be a signature and let $A \in \cocycles$ be a cocycle with a $\taugp$. Then
$$\lim_{n \to \infty} \Larn{n} (A) = \Lar (A) > 0 \quad \text{for all } \tausvr \, \rho$$
so for $n \ge \nzero (A)$ we have
$$\int_{\T} \, \frac{1}{n} \, \log \, \rho (\An{n} (x) ) \, d x =  \Larn{n} (A) > 0 \quad \text{for all } \tausvr \, \rho$$

Therefore,
$$  \rho (\An{n} (x) )  \neq 1 \quad \text{for a.e. } x \in \T \ \text{ and for all } \tausvr \, \rho$$
which means that $\An{n} (x)$ has a $\taugp$ for a.e. $x \in \T$. Then  the most expanding $\tau$ - flags of $\An{n} (x)$ are well defined for a.e. $x \in \T$. We denote them by $\filtn{n} \, (A) \, (x)$, hence
$$\filtn{n} (A) \, (x) := \hat{v}_{\tau, -} (\An{n} (x))$$

It is well known that the Oseledets filtration $\filt (A)$ is the pointwise limit of 
the most expanding $\tau$-flags  $ \filtn{n} (A)$ of the iterates $\An{n}$ of the cocycle $A$.
Since the  flag manifold $\mathscr{F}_\tau^m$ is bounded, 
 pointwise convergence implies convergence in ${\rm L}^1(\T,\mathscr{F}_\tau^m)$, so 
\begin{equation*} 
\lim_{n\to \infty}  \filtn{n} (A) = \filt (A)\quad \text{ in }\;
{\rm L}^1(\T,\mathscr{F}_\tau^m)
\end{equation*}

\medskip


The following lemma shows the continuity of the mappings
$\filtn{n}$.

\begin{lemma} \label{cont-filtr-finite-scale}
Given a signature $\tau$,
$\filtn{n}:\mathscr{G}_\tau\to  {\rm L}^1(\T,\mathscr{F}^m_\tau)$
is continuous on the open set $\mathscr{G}_\tau$ 
of cocycles $A\in  \cocycles$ with a $\tau$-gap pattern.
\end{lemma}

\begin{proof}
We can assume $n=1$ so that
$\filtn{n}(A)(x)=\hatv_{\tau,-}(A(x))$. 
Let $\tilde{A}:\R\to \GL(m,\R)$ denote the lifting of a cocycle $A\in\cocycles$ 
to the real line.
There are families of analytic functions
$\lambda_i^A:\R\to \R$, and $\hatv_i^A:\R\to \Pp(\R^m)$, with $1\leq i\leq m$, such that:
\begin{enumerate}
\item[(1)] the $\lambda_i^A(x)$
are the singular values of $\tilde{A}(x)$, i.e., $\lambda^A_i(x)^2$ is an eigenvalue of $\tilde{A}(x)^T \tilde{A}(x)$,
\item[(2)]  $\hatv^A_i(x)$ is an eigen-direction of $\tilde{A}(x)^T \tilde{A}(x)$,
\item[(3)] there is a permutation
$\pi\in S_m$ such that
$\lambda^A_i(x+1)=\lambda^A_{\pi_i}(x)$ and 
$\hatv^A_i(x+1)=\hatv^A_{\pi_i}(x)$.
\end{enumerate}
This follows from the perturbation theory of
symmetric operators applied  to the analytic function
$\tilde{A}(x)^T \tilde{A}(x)$.
See~\cite{Kato-Book} chapter two, \S 6.
Item (3) is an obvious consequence of the periodicity of $A$.
We claim that the functions $\lambda^A_i(x)$ and $\hatv^A_i(x)$
depend continuously on the cocycle $A$. More precisely, 
given  $A\in\cocycles$ and $0<r'<r$, there is a constant $K>0$ 
and a neighbourhood $\mathscr{V}$ of $A$ in $\cocycles$ such that
for all $B_1, B_2\in\mathscr{V}$,
\begin{align*}
\norm{\lambda_i^{B_1}-\lambda_i^{B_2}}_{r'} & \leq K\,\norm{B_1-B_2}_r \;,\\
\norm{\hatv_i^{B_1}-\hatv_i^{B_2}}_{r'} & \leq K\,\norm{B_1-B_2}_r \;.
\end {align*}
To see this consider the discriminant function
$$\Delta_A(z):=(-1)^{\frac{m(m-1)}{2}}\,\prod_{i<j} (\lambda^A_i(z)- \lambda^A_j(z))^2\;,$$
that associates to each $z$ the discriminant of the characteristic
polynomial $p_z(\lambda):=\det( A(z)^T\, A(z)-\lambda\, I)$.
Then  $\Delta_A(z)=0$\, iff\, $\tilde{A}(z)^T \tilde{A}(z)$
has a non-simple eigenvalue.
Fix a compact set $\Sigma= [0,3/2]  \times [-r',r']$ 
such that the functions $\lambda_i^A$ and $\hatv_i^A$
extend holomorphically to a neighbourhood of $\Sigma$, and 
 $\Delta_A(z)$ has only real zeros $x_1,\ldots, x_m$ on $\Sigma$.
Let $\Sigma_A=\Sigma\setminus \cup_{j=1}^ m D_j$, where each $D_j$ is
 a small isolating disk around the zero $x_j$.
 Then $\Delta_A(z)$ has a positive lower bound on $\Sigma_A$.
 Take a  neighbourhood $\mathscr{V}$ of $A$ in $\cocycles$ such that 
 for any $B\in\mathscr{V}$  the functions $\lambda_i^B$ and $\hatv_i^B$
extend holomorphically to a neighbourhood of $\Sigma$, and $\Delta_B(z)>0$
on $\Sigma_A$. Combining the maximum modulus principle with 
the implicit function theorem, 
$\lambda_i^B(z)$ and $\hatv_i^B(z)$ are Lipschitz continuous functions of
 $B$. More precisely
\begin{align*}
 \norm{\lambda_i^{B_1}-\lambda_i^{B_2}}_{r'} &= \max_{z\in \Sigma} \abs{\lambda^{B_1}_i(z) - \lambda^{B_2}_i(z)} \\
 &=
  \max_{z\in \Sigma_A} \abs{\lambda^{B_1}_i(z) - \lambda^{B_2}_i(z)}
\lesssim \norm{B_1-B_2}_r\;.
\end{align*}
Similarly
 $$ \norm{\hatv_i^{B_1}-\hatv_i^{B_2}}_{r'} 
\lesssim \norm{B_1-B_2}_r\;.$$
Consider now, for a given a cocycle $A$,  the set of zeros
$$Z(A)=\{\, x\in \T\,:\, \Delta_A(x)=0\,\}\;.$$  
We claim that $Z(A)$ depends continuously on $A$ w.r.t.
the Hausdorff distance on compact subsets of $\T$.
By analiticity of the $\lambda_i^A(z)$ 
there are constants $c>0$
and $p\in\N$ such that for $x\in\R$, and $1\leq i<j\leq m$,
$$ \abs{\lambda_i^A(x)-\lambda_j^A(x)}\geq c\, {\rm dist}(x,Z(A))^p\;.$$
Furthermore these constants can be chosen to be uniform over the neighbourhood $\mathscr{V}$ of $A$ in $\cocycles$.
Hence, taking $x\in Z(B_2)$, and choosing $i<j$
such that $\lambda_i^{B_2}(x)=\lambda_j^{B_2}(x)$,
$$ \norm{B_1-B_2}_r  \gtrsim \abs{\lambda_i^{B_1}(x)-\lambda_j^{B_1}(x)}
\geq c\, {\rm dist}(x,Z(B_1))^p\;,$$
which implies that ${\rm dist}(x,Z(B_1))\lesssim \norm{B_1-B_2}_r^{\frac{1}{p}}$. Thus
 $$d_H(Z(B_1),Z(B_2))\lesssim \norm{B_1-B_2}_r^{\frac{1}{p}} \;.$$
Given $B_1, B_2\in\mathscr{V}$ with $\norm{B_1-B_2}_r=\varepsilon$,
there is a $\varepsilon^{1/p}$-neighbourhood $\mathscr{B}$ of
 $Z(A)=\{\, x\in\T\,:\, \rho_\tau(A(x))=1\,\}$
with  $\abs{\mathscr{B}}\lesssim \varepsilon^{1/p}$, such that
$\Delta_{B_1}(x)>0$ and
$\Delta_{B_2}(x)>0$,  for $x\notin \mathscr{B}$.
Thus  
$\rho_\tau(B_1(x))>1$ and $\rho_\tau(B_2(x))>1$,  for $x\notin \mathscr{B}$.
Hence 
$d\left( \hatv_{\tau,-}(B_1(x)), \, \hatv_{\tau,-}(B_1(x))\right)\lesssim \varepsilon$ 
 for every $x\notin\mathscr{B}$. 
Integrating,   $d(\filt(B_1), \filt(B_2))\lesssim \varepsilon^{1/p} = \norm{B_1-B_2}_r^{1/p}$.
\end{proof}

\begin{remark}\label{Hoder:Oseledets:remark}
 More than continuous, the $\tau$-gap filtrations
$\filtn{n}(A)$ were proven to be
 H\"older continuous functions of the cocycle $A$, with exponent $\theta_n=1/p_n$, where $p_n$ is the maximum order of the zeros of the discriminant $\Delta_{A^{(n)}}(z)$. However,
  if $p_n\to\infty$ the H\"older exponents $\theta_n$ can not be made uniform in $n$. 
\end{remark}

\medskip


The following theorem provides the inductive step in our argument, and it will be used repeatedly throughout the paper. It is based on the reformulation in Proposition~\ref{AP-practical} of the Avalanche Principle and on the LDT~\eqref{LDT1}.

\begin{theorem}\label{indstep-thm}
Let $A \in \cocycles$, $C > 0$ such that $C (A) < C$, let $c = C^{-2}$ and let $\tau$ be a signature. 

Assume that for a scale $n_0 \in \N$ we have:
\begin{align}
\rm{(gaps)} \  & \Larn{n_0} (A) > \ga_0 &  \text{for all } \tausvr \, \rho
\label{indstep-gaps-eq} \\
\rm{(angles)} \  & \abs{ \Lapn{n_0} (A) - \Lapn{2 n_0} (A) } < \eta_0  & \
 \text{for all } \tausvp \, \pi  \label{indstep-angles-eq}
\end{align}
where the positive constants $\ga_0, \eta_0$ are such that
\begin{equation*} 
\ga_0 > 4 \eta_0
\end{equation*}

Fix the constants $\delta$,  $\deltabar$ such that:
\begin{align}
\delta < \frac{\ga_0 - 4 \eta_0}{10} \label{indstep-delta-eq} \\
0 < \deltabar < \delta \label{indstep-deltabar<delta}\\
n_{0}^{-3/4} \le \deltabar \le c \delta^3 /2 \label{indstep-deltabar-eq}
\end{align}

Assume, moreover, that $n_0$ is large enough, $n_0 \ge \nzero (C, \om, \delta)$ so that the LDT~\eqref{LDT1} applies at scale $n_0$ and \eqref{indstep-deltabar<delta},  \eqref{indstep-deltabar-eq} make sense.

\smallskip

If $n_1$ is any other scale such that
\begin{equation*} 
n_1 \le n_0 \cdot e^{\deltabar n_0}
\end{equation*}
then we have:
\begin{align}
\rm{(gaps++)} \  & \Larn{n_1} (A) > \ga_1 &  \text{for all } \tausvr \, \rho
\label{indstep-gaps++-eq} \\
\rm{(angles++)} \  & \abs{ \Lapn{n_1} (A) - \Lapn{2 n_1} (A) } < \eta_1  & \
 \text{for all }\tausvp \, \pi \label{indstep-angles++-eq}
\end{align}
where
\begin{align*}
\ga_1 & = \ga_0 - 4 \eta_0 - 9 \delta - C \frac{n_0}{n_1} 
\\
\eta_1 & = C \frac{n_0}{n_1}  
\end{align*}

\end{theorem}

\begin{remark}\label{indstep-remark}
\rm{ The proof of this theorem will provide two more estimates that are needed later.

The first relates the quantities $\Lapn{n} (A)$ at scales $n = n_1$ and $n \approx n_0$:  
\begin{equation}\label{indstep-n1vsn0-eq}
\abs{ \Lapn{n_1} (A) + \Lapn{n_0} (A) - 2 \Lapn{2 n_0} (A) } < C \frac{n_0}{n_1} \ \text{ for all } \tausvp \, \pi
\end{equation} 

The second is an estimate on the (average in $x$) distance between the most expanding $\tau$ - flags of $\An{n} (x)$ at  scales $n = n_1$  and $n = n_0$:
\begin{equation}\label{distance-exp-flags}
d  \left(\filtn{n_1} (A), \, \filtn{n_0} (A)\right) < C \frac{n_0}{n_1}
\end{equation}
}
\end{remark}

\begin{proof}
Using \eqref{lemma1i}, \eqref{lemma1iii} in Lemma~\ref{average-pointwise_gaps-angles}, assumptions \eqref{indstep-gaps-eq}, \eqref{indstep-angles-eq} on gaps and angles imply the following:
 \begin{align}
\rho (\An{n_0} (x))  & > e^{(\ga_0 - \delta) n_0}  =: \frac{1}{\ka} \label{indstep-gaps-n0}\\
\frac{\pi (\An{2n_0} (x))}{\pi (\An{n_0} (x)) \cdot \pi (\An{n_0} (\transl^{n_0} \, x))} & > \frac{1}{e^{(2\eta_0+4\delta) n_0}}  =: \ep \label{indstep-angles-n0}
\end{align}
for all $x \notin \Bt_{n_0}$, where $\abs{\Bt_{n_0}} \less e^{-c \delta^3 n_0}$ and for any $\tausvr \, \rho$, and $\tausvp \, \pi$ respectively.

Let $n_1 \le n_0 \cdot e^{\deltabar n_0}$. From Lemma~\ref{scales-divide-lemma}, given the estimates we need to obtain for $\ga_1, \, \eta_1$, it is clear that we may assume that $n_1 = n \cdot n_0$ for some integer $n$. We apply the Avalanche Principle~\ref{AP-practical}  to the following matrices:
$$g_i := \An{n_0} (\transl^{i n_0} x) \quad 0 \le i \le n-1$$
so that
$$g^{(n)} = g_{n-1} \cdot \ldots \cdot g_0 = \An{n \cdot n_0} (x) = \An{n_1} (x) $$
where $ x \notin \Bbar_{n_0} := \cup_{i=0}^{n-1} \ \transl^{-i n_0} \ \Bt_{n_0}$, so
$$\abs{ \Bbar_{n_0} } \le n \abs{ \Bt_{n_0} } < e^{-(c \delta^3 - \deltabar) n_0} < e^{- \deltabar n_0} \le \frac{n_0}{n_1}$$

From \eqref{indstep-gaps-n0}, \eqref{indstep-angles-n0} we have:
\begin{align*}
\rho (g_i)  >  \frac{1}{\ka}  &  \quad \text{for all } \tausvr \, \rho,  \  0 \le i \le n-1  \\
\frac{\pi (g_i \cdot g_{i-1})}{\pi ({g_i}) \cdot \pi (g_{i-1})}   >  \ep   & \quad
 \text{for all } \tausvp \, \pi,  \   1 \le i \le n-1
\end{align*}

Since $\delta < \frac{\ga_0 - 4 \eta_0}{10}$, we have $\ga_0 - \delta > 2 (2 \eta_0 + 4 \delta)$, hence
$$\ka \ll \ep^2$$

In fact, 
$$\frac{\ep^2}{\ka} = e^{\gabar_1 n_0} \quad \text{ and } \quad \frac{\ka}{\ep^2} = e^{- \gabar_1 n_0}$$
where 
$\gabar_1  = (\ga_0 - \delta) - 2 (2 \eta_0 + 4 \delta)$, so
$$\gabar_1 = (\ga_0 - 4 \eta_0) - 9 \delta$$

The Avalanche Principle~\ref{AP-practical}  then implies:
\begin{equation}\label{indstep-distances-AP}
d(\hatv_-(g^{(n)}), \hatv_-(g_{0})) < \frac{\ka}{\ep^2}
\end{equation}
\begin{equation} \label{indstep-gaps-AP}
 \rho (g^{(n)}) >  \left(\frac{\ep^2}{\ka}\right)^n
\end{equation}
\begin{equation} \label{indstep-angles-AP}
\sabs{ \log \pi (g^{(n)}) + \sum_{i=1}^{n-2} \log \pi (g_i) -  \sum_{i=1}^{n-1} \log \pi (g_i \cdot g_{i-1}) } < n \cdot \frac{\ka}{\ep^2} 
\end{equation}
where \eqref{indstep-gaps-AP} holds for all $\tausvr \, \rho$ and \eqref{indstep-angles-AP} holds for all $\tausvp \, \pi$.

\medskip

From \eqref{indstep-distances-AP} we then get:
$$d \left(\filtn{n_1} (A) \, (x), \  \filtn{n_0} (A) \, (x)\right) < e^{- \gabar_1 n_0} \quad \text{for all } x \notin \Bbar_{n_0}$$

Integrate in $x \in \T$ to obtain:
\begin{equation*} 
d \left(\filtn{n_1} (A), \, \filtn{n_0} (A)\right) <  e^{- \gabar_1 n_0} +  \abs{ \Bbar_{n_0}} \le C \frac{n_0}{n_1} 
\end{equation*}

The last inequality is due to the fact that $\abs{ \Bbar_{n_0} } \le \frac{n_0}{n_1}$ and since by \eqref{indstep-delta-eq} $\delta < \frac{\ga_0 - 4 \eta_0}{10}$, we have $\gabar_1 > \delta > \deltabar$ hence $e^{- \gabar_1 n_0} < e^{- \deltabar n_0} \le \frac{n_0}{n_1}$. Hence we proved estimate~\eqref{distance-exp-flags}.

\smallskip

From \eqref{indstep-gaps-AP} we have:
$$\rho (\An{n_1} (x)) > (e^{\gabar_1 n_0})^n = e^{\gabar_1 n_1} \quad \text{for all } x \notin \Bbar_{n_0}$$

Using \eqref{lemma1ii} in Lemma~\ref{average-pointwise_gaps-angles} we conclude:
$$\Larn{n_1} (A) > \gabar_1 - C \abs{ \Bbar_{n_0} } > \gabar_1 - C \frac{n_0}{n_1} = \ga_1$$

\smallskip

From \eqref{indstep-angles-AP} we have that for all $x \notin \Bbar_{n_0}$
\begin{align*}
\abs{ \log \, \pi (\An{n_1} (x) ) + \sum_{i=1}^{n-2} \log \, \pi (\An{n_0} (\transl^{i n_0} \, x)) - \\
- \sum_{i=1}^{n-1} \log \, \pi (\An{2 n_0} (\transl^{(i-1) n_0} \, x)) } < n \, e^{- \gabar_1 n_0}
\end{align*}

Dividing both sides by $n_1 = n \cdot n_0$, we get, for all  $x \notin \Bbar_{n_0}$:
\begin{align*}
\abs{ \frac{1}{n_1} \, \log \, \pi (\An{n_1} (x) ) +  \frac{1}{n} \, \sum_{i=1}^{n-2}  \frac{1}{n_0} \, \log \, \pi (\An{n_0} (\transl^{i n_0} \, x)) - \\
-  \frac{2}{n} \, \sum_{i=1}^{n-1}  \frac{1}{2 n_0} \, \log \, \pi (\An{2 n_0} (\transl^{(i-1) n_0} \, x)) } <  \frac{1}{n_0} \, e^{- \gabar_1 n_0}
\end{align*}

Integrate in $x \in \T$ to get:
\begin{align}
\abs{ \Lapn{n_1} (A) + \frac{n-2}{n} \Lapn{n_0} (A) - \frac{2 (n-1)}{n} \Lapn{2 n_0} (A) } < \label{indstep-proof-eq1} \\
<  \frac{1}{n_0} e^{- \gabar_1 n_0} + C \abs{ \Bbar_{n_0}} \le C \frac{n_0}{n_1} \nonumber 
\end{align}

Moreover, the expression in \eqref{indstep-proof-eq1} equals
$$\abs{ [ \Lapn{n_1} (A) + \Lapn{n_0} (A) - 2 \Lapn{2 n_0} (A) ] - \frac{2}{n} [\Lapn{n_0} (A) - \Lapn{2 n_0} (A) ] }$$

Since from \eqref{indstep-angles-eq} we have $$\abs{ \Lapn{n_0} (A) - \Lapn{2 n_0} (A) } < \eta_0$$ we conclude:
$$\abs{ \Lapn{n_1} (A) + \Lapn{n_0} (A) - 2 \Lapn{2 n_0} (A) } < C \frac{n_0}{n_1} + \frac{2}{n} \, \eta_0$$

But $2 \eta_0 < \ga_0$ and $\ga_0  < C$ due to  \eqref{scalingct-eq6} and \eqref{indstep-gaps-eq}, so we obtain the following:
\begin{equation}\label{indstep-proof-eq3}
\abs{ \Lapn{n_1} (A)  + \Lapn{n_0} (A) - 2 \Lapn{2 n_0} (A) } \less C \frac{n_0}{n_1}
\end{equation}
which establishes \eqref{indstep-n1vsn0-eq}. Moreover,
everything we did applies also at scale $2 n_1$, so we get:
\begin{equation}\label{indstep-proof-eq4}
\abs{ \Lapn{2 n_1} (A)  + \Lapn{n_0} (A) - 2 \Lapn{2 n_0} (A) } \less C \frac{n_0}{n_1}
\end{equation}

Then \eqref{indstep-proof-eq3} and \eqref{indstep-proof-eq4} imply \eqref{indstep-angles++-eq}.
\end{proof}

\medskip


The following lemma shows that if a cocycle $A$ has a gap pattern, then that gap pattern is uniform in a neighborhood of $A$, and it holds at all finite scales that are large enough.

\begin{lemma}\label{uniformgaps-lemma}
Let $A \in \cocycles$ and let $\tau$ be a signature such that:
\begin{equation*} 
\Lar (A) > \ga > 0 \quad \text{for all } \tausvr \, \rho
\end{equation*}
Then there are $\ep = \ep (A, \ga, \om) > 0$ and $\nzero = \nzero (A, \ga, \om) \in \N$ such that if $B \in \cocycles$ with 
$\normr{ A- B} < \ep$, then we have:
\begin{equation*} \label{uniformgaps-lemma-cn}
\Larn{n} (B) > \ga 
\end{equation*}
for all $n \ge \nzero$ and for all $\tausvr \, \rho$ 

In particular,
\begin{equation*} 
\Lar (B) > \ga \quad \text{for all } \tausvr \, \rho
\end{equation*}
\end{lemma} 

\begin{proof} We will use the inductive step Theorem~\ref{indstep-thm}. We choose a large enough initial scale $\nzero (A, \ga, \om)$ at which conditions (gaps) \eqref{indstep-gaps-eq} and (angles)  \eqref{indstep-angles-eq} in Theorem~\ref{indstep-thm} are satisfied for the cocycle $A$. We then choose an $\ep$ - neighborhood of $A$ (where $\ep$ depends on $\nzero$, hence on $A, \ga, \om$) in such a way that conditions \eqref{indstep-gaps-eq}, \eqref{indstep-angles-eq} in Theorem~\ref{indstep-thm} are satisfied uniformly for cocycles $B$ in this neighborhood. Because we want the conclusion \eqref{uniformgaps-lemma-cn} to hold for {\em all} scales $n \ge \nzero$, and not only for an increasing sequence of scales, we will actually choose $\nzero, \, \ep$ so that \eqref{indstep-gaps-eq}, \eqref{indstep-angles-eq} are satisfied for all scales $n_0$ in a finite {\em interval} of integers $\scale_0$. Then the conclusions (gaps++)  \eqref{indstep-gaps++-eq} and (angles++) \eqref{indstep-angles++-eq} in Theorem~\ref{indstep-thm} will hold for all scales $n_1$ in an {\em interval} of integers $\scale_1$, where $\scale_0$ and $\scale_1$ overlap. Continuing this inductively, we obtain  \eqref{uniformgaps-lemma-cn} for all $n \ge \nzero$.

\medskip

$\blob$ We now establish assumptions (gaps) \eqref{indstep-gaps-eq} and (angles)  \eqref{indstep-angles-eq} of the inductive step Theorem~\ref{indstep-thm} in a finite interval  $\scale_0$ and uniformly in a small enough neighborhood of the cocycle $A$.

\smallskip

Since $\Lar (A) > \ga > 0$ for all $\tausvr \, \rho$, there are positive constants $\gabar, \,  \etabar$ depending only on $A$ and $\ga$ such that 
$$ \gabar - 4 \etabar > \ga > 0$$ 
and
$$\Lar (A) > \gabar  \quad \text{for all } \tausvr \, \rho$$

To get this, just choose $\gabar$ such that $\La_{\rho_\tau} (A) > \gabar > \ga$ and then pick $0 < \etabar < \frac{\gabar - \ga}{4}$

For any $\tausvr \, \rho$ we have:
$$\lim_{n \to \infty} \, \Larn{n} (A) = \Lar (A) > \gabar$$
and for any $\tausvp \, \pi$ we have:
$$\lim_{n \to \infty} \, \Lapn{n} (A) = \Lap (A)$$

Then there is $\nzero = \nzero (A, \ga)$ such that for all $n_0 \ge \nzero$ we have:
 \begin{align}
\Larn{n_0} (A) > \gabar \quad  &  \text{for all } \tausvr \, \rho
\label{lsc-gaps-An0} \\
\abs{ \Lapn{n_0} (A) - \Lapn{2 n_0} (A) } < \etabar \quad & 
 \text{for all } \tausvp \, \pi  \label{lsc-angles-An0}
\end{align}
By the choice of these constants, $\gabar > 4 \etabar$.

We  also assume that $\nzero$ is large enough, depending on the bound $C$ on the size of $A$ and on $\om$ such that the LDT~\eqref{LDT2} applies for $n \ge \nzero$ uniformly in some neighborhood of $A$. We make some additional assumptions on the size of $\nzero$. It will become clear later why we need them.  
\begin{align}
\nzero & \ge (2 / c)^4  \quad \text{ where } c = C(A)^{-2} \label{nzero-h1}\\
\nzero & \ge ( 10 / \ga )^4  \label{nzero-h2}\\
n  & \le e^{n^{1/4}}  \quad \quad \text{ if } n \ge \nzero \label{nzero-h3}
\end{align}

 Note that all conditions on $\nzero$ depend only on some measurements on the size of $A$ (so they are uniform in a small neighborhood of $A$) and on $\ga$, $\om$.
 
 \smallskip

Let $\scale_0 := [\nzero, \nzero^3]$, and let $\ep := e^{-C \nzero^4}$.

We show that \eqref{lsc-gaps-An0}, \eqref{lsc-angles-An0}, with slightly different constants $\gabar, \, \etabar$, hold for all $n_0 \in \scale_0$ and for all cocycles $B \in \cocycles$ such that $\normr{A - B} < \ep$.

Let $B$ be such a cocycle and let $n_0 \in \scale_0$, hence $n_0 \le \nzero^3$. Then $\ep = e^{-C \nzero^4} < e^{- 2 C n_0}$ so we can apply Lemma~\ref{cont-finite-lemma}. 

For any s.v.f. $s \in \allsvf$, from \eqref{cont-finite-scale-eq} we have:
\begin{equation}\label{lsc-lemma-cont-n0}
\abs{ \Lasn{n_0} (A) - \Lasn{n_0} (B) } < \sqrt{\ep}
\end{equation}

By taking $s$ in \eqref{lsc-lemma-cont-n0} to be any $\tausvr \, \rho$, we get that for any $n_0 \in \scale_0$ 
\begin{equation*}
\Larn{n_0} (B) > \Larn{n_0} (A) - \sqrt{\ep} > \gabar - \sqrt{\ep} =: \ga_0
\end{equation*}

We may, of course, assume that \eqref{lsc-lemma-cont-n0} also holds at scale $2 n_0$. That is because in fact \eqref{lsc-lemma-cont-n0} holds for a larger interval of scales, one that includes $[\nzero, 2 \nzero^3]$. 

By taking $s$ in \eqref{lsc-lemma-cont-n0} to be any $\tausvp \, \pi$, we get that for any $n_0 \in \scale_0$
\begin{align*}
\abs{ \Lapn{n_0} (B) - \Lapn{2 n_0} (B) } < \\
< \abs{ \Lapn{n_0} (A) - \Lapn{2 n_0} (A) } + 2 \sqrt{\ep} < \etabar + 2 \sqrt{\ep} =: \eta_0
\end{align*}

We have $\ga_0 - 4 \eta_0 = \gabar - 4 \etabar - 9 \sqrt{\ep}$.

Since $\gabar - 4 \etabar > \ga$ and $\ep := e^{-C \nzero^4}$, by making an additional assumption on the magnitude of $\nzero$, one that also only depends on $A$ and $\ga$, we may assume that $\ga_0 - 4 \eta_0 > \ga$.

We summarize this step of the proof as follows.

\smallskip

There are $\nzero = \nzero (A, \ga) \in \N$ and $\ep = \ep (A, \ga) > 0$ such that if $\scale_0 := [\nzero, \nzero^3]$, then we have the following:
\begin{align}
\Larn{n_0} (B) > \ga_0 \quad  &  \text{for all } \tausvr \, \rho
\label{lsc-gaps-Bn0} \\
\abs{ \Lapn{n_0} (B) - \Lapn{2 n_0} (B) } < \eta_0 \quad & 
 \text{for all } \tausvp \, \pi  \label{lsc-angles-Bn0}
\end{align}
for all $n_0 \in \scale_0$ and for all $B \in \cocycles$ such that $\normr{A-B} < \ep$.

Moreover, $\ga_0 > 4 \eta_0$, and in fact
\begin{equation}\label{lsc-proof-eq1}
\ga_0 - 4 \eta_0 > \ga
\end{equation}

\medskip

$\blob$ We now explain how the inductive process works. We apply the inductive step Theorem~\ref{indstep-thm} in order to obtain similar estimates to the ones above at a larger scale $n_1 \gg n_0$. In order to do that, we have to choose $0 < \delta_0 < \frac{\ga_0 - 4 \eta_0}{10}$, where $\delta_0$ represents the (allowed) size of the deviation from the mean of quantities of the form $\frac{1}{n} \, \log s( \Bn{n} (x))$ at scale $n = n_0$. The estimate we  get on gaps at scale $n_1$ will be of the form:
$$\Larn{n_1} (B) > \ga_1$$
where
\begin{align*} 
\ga_1 & = (\ga_0 - 4 \eta_0) - 9 \delta_0 - C \frac{n_0}{n_1} > \ga - 9 \delta_0 - C \frac{n_0}{n_1} \\
\eta_1 & = C \frac{n_0}{n_1}
\end{align*}
We continue this inductively, moving to a scale $n_2 \gg n_1$. Choosing  $0 < \delta_1 < \frac{\ga_1 - 4 \eta_1}{10}$ we get
\begin{align*}
\ga_2 & = \ga_1 - 4 \eta_1 - 9 \delta_1 - C \frac{n_1}{n_2} = 
\\ & =  (\ga_0 - 4 \eta_0)  - 9 (\delta_0 + \delta_1) - 5  C \frac{n_0}{n_1} - C \frac{n_1}{n_2} \\
\eta_1 & = C \frac{n_0}{n_1}
\end{align*}
and so on.

Our goal in this lemma is to obtain {\em sharp} lower bounds on the size $\ga_1, \ga_2, \ldots, \ga_k, \ldots $ of the gaps, so in the limit, $\Lar (B)$ stays above $\ga$. 

Therefore, the quantities $\delta_0, \delta_1, \ldots, \delta_k, \ldots$ that are subtracted from the size of these gaps at each step of the induction have to sum up to a small enough constant. In order for that to happen, the size $\delta_k$ of the deviation set has to be adapted to the scale $n_k$. We then choose $\delta_k := n_{k}^{-1/6}$ as in LDT \eqref{LDT2}.

\medskip

$\blob$ For every step $k \ge 0$, we define the scale $n_k$ and the corresponding size $\delta_k$ of the deviation set.

Let $\scale_0 := [\nzero, \nzero^3]$, $\scale_1 := [\nzero^2, \nzero^6], \ldots $ and in general:
\begin{equation*} 
\scale_k := [\nzero^{2^k}, \nzero^{3\cdot 2^k}]
\end{equation*}

These intervals clearly overlap, so their union is the set of all integers $n \ge \nzero$.

Moreover, if $n_{k+1} \in \scale_{k+1}$, let $n_k := \intpart{ \sqrt{n_{k+1}} }$. Then clearly 
$$n_{k+1} = n_{k}^2 + q, \quad \text{ where } 0 \le q \le 2 n_k$$

Using Lemma~\ref{scales-divide-lemma}, we have that for any s.v.f. $s \in \allsvf$,
$$\abs{ \Lasn{n_{k+1}} (B) - \Lasn{n_k^2} (B) } < C \, \frac{q}{n_{k+1}} \less C \frac{n_k}{n_{k+1}}$$

Since all of our estimates will be of this order at best, we may assume from now on that if $n_{k+1} \in \scale_{k+1}$, then 
$n_{k+1} = n_{k}^2$ for some $n_k \in \scale_k$.  In particular, these scales satisfy $n_k = n_0^{2^k}$, so
\begin{equation*}
\sum_{k=0}^{\infty} \frac{n_k}{n_{k+1}} =\sum_{k=0}^{\infty} \frac{1}{n_{k}} < \frac{2}{n_0} \le  \frac{2}{\nzero} 
\end{equation*}

For each step $k \ge 0$ we choose $\delta_k := n_{k}^{-1/6}$ as in LDT \eqref{LDT2}.
Note that
\begin{equation*}
\sum_{k=0}^{\infty}  \delta_k = \sum_{k=0}^{\infty}  \frac{1}{n_k^{1/6}} < \frac{2}{n_0^{1/6}} \le \frac{2}{\nzero^{1/6}}
\end{equation*} 

We make a final assumption on the magnitude of $\nzero$. It will be clear immediately why this is needed. We have
\begin{align*}
(\ga_0 - 4 \eta_0)  - 9 \, \sum_{k=0}^{\infty}  \delta_k  -5  C \, \sum_{k=0}^{\infty} \frac{n_k}{n_{k+1}} > \\
>  (\ga_0 - 4 \eta_0)  -  \frac{18}{\nzero^{1/6}}  - \frac{10 C}{\nzero} 
\end{align*} 
Since $\ga_0 - 4 \eta_0 > \ga$, by choosing in the beginning $\nzero$ large enough depending on $A$ and $\ga$, we may then assume that
\begin{equation}\label{lsc-proof-eq3}
(\ga_0 - 4 \eta_0)  - 9 \, \sum_{k=0}^{\infty}  \delta_k -5  C \, \sum_{k=0}^{\infty} \frac{n_k}{n_{k+1}} > \ga
\end{equation}
Moreover, since for all $k \ge 0$ we have
$$\delta_k \le \delta_0 = n_0^{-1/6} \le \nzero^{-1/6} < \frac{\ga}{10}$$
where the last inequality is due to \eqref{nzero-h2}, we conclude that:
\begin{equation}\label{lsc-lemma-deltak}
0 < \delta_k < \frac{\ga}{10}
\end{equation}

Now set $\deltabar_k := n_k^{-3/4}$ so
\begin{equation}\label{indstep-deltabark<deltak}
0 < \deltabar_k < \delta_k
\end{equation}
and since $n_k \ge n_0 \ge \nzero \ge (2 / c)^4$ (the last inequality being  \eqref{nzero-h1}) then $n_k^{-3/4} < \frac{c}{2} \, n_k^{-1/2}$, so
\begin{equation}\label{lsc-lemma-deltabark}
n_k^{-3/4}  \le \deltabar_k \le c \delta_k^3 / 2
\end{equation}

Estimates \eqref{lsc-proof-eq3} - \eqref{lsc-lemma-deltabark} will ensure that \eqref{indstep-delta-eq} - \eqref{indstep-deltabar-eq} are satisfied at all steps of the induction process.

\medskip

$\blob$ We now describe the induction process, which uses the inductive step Theorem~\ref{indstep-thm} to pass from a scale $n_k \in \scale_k$ to $n_{k+1} \in \scale_{k+1}$.

\smallskip

At step $k=0$, \eqref{lsc-gaps-Bn0}, \eqref{lsc-angles-Bn0} imply the (gaps) and respectively the (angles) conditions in the inductive step Theorem~\ref{indstep-thm}, while \eqref{lsc-lemma-deltak} and \eqref{lsc-proof-eq1} imply the condition \eqref{indstep-delta-eq} on the relation between $\ga_0$, $\eta_0$, $\delta_0$. The conditions on $\deltabar_0$ are satisfied due to \eqref{indstep-deltabark<deltak} and
\eqref{lsc-lemma-deltabark}.
    
Let $n_1 \in \scale_1$. Since we may assume that $n_1 = n_0^2$ for some $n_0 \in \scale_0$,  then from \eqref{nzero-h3} we get:
$$n_1 = n_0^2 \le e^{n_0^{1/4}} = e^{\deltabar_0 n_0}$$

Therefore, the conclusion of the inductive step Theorem~\ref{indstep-thm} applies to any $n_1 \in \scale_1$ and we have:
\begin{align*}
\rm{(gaps++)} \  \ & \Larn{n_1} (B) > \ga_1 &  \text{for all } \tausvr \, \rho
\\
\rm{(angles++)} \  \ & \abs{ \Lapn{n_1} (B) - \Lapn{2 n_1} (B) } < \eta_1  & \
 \text{for all }\tausvp \, \pi 
\end{align*}
where
\begin{align*}
\ga_1 & = (\ga_0 - 4 \eta_0) - 9 \delta_0 - C \frac{n_0}{n_1} \\
\eta_1 & = C \frac{n_0}{n_1} 
\end{align*}

Moreover,
\begin{align*}
\ga_1 - 4 \eta_1 = (\ga_0 - 4 \eta_0)  - 9 \delta_0 - 5 C \frac{n_0}{n_1} > \\
> (\ga_0 - 4 \eta_0)  - 9 \, \sum_{k=0}^{\infty}  \delta_k - 5  C \, \sum_{k=0}^{\infty} \frac{n_k}{n_{k+1}} > \ga > 0
\end{align*}
This together with \eqref{lsc-lemma-deltak} shows that
$$0 < \delta_1 < \frac{\ga}{10} < \frac{\ga_1 - 4 \eta_1}{10}$$
hence the relative conditions on $\ga_1, \, \eta_1, \, \delta_1$ are satisfied, while the ones on $\deltabar_1$ are (always) satisfied by \eqref{indstep-deltabark<deltak} and \eqref{lsc-lemma-deltabark}.

The inductive step Theorem~\ref{indstep-thm} applies again, and we can pass to a scale $n_2 \in \scale_2$ and so on.

Continuing this procedure inductively, we conclude that for every $B \in \cocycles$ such that $\normr{A - B} < \ep$, and for every step $k \ge 1$, if $n_k \in \scale_k$ then the following hold:
 \begin{align*}
\rm{(gaps)} \  \ & \Larn{n_k} (B) > \ga_k &  \text{for all } \tausvr \, \rho
\\
\rm{(angles)} \  \ & \abs{ \Lapn{n_k} (B) - \Lapn{2 n_k} (B) } < \eta_k  & \
 \text{for all }\tausvp \, \pi 
\end{align*}
where
\begin{align*}
\ga_k & = (\ga_0 - 4 \eta_0) - 9 \, \sum_{j=0}^{k-1}  \delta_j - 5  C \, \sum_{j=0}^{k-2} \frac{n_j}{n_{j+1}} - C \, \frac{n_{k-1}}{n_{k}} \\
\eta_k & = C \frac{n_{k-1}}{n_{k}} 
\end{align*}
Then clearly
\begin{align*}
\ga_k - 4 \eta_k =  (\ga_0 - 4 \eta_0) - 9 \, \sum_{j=0}^{k-1}  \delta_j - 5  C \, \sum_{j=0}^{k-1} \frac{n_j}{n_{j+1}} > \\
> (\ga_0 - 4 \eta_0)  - 9 \, \sum_{k=0}^{\infty}  \delta_k - 5  C \, \sum_{k=0}^{\infty} \frac{n_k}{n_{k+1}} > \ga > 0
\end{align*}
In particular, 
$$\Larn{n_k} (B) > \ga_{k} > \ga \quad \text{for all } k \ge 0$$
Since the scale intervals $\scale_k$, $k \ge 0$ cover the set of all integers $n \ge \nzero$, the estimate \eqref{uniformgaps-lemma-cn} follows.
\end{proof}

The  lemma we have just proven implies  the following:

\begin{corollary}\label{lscgaps}
Given a signature $\tau$ and a constant $\ga > 0$, for every $\tausvr \, \rho$, the set  $\{ A \colon \Lar (A) > \ga \}$
is open in $\cocycles$.
Moreover, the set of all cocycles with a $\taugp$ is also open, and the map $A \mapsto \Lar (A)$ is lower semi-continuous.
\end{corollary}

\medskip


The following lemma establishes a rate of converges of the quantities $\Lapn{n} (A)$ to the corresponding $\taublockL$ $\Lap (A)$ for a cocycle $A$ satisfying a $\taugp$. It also establishes a rate of convergence of the most expanding $\tau$-flags $\filtn{n} (A)$ to the Oseledets filtration $\filt (A)$. The rate of convergence is uniform in a neighborhood of $A$.

\begin{lemma}\label{rateconv-lemma}
Let $A \in \cocycles$ be a cocycle and let $\tau$ be a signature. Assume that
\begin{equation*} 
\Lar (A) > \ga > 0 \quad \text{for all } \tausvr \, \rho
\end{equation*}

There are $\ep = \ep (A, \ga, \om) > 0$ and $K = K (A, \ga, \om) > 0$ such that if $B \in \cocycles$ with  $\normr{ A - B } < \ep$, then for every $\tausvp \, \pi$ we have:
\begin{equation}\label{rateconv-c}
\abs{ \Lapn{n} (B) - \Lap (B) } < K \, \frac{\log n}{n} \quad \text{for all } n \ge 2
\end{equation}

Moreover, there is $b = b (A, \ga, \om) > 0$ such that:
\begin{equation}\label{rateconv-flags}
d \,  ( \filtn{n} (B), \, \filt (B) ) < K \, e^{- b \, n}
\end{equation}
\end{lemma}

\begin{proof}
We will apply the same inductive procedure as in Lemma~\ref{uniformgaps-lemma}. While the goal in Lemma~\ref{uniformgaps-lemma} was to obtain sharp lower bounds for the gaps, uniformly in a neighborhood of the cocycle $A$, in this statement we need sharp estimates on the {\em ``angles''}, and they need to be uniform in a neighborhood of the cocycle as well. 

From the inductive step Theorem~\ref{indstep-thm}, if we have appropriate estimates at a scale $n_0$, then for the next scale $n_1 \gg n_0$, we get
$$\abs{ \Lapn{n_1} (B) - \Lapn{2 n_1} (B) } < C \, \frac{n_0}{n_1}$$
Therefore, in order to get a sharp rate of convergence of $\Lapn{n} (B)$ to $\Lap (B)$, the next scale $n_1$ should be as large relative to the previous scale $n_0$ as it can be, namely, from the inductive step Theorem~\ref{indstep-thm}, of order $n_1 \approx e^{\deltabar n_0}$

Now that we know, from Lemma~\ref{uniformgaps-lemma}, that the lower bound $\ga$ is uniform in a neighborhood of the cocycle $A$, and it applies to all (large enough) finite scales $n$ as well, we can keep the size $\delta$ of the deviation set the same at all steps of the induction, and hence we can always choose $n_{k+1} \approx e^{\deltabar n_k}$, so that $\frac{n_{k}}{n_{k+1}} \approx  \deltabar^{-1} \, \frac{\log n_{k+1}}{n_{k+1}}$, which will lead to the estimate \eqref{rateconv-c}.  

\smallskip

As in Lemma~\ref{uniformgaps-lemma}, we first establish the assumptions on (gaps) \eqref{indstep-gaps-eq} and on (angles) \eqref{indstep-angles-eq} from the inductive step Theorem~\ref{indstep-thm} for $n_0$ in a range of initial scales and for all cocycles $B$ near $A$.

From Lemma~\ref{uniformgaps-lemma}, there are constants $\nzero = \nzero (A, \ga, \om) \in \N$ and $\ep = \ep(A, \ga, \om) > 0$ such that:
\begin{equation}\label{lsc-gaps-eq}
\Larn{n} (B) > \ga \quad \text{for all } n \ge \nzero
\end{equation}
and for all $B \in \cocycles$ such that $\normr{A-B} < \ep$.

Choose  $0 < \etabar < \frac{\ga}{4}$. Since by definition, for any $\tausvp \, \pi$ (or for any s.v.f. for that matter) we have
$$\lim_{n \to \infty} \Lapn{n} (A) = \Lap (A)$$
we may assume, by possibly increasing $\nzero$, that
\begin{equation}\label{lsc-angles-An0-eq}
\abs{ \Lapn{n} (A) - \Lapn{ 2n} (A) } < \etabar \quad \text{for all } n \ge \nzero 
\end{equation} 

Let $$\scale_0 := [\nzero, e^{\nzero}] =: [n_0^-, n_0^+]$$
and let
$$\ep := e^{- 4 C \, e^{\nzero}}$$
Therefore, if $n_0 \in \scale_0$, then $\ep \le e^{- 4 C n_0}$, so if $\normr{A-B} \le \ep$, then $\normr{A-B} \le e^{- 2 C \, 2 n_0}$. Lemma~\ref{cont-finite-lemma} then applies at scales $n_0, 2 n_0$ and we have:
\begin{align}
\abs{ \Lapn{n_0} (A) - \Lapn{n_0} (B) } < \sqrt{\ep} 
\nonumber
\\
\abs{ \Lapn{2n_0} (A) - \Lapn{2n_0} (B) } < \sqrt{\ep} \label{cont-2n0-AB}
\end{align}
From \eqref{lsc-angles-An0-eq} - \eqref{cont-2n0-AB} we have that if $n_0 \in \scale_0$ and if $B \in \cocycles$ such that $\normr{A - B} < \ep$, then:
$$\abs{ \Lapn{n_0} (B) - \Lapn{ 2n_0} (B) } < \etabar + 2 \sqrt{\ep} =: \eta$$
Since $\etabar < \frac{\ga}{4}$, by choosing $\nzero$ large enough depending on $\ga$, we may assume that $\ep$ is small enough that $\eta < \frac{\ga}{4}$. 

We conclude: for all $n_0 \in \scale_0$ and for every cocycle $B$ such that $\normr{A-B} \le \ep$, we have:
\begin{align*}
\Larn{n_0} (B) > \ga \quad & \text{for all } \tausvr \, \rho \\
\abs{ \Lapn{n_0} (B) - \Lapn{2 n_0} (B) } < \eta \quad & \text{for all } \tausvp \, \pi
\end{align*}
where $\ga > 4 \eta$.

Fix a constant $\delta$ such that 
\begin{equation}\label{lsc-delta}
0 < \delta < \frac{\ga - 4 \eta}{10}
\end{equation}
and a constant $\deltabar$ such that
\begin{align}
0 < \deltabar < \delta \label{lsc-deltabar<delta} \\
\nzero^{-3/4} \le \deltabar \le c \delta^3 / 2  \label{lsc-deltabar} 
\end{align}
For the last two relations to be possible, we should of course assume from the beginning that $\nzero > \delta^{- 4/3}$ and $\nzero > (c / 2)^{- 4/3} \,   \delta^{-4}$, which again are assumptions on the magnitude of $\nzero$ that depend only on $A$ and $\ga$.

We can now apply the inductive step Theorem~\ref{indstep-thm} and Remark~\ref{indstep-remark}, and get that for every scale $n_1$ such that 
$$n_1 \le n_0 \, e^{\deltabar n_0}$$
we have (we are only interested in the conclusions regarding the angles, and the distances between most expanding $\tau$-flags):
$$\abs{ \Lapn{n_1} (B) - \Lapn{2n_1} (B) } < C \frac{n_0}{n_1}$$
$$d \, (\filtn{n_1} (B), \, \filtn{n_0} (B) ) < C \frac{n_0}{n_1}$$

Now define the next scale range 
$$\scale_1 := [e^{\deltabar n_0^-}, e^{\deltabar n_0^+}] =: [n_1^-, n_1^+]$$
Here and throughout the paper, we of course mean the integer part of expressions like the ones above.

We may of course assume that $\deltabar < 1$, so clearly $\scale_0$ and $\scale_1$ overlap. Moreover, if $n_1 \in \scale_1$, let $n_0 := \intpart{\deltabar^{-1} \, \log n_1}$. Then $n_0 \in \scale_0$, $\frac{n_0}{n_1} \approx \deltabar^{-1} \, \frac{\log n_1}{n_1}$, and $n_1 < e^{\deltabar n_0} \cdot e^{\deltabar} < n_0 \, e^{\deltabar n_0}$.   

We conclude: if $n_1 \in \scale_1$, and if $B \in \cocycles$, $\normr{A-B} < \ep$, then for any $\tausvp \, \pi$ we have:
\begin{equation}\label{lsc-angles-n1}
\abs{ \Lapn{n_1} (B) - \Lapn{2n_1} (B) } < C \frac{n_0}{n_1} \approx \frac{C}{\deltabar} \, \frac{\log n_1}{n_1} =: K  \, \frac{\log n_1}{n_1} =: \eta_1
\end{equation}
where $K$ depends only on $C$ and $\deltabar$, hence only on $A, \ga, \om$.
Moreover,
\begin{equation*} 
d \, (\filtn{n_1} (B), \, \filtn{n_0} (B) ) < C \frac{n_0}{n_1}  \approx K \, \frac{\deltabar n_0}{{e^{\deltabar n_0}}} < K \, e^{ - \frac{\deltabar}{2} \, n_0} =: K \, e^{-  b \, n_0}
\end{equation*}

We may assume, by possibly increasing $\nzero$, depending only on  $A, \ga, \om$, that:
\begin{equation}\label{lsc-etak-eq}
K  \, \frac{\log n}{n} < \eta \quad \text{for all } n \ge \nzero
\end{equation}

Moreover, according to \eqref{lsc-gaps-eq}, we have:
\begin{equation}\label{lsc-gaps-n1}
\Larn{n_1} (B) > \ga \quad \text{for all } \tausvr \, \rho
\end{equation}

We use the same constants $\delta, \, \deltabar$ as above, and clearly from \eqref{lsc-delta}, \eqref{lsc-etak-eq}
\begin{equation}\label{lsc-delta-n1}
\delta < \frac{\ga - 4 \eta}{10} < \frac{\ga - 4 \eta_1}{10} 
\end{equation}

Then \eqref{lsc-gaps-n1}, \eqref{lsc-angles-n1}, together with \eqref{lsc-delta-n1}, \eqref{lsc-deltabar<delta}, \eqref{lsc-deltabar} show that  the assumptions of the inductive step Theorem~\ref{indstep-thm} are satisfied at scale $n_1$. We then conclude that for every scale $n_2$ such that 
$$n_2 \le n_1 \, e^{\deltabar n_1}$$
we have:
\begin{equation}\label{lsc-angles-n2-part}
\abs{ \Lapn{n_2} (B) - \Lapn{2n_2} (B) } < C \frac{n_1}{n_2}
\end{equation}
\begin{equation*} 
d \, (\filtn{n_2} (B), \, \filtn{n_1} (B) ) < C \frac{n_1}{n_2} 
\end{equation*}

Now define 
$$\scale_2 := [e^{\deltabar n_1^-}, e^{\deltabar n_1^+}] =: [n_2^-, n_2^+]$$
Since the map $n \mapsto e^{\deltabar n}$ is increasing and $\scale_0, \, \scale_1$ overlap, so do $\scale_1, \, \scale_2$. 

If $n_2 \in \scale_2$, let $n_1 := \intpart{\deltabar^{-1} \, \log n_2}$. Then $n_1 \in \scale_1$, $\frac{n_1}{n_2} \approx \deltabar^{-1} \, \frac{\log n_2}{n_2}$, and $n_2 < e^{\deltabar n_1} \cdot e^{\deltabar} < n_1 \, e^{\deltabar n_1}$,   hence \eqref{lsc-angles-n2-part} applies to every $n_2 \in \scale_2$.

We conclude: if $n_2 \in \scale_2$ and if $B \in \cocycles$ such that $\normr{A-B} < \ep$, then for all $\tausvp \, \pi$:
\begin{equation*} 
\abs{ \Lapn{n_2} (B) - \Lapn{2n_2} (B) } < C \frac{n_1}{n_2} \approx \frac{C}{\deltabar} \, \frac{\log n_2}{n_2} = K  \, \frac{\log n_2}{n_2} =: \eta_2
\end{equation*}
Moreover,
\begin{equation}\label{lsc-distances-n2}
d \, (\filtn{n_2} (B), \, \filtn{n_1} (B) ) < C \frac{n_1}{n_2}  \approx K \, \frac{\deltabar n_1}{{e^{\deltabar n_1}}} < K \, e^{ - \frac{\deltabar}{2} \, n_1} = K \, e^{-  b \, n_1}
\end{equation}

It is clear now how the induction process continues. There is a sequences of integer intervals 
$\scale_k := [n_k^-, n_k^+]$ such that $\scale_0 := [\nzero, e^{\nzero}]$ and $\scale_{k+1} = [e^{\deltabar n_k^-}, e^{\deltabar n_k^+}]$. For every $k \ge 1$, if  $n_k \in \scale_k$, and $B \in \cocycles$, $\normr{A-B} < \ep$, then:
\begin{equation*} 
\abs{ \Lapn{n_k} (B) - \Lapn{2n_k} (B) } < K  \, \frac{\log n_k}{n_k} \quad \text{for all } \tausvp \, \pi
\end{equation*}
and
\begin{align*}
d \, (\filtn{n_k} (B), \, \filtn{n_{k-1}} (B) ) & < C \frac{n_{k-1}}{n_k}   \approx K \, \frac{\deltabar n_{k-1}}{{e^{\deltabar n_{k-1}}}} \notag \\
& < K \, e^{ - \frac{\deltabar}{2} \, n_{k-1}} = K \, e^{-  b \, n_{k-1}}
\end{align*}
Moreover, the integer intervals $\scale_1, \scale_2, \ldots $ overlap, so their union is the set of all integers $n \ge n_1^- = e^{\deltabar \nzero} =: \overline{\nzero} $.

 Clearly $\overline{\nzero}$ depends only on $A, \ga, \om$, and for all $n \ge \overline{\nzero}$ we have:
$$ \abs{ \Lapn{n} (B) - \Lapn{2n} (B) } < K  \, \frac{\log n}{n} \quad \text{for all } \tausvp \, \pi$$

Summing up over $j$ the expressions $ \abs{ \Lapn{2^j n} (B) - \Lapn{2 \cdot 2^j n} (B) }$, we conclude that for every $\tausvp \, \pi$ we have:
$$ \abs{ \Lapn{n} (B) - \Lap (B) } \less K  \, \frac{\log n}{n} $$
for all $n \ge \overline{\nzero}$ and for all cocycles $B$ such that $\normr{A-B} < \ep$.

Moreover, for all $n \ge \nzero$, there is $k \ge 1$ and $n_k \in \scale_k$ such that $n = n_k$. Then
\begin{align*}
d \, ( \filt (B), \filtn{n} (B) ) = d \, ( \filt (B), \filtn{n_k} (B) ) \le \\ \sum_{j \ge k} \, d \, ( \filtn{n_{j+1}} (B), \filtn{n_j} (B) ) 
\less K \, e^{- b n_k} = K \, e^{- b n}  
\end{align*}

By increasing the constant $K$ above, depending on $\overline{\nzero}$ and on $A$, hence only on $A, \, \ga, \, \om$, we get \eqref{rateconv-c} and \eqref{rateconv-flags} for all $n \ge 2$.

\end{proof}


\section{The proofs of the main statements}\label{proofs_section}

The first statement says that every $\taublockL$ is {\em H\"{o}lder} continuous in a neighborhood of a cocycle whose Lyapunov spectrum has a $\taugp$.  

\begin{theorem}\label{Holder-cont-thm}
Let $A \in \cocycles$ and let $\tau$ be a signature. Assume that for some $\ga > 0$ we have
\begin{equation*} 
\Lar (A) > \ga \quad \text{for all } \tausvr \, \rho
\end{equation*}

Then there are positive constants $\ep = \ep (A, \ga, \om)$, $K = K (A, \ga, \om)$, $\theta = \theta (A, \ga, \om)$ such that:
\begin{equation*} 
\abs{ \Lap (B_1) - \Lap (B_2) } < K \, \normr{ B_1 - B_2 }^\theta
\end{equation*}
for all cocycles $B_1, B_2 \in \cocycles$ such that $\normr{ A - B_i } < \ep$, $i = 1, 2$ and for all $\tausvp \, \pi$. 
\end{theorem}

\begin{proof} The idea of the proof is to relate the ``infinite" scale quantities $\Lap (\cdot)$ to the finite scale quantities $\Lapn{n} (\cdot)$ via the rate of convergence Lemma~\ref{rateconv-lemma}. Then,  since at finite scale we have locally H\"{o}lder continuity due to Lemma~\ref{cont-finite-lemma}, this will transfer over to the infinite scale quantities $\Lap (\cdot)$. 
However, comparing directly $\Lap (\cdot)$ to $\Lapn{n_0} (\cdot)$ leads to an error of order $\frac{\log n_0}{n_0}$, which is relatively too large, so it leads to a weak modulus of continuity for the limit $\Lap  (\cdot)$. 
In order to obtain an optimal (i.e. H\"{o}lder) modulus of continuity for $\Lap (\cdot)$, we use an intermediate scale $n_0 \ll n_1 \approx e^{\deltabar n_0} < \infty$.

\smallskip

Lemma~\ref{uniformgaps-lemma} and Lemma~\ref{rateconv-lemma} show that there exist constants $\nzero = \nzero (A, \ga, \om) \in \N$, $\ep_1 = \ep_1 (A, \ga, \om) > 0$, $K = K (A, \ga, \om) > 0$ such that if $B \in \cocycles$, $\normr{ A - B } < \ep$, then for all $n \ge \nzero$ we have:
\begin{align}
\Larn{n} (B) & > \ga & \ \text{ for all } \tausvr \, \rho \label{holdercont-eq1} \\
\abs{ \Lapn{n} (B) - \Lapn{2 n} (B) } & < K \, \frac{\log n}{n} & \ \text{ for all } \tausvp \, \pi \label{holdercont-eq2} \\
\abs{ \Lapn{n} (B) - \Lap (B) } & < K \, \frac{\log n}{n} & \ \text{ for all } \tausvp \, \pi \label{holdercont-eq3} 
\end{align} 

For fixed, appropriately chosen $n_0 \ge \nzero$, we will denote
\begin{equation}\label{holdercont-eta}
\eta_0 :=  K \, \frac{\log n_0}{n_0} 
\end{equation}

We may of course assume that $\nzero$ is large enough so that for any such $n_0 \ge \nzero$ we have:
\begin{equation}\label{holdercont-eq4}
4 \eta_0 =  4 K \, \frac{\log n_0}{n_0} < \frac{\ga}{2} 
\end{equation}

Fix constants $\delta, \, \deltabar$ such that:
\begin{align}
0 < \delta < \frac{\ga}{20} \label{holdercont-eq5}\\
0 < \deltabar < \delta \label{holdercont-eq-ddb}\\
\nzero^{-3/4} \le \deltabar \le c \delta^3 / 2 \label{holdercont-eq6}
\end{align}
Of course, for \eqref{holdercont-eq-ddb} and  \eqref{holdercont-eq6} to make sense, we need to make two other obvious assumptions on the magnitude of $\nzero$,  which only depend on $A$ and $\ga$.

From \eqref{holdercont-eq4} and \eqref{holdercont-eq5} we get:
$$\delta < \frac{\ga}{20} = \frac{\ga - \ga/2}{10} < \frac{\ga - 4 \eta_0}{10}$$
hence
\begin{equation}\label{holdercont-eq7}
0 < \delta < \frac{\ga - 4 \eta_0}{10}
\end{equation}

Estimates  \eqref{holdercont-eq1},  \eqref{holdercont-eq2} together with  \eqref{holdercont-eq7},  \eqref{holdercont-eq-ddb}, \eqref{holdercont-eq6}  say that we can apply the inductive step Theorem~\ref{indstep-thm} at {\em any} initial scale $n_0 \ge \nzero$ with $\eta_0$ given by \eqref{holdercont-eta}, $\ga_0 = \ga$, and its  conclusion will hold for $n_1 \approx e^{\deltabar n_0}$.

Now let $$\ep := \min \{ \ep_1, \ \frac{1}{2} \, e^{- 4 C \nzero} \}$$

For any cocycles $B_1, B_2 \in \cocycles$ such that $\normr{A - B_i} < \ep$, $i=1,2$, let
$$\normr{B_1 - B_2} =: h \ ( < 2 \ep )$$

Choose an integer $n_0$ such that $h \approx e^{-4 C n_0}$ and let  $n_1 := \intpart{ e^{\deltabar n_0} }$.

By the choice of $\ep$, we have $n_0 \ge \nzero$, so we can apply the inductive step Theorem~\ref{indstep-thm} (or rather Remark~\ref{indstep-remark}) at this scale $n_0$. We get:
\begin{align}
\abs{ \Lapn{n_1} (B_i) + \Lapn{n_0} (B_i) - 2 \Lapn{2 n_0} (B_i) } < C \, \frac{n_0}{n_1} \approx \notag \\
\approx \frac{C}{\deltabar} \, \frac{\log n_1}{n_1} = K \, \frac{\log n_1}{n_1} \label{holdercont-eq10}
\end{align} 
for $i = 1, 2$ and for any $\tausvp \, \pi$.

Moreover, since $n_1 > n_0 \ge \nzero$, we can apply \eqref{holdercont-eq3} at scale $n_1$ for $B_i$, $i = 1, 2$ and get:
\begin{equation}\label{holdercont-eq11}
\abs{ \Lapn{n_1} (B_i) - \Lap (B_i) }  < K \, \frac{\log n_1}{n_1}  \quad \text{ for all } \tausvp \, \pi 
\end{equation}

Combining \eqref{holdercont-eq10}, \eqref{holdercont-eq11}, for $i = 1, 2$ and for all $\tausvp \, \pi$ we have:
\begin{equation}\label{holdercont-eq12}
\abs{ \Lap  (B_i) + \Lapn{n_0} (B_i) - 2 \Lapn{2 n_0} (B_i) } < K \, \frac{\log n_1}{n_1} 
\end{equation}
where $n_1 \approx e^{\deltabar n_0}$.

This last estimate shows that the ``infinite" scale quantity $\Lap (\cdot)$ and the corresponding finite scale quantity $\Lapn{n_0} (\cdot)$ are {\em exponentially} close in the scale $n_0$.

Since $\normr{B_1 - B_2} = h \approx e^{- 4 C n_0}$, we can apply the finite scale continuity Lemma~\ref{cont-finite-lemma} at scale $n_0$ and $2 n_0$ and get:
\begin{align}
\abs{ \Lapn{n_0} (B_1) - \Lapn{n_0} (B_2) } \le h^{1/2} 
\nonumber \\
\abs{ \Lapn{2 n_0} (B_1) - \Lapn{2 n_0} (B_2) } \le h^{1/2} \label{holdercont-eq14}
\end{align}

Combining \eqref{holdercont-eq12} - \eqref{holdercont-eq14} we have:
\begin{align*}
\abs{ \Lap (B_1) - \Lap (B_2) } < 2 K \, \frac{\log n_1}{n_1} + 3 h^{1/2} \less K h^{\theta}
\end{align*}

The last inequality follows from
$$n_1 \approx e^{\deltabar n_0}, \quad h \approx e^{-4 C n_0}$$
hence
$$\frac{1}{n_1} \approx h^{\deltabar / 4 C}$$
so the H\"{o}lder exponent $\theta$ will be chosen such that $0 < \theta < \min \{ \frac{1}{2}, \, \frac{\deltabar}{ 4 C} \}$. 
\end{proof}

\medskip


We will prove that every Lyapunov exponents is continuous as a function of the cocycle, regardless of whether the cocycle at which we prove continuity has a gap pattern or not. The main ingredients in the proof are the continuity  of the 
$\taublockL$s, shown in Theorem~\ref{Holder-cont-thm},  and the lower semi-continuity of the $\taugapL$ shown in Corollary~\ref{lscgaps}. We first derive a simple general continuity lemma that synthesizes this information, which we then apply several times to establish continuity of the Lyapunov exponents. 

\begin{lemma}\label{abstract-cont-lemma}
Let $\tops$ be a topological space, let $a \in \tops$ be a point and let $f_1, f_2, \ldots, f_p \colon \tops \to \R$ be functions satisfying the following properties:

$\blob$ $f_1 \ge f_2 \ge \ldots \ge f_p$ at all points $x \in \tops$;

$\blob$ $f_1 (a) = f_2 (a) = \ldots = f_p (a)$;

$\blob$ $g := f_1 + f_2 + \ldots + f_p$ is continuous at $a$; 

$\blob$ $f_1$ is upper semi-continuous at $a$.

Then  for every $1 \le j \le p$, $f_j$ is continuous at $a$.

\end{lemma}

\begin{proof} Assume that $f_1$ is {\em not} lower semi-continuous at $a$. Then there is $\ep > 0$ and there is $x_k \to a$ such that $f_1 (x_k) \le f_1 (a) - \ep$ for all $k$. 

For every $1 \le j \le p$ we then have:
$$f_j (x_k) \le f_1 (x_k) \le f_1 (a) - \ep = f_j (a) - \ep$$

Summing up over $1 \le j \le p$, for all $k$ we have: 
 \begin{equation}\label{cont-lemma-eq1}
g (x_k) \le g (a) - p \, \ep \le g (a) - \ep
 \end{equation}

Since  $g$ is continuous at $a$, for $k$ large enough we have:
 \begin{equation*} 
 \abs{ g (x_k) - g (a)} < \frac{\ep}{2} 
 \end{equation*}
 which contradicts \eqref{cont-lemma-eq1}.

Therefore, $f_1$ is continuous at $a$.

\smallskip 

Since $f_2 (a) = f_1 (a)$ and $f_2 (x) \le f_1 (x)$ for all $x \in \tops$, and since $f_1$ is upper semi-continuous at $a$, then clearly $f_2$ is also upper semi-continuous at $a$. Moreover, since $f_1$ was shown to be continuous at $a$, it follows that $\text{ new }  g:= f_2 + \ldots + f_p$ is continuous at $a$, hence the continuity of  $f_2, \ldots, f_p$ at $a$ follows by induction.   

\end{proof}

\begin{theorem} \label{cont-all-Lyap-thm}
All Lyapunov exponents $L_i (\cdot)$, $1 \le i \le m$ are continuous functions on $\cocycles$.
\end{theorem}

\begin{proof} First note that the map $\cocycles \ni A \mapsto \Lam (A) \in \R$ is continuous, where according to the notations of Section~\ref{defs-nots_section}, $\Lam (A)$ refers to the sum $L_1 (A) + \ldots + L_m (A)$ of {\em all} Lyapunov exponents. This is obvious, since the product of all singular values of a matrix is the absolute value of its determinant, hence:
\begin{align*}
\Lamn{n} (A) = \int_{\T} \frac{1}{n} \, \log \, \abs{ \det \An{n} (x) } \, d x =   \int_{\T} \frac{1}{n} \, \log \, \prod_{j=n-1}^{0}\abs{ \det A (\transl^j \, x) } \, d x = \\
= \frac{1}{n} \, \sum_{j=n-1}^{0} \int_{\T} \log \, \abs{ \det A (\transl^j \, x) } \, d x = \int_{\T} \log \, \abs{ \det A  (x) } \, d x
\end{align*}
Therefore,
$$\Lam (A) = \lim_{n\to\infty} \Lamn{n} (A) =  \int_{\T} \log \, \abs{ \det A  (x) } \, d x$$
which is clearly a continuous function of $A \in \cocycles$.

Let $A \in \cocycles$ be a {\em fixed} cocycle. We prove continuity of the Lyapunov exponents $L_i (\cdot)$, $ 1 \le i \le m$ at $A$.  

Either all Lyapunov exponents of $A$ are equal, in which case we treat them as a single block 
$\Lam (A) = [ L_1 + \ldots + L_m ] (A)$, or there are gaps between some of them, which can be encoded by a signature $\tau$.
In other words, let $\tau = (\tau_1, \ldots, \tau_k)$ be such that  $L_i (A) > L_{i+1} (A)$ for some $1 \le i \le m$ if and only if $i=\tau_j$ for some $1 \le j \le k$. 

From Theorem~\ref{Holder-cont-thm}, all $\taublockL$s are continuous at $A$.  Hence for all $1 \le j \le k$, the blocks 
$\La_{\pi_{\tau,j}} (\cdot) = [ L_{\tau_{j-1}+1} + \ldots +  L_{\tau_{j}} ] (\cdot)$ are continuous functions at $A$. Then clearly the last (redundant) block $[ L_{\tau_{k+1}} + \ldots + L_m ] (\cdot) $ is also continuous at $A$.
To simplify notations, any such block will have the form and be denoted by:
\begin{equation*} 
\La_{j, p}  :=  L_{j+1} + \ldots + L_{p} 
\end{equation*}
where $1 \le j+1 \le p \le m$ and if $j+1 > 1$ then $L_j (A) > L_{j+1} (A)$, while if $p < m$ then $L_p (A) > L_{p+1} (A)$.

\smallskip

We apply Lemma~\ref{abstract-cont-lemma} to each of these blocks. Clearly the first three assumptions in Lemma~\ref{abstract-cont-lemma} hold for the functions $L_{j+1},  L_{j+2}, \ldots, L_p$ and their sum $\La_{j, p}$.   

If $j  = 0$,  then $L_{j+1} = L_1$ is known to (always) be upper-semicontinuous, so the fourth assumption in Lemma~\ref{abstract-cont-lemma} is also satisfied, and we can then conclude that $L_1, L_2, \ldots, L_p$ are all continuous at $A$. 

If $p = m$ we are done. If $p < m$, then $L_p (A) > L_{p+1} (A)$, so the corresponding Lyapunov spectrum gap $[ L_p - L_{p+1} ] (\cdot)$ is lower semi-continuous at $A$ by Corollary~\ref{lscgaps}. Since we now know that $L_p$ is continuous at $A$, it follows that $L_{p+1}$ is upper semi-continuous at $A$.  The fourth assumption in Lemma~\ref{abstract-cont-lemma} is then satisfied for the (next) block starting with $L_{p+1}$, and so we get continuity at $A$ of all Lyapunov exponents forming this block. Continuing this inductively we obtain continuity at $A$ of all Lyapunov exponents.

\end{proof}

\begin{theorem} \label{oseledets:cont} Given a signature $\tau$,
the Oseledets $\tau$-decomposition, and the 
Oseledets $\tau$-filtration, are continuous functions on
the open set of cocycles in $\cocycles$ with a $\tau$-gap pattern.
\end{theorem}

Given linear subspaces $V\subset W$, define
$$ W\ominus V:= W\cap V^\perp \;.$$

\begin{proposition} Given a signature $\tau$,
the map $\mathscr{F}^m_{\tau} \to \prod_{j=1}^k\Gr^m_{\tau_j-\tau_{j-1}}$,
$\underline{F}=(F_1,\ldots, F_k)\mapsto (F_{\tau_{j}}\ominus F_{\tau_{j-1}})_{1\leq j\leq k}$, is continuous.
\end{proposition}

\begin{proof}
It is enough to prove the continuity of the map
$\mathscr{F}^m_{(k,n)}\to \Gr^m_{n-k}$,  $(V,W)\mapsto W\ominus V$.
We can easily check the topology of $\Gr^m_n$ is characterized
by the following proximity: two subspaces $W,W'\in\Gr^m_n$ are
$\varepsilon$-close iff the orthogonal projection $\pi_{W,W'}:W\to W'$
has minimum expansion $m(\pi_{W,W'})\geq 1-\varepsilon$.
Given two close flags   $(V,W)$ and $(V',W')$,
assume that $m(\pi_{W,W'})\geq 1-\varepsilon$ and $m(\pi_{V,V'})\geq 1-\varepsilon$.
Given $x\in V$,
\begin{align*}
\norm{\pi_{V,W'\ominus V'}(x)}^2 &= 
\norm{\pi_{V,W'}(x)}^2 - \norm{\pi_{V,V'}(x)}^2\\
&\leq \norm{x}^2-(1-\varepsilon)^2\norm{x}^2\\
&= (1- (1-\varepsilon)^2)\norm{x}^2 = (2\varepsilon-\varepsilon^ 2)\norm{x}^2\;. 
\end{align*}
Similarly we prove that $\norm{\pi_{V',W\ominus V}}\leq \sqrt{2\varepsilon-\varepsilon^ 2}$.
Now, given $x\in V'$ and $y\in  W\ominus V$, since
$$ \langle \pi_{V',W\ominus V}(x), y\rangle =  \langle x, y \rangle
=  \langle x, \pi_{W\ominus V, V'}(y)\rangle \;, $$
it follows that $\pi_{W\ominus V, V'} =   (\pi_{V',W\ominus V})^\ast$.
Hence 
$$\norm{ \pi_{W\ominus V, V'} }= \norm{ (\pi_{V',W\ominus V})^\ast }
= \norm{ \pi_{V',W\ominus V} }= \sqrt{2\varepsilon-\varepsilon^ 2}\;.$$
Finally, since  $m(\pi_{W,W'})\geq 1-\varepsilon$,
given $y\in  W\ominus V$,
\begin{align*}
(1-\varepsilon)^2\norm{y}^2 &\leq \norm{ \pi_{W,W'}(y) }^2\\
&=  \norm{ \pi_{W\ominus V,V'}(y) }^2 + \norm{ \pi_{W\ominus V,W'\ominus V'}(y) }^2\\
&\leq  (2\varepsilon-\varepsilon^2)\norm{y}^2 + \norm{ \pi_{W\ominus V,W'\ominus V'}(y) }^2\;,
\end{align*}
which implies that
$$ m\left( \pi_{W\ominus V,W'\ominus V'}\right)\geq \sqrt{1-4\varepsilon +2\,\varepsilon^2}\;.$$
Thus $W\ominus V$ is close to $W'\ominus V'$.
\end{proof}

\begin{proof}[Proof of Theorem~\ref{oseledets:cont}]
By the previous proposition it is enough proving the continuity of Oseledets $\tau$-filtration $\filt(A)$.
We already know, see Lemma~\ref{cont-filtr-finite-scale}, 
that at any fixed finite scale $n$, $\filtn{n}(A)$ is continuous.
We also know, from Lemma~\ref{rateconv-lemma}~(\ref{rateconv-flags}),
that $\filtn{n}(A)$ converges uniformly to $\filt(A)$.
Therefore, the limit function $\filt(A)$ is continuous.
\end{proof}

\begin{remark}\rm{
The H\"older continuity of the Oseledets $\tau$-filtration $\filt(A)$ may fail due to the lack of a uniform  H\"older exponent
for the H\"older continuity of the finite scale functions $\filtn{n}(A)$ (see remark~\ref{Hoder:Oseledets:remark}).
}
\end{remark}

\smallskip

\section{Some consequences of the main stattements}\label{consequences_section}
\subsection*{Possible extensions to other types of dynamics.} Establishing the LDT~\ref{LDT-thm} is the only place where we used the analyticity assumption on the cocycles. It is also the only place where we used the Diophantine condition on the frequency, or even where anything specific about the base dynamics (besides its ergodicity) was needed. 

Therefore, suppose we have a cocycle $(\transl, A)$, where $\transl$ is some ergodic transformation and $A$ belongs to some matrix valued space of functions (possibly more general than analytic). Assume that in this context we have a LDT of the form
\begin{equation*} 
\abs{ \{ x \in X \colon \abs{\frac{1}{n} \log s (\An{n} (x))  - \Lasn{n} (A) } > \delta (n) \} } < \ep (n)
\end{equation*}
where $\ep (n) \ll \delta (n)$ as $n \to \infty$.

Then the arguments  used in this paper for proving continuity of the Lyapunov exponents  would apply in the same manner. We would thus obtain global continuity of all Lyapunov exponents and local quantitative continuity of the Lyapunov spectrum blocks associated to a gap pattern, where the modulus of continuity would depend on the  sharpness of  $\ep (n)$. The next subsection presents such an extension.

\smallskip

\subsection*{Diophantine translations on higher dimensional tori.} 
Consider the transformation 
$$\transl \, \underline{x} = \transl_{\underline{\om}} \, \underline{x} := \underline{x} + \underline{\om} $$ 
on the torus $\T^d$ of dimension $ d \ge 1$.

If we assume that the frequency $\underline{\om}$ satisfies a standard Diophantine condition, 
then an estimate in the spirit of   
\eqref{quantB}, but {\em weaker} holds for the transformation $\transl_{\underline{\om}}$ as well (see  \cite{B-book} or Proposition 4.1 in \cite{sK2}). For any cocycle $A \in \cocyclesTd$, this in turn leads to a LDT of the form:
\begin{equation}\label{LDT-other-dyns}
\abs{ \{ \underline{x} \in \T^d \colon \abs{\frac{1}{n} \log s (\An{n} (\underline{x}))  - \Lasn{n} (A) } > n^{- a} \} } < e^{- c n^{b}}
\end{equation}
for some absolute constants $a, \, b > 0 $. 

Along with the Avalanche Principle~\ref{AP-practical}, this LDT~\eqref{LDT-other-dyns} allows us to establish the inductive step Theorem\ref{indstep-thm}. However, since the deviation set is not as sharp, the size of the {\em next} scale $n_1$ will have to be smaller, namely
$$n_1 \ll e^{ c n_0^b}$$
Because of this, the rate of convergence~\eqref{rateconv-c} will be of the form:
$$\abs{ \Lapn{n} (B) - \Lap (B) } < K \, \frac{(\log n)^{1/b}}{n} $$
which in turn will lead to a weaker, log-H\"{o}lder modulus of continuity for the $\taublockL$s, while the global continuity statement of each individual Lyapunov exponent will be the same.  We conclude:

\begin{theorem} \label{continuity-Lyap-Td}
Assume that the base dynamics is given by a Diophantine translation on $\T^d$, where $d \ge 1$. Then all Lyapunov exponents are continuous functions on $\cocyclesTd$. 

Moreover, if  $\tau$ is a signature and if  $A$ has a $\taugp$, then the corresponding $\taublockL$s are log - H\"older continuous functions in a neighborhood of $A$. 
\end{theorem}

The continuity of the Oseledets decomposition, resp. filtration,
also holds for cocycles on higher dimensional tori.  
The only difference  rests in proving the
finite scale Lemma~\ref{cont-filtr-finite-scale}. The proof of this
 lemma is based on a one dimensional fact:
the spectrum of an  analytic symmetric matrix valued function
of a one real variable can be locally parametrized by finitely many
 analytic functions. This difficulty is easily overcome
by looking at a cocycle $A:\T^d\to\GLmR$ as a compact family
$\{A(x,\cdot):\T\to\GLmR\}_{x\in\T^{d-1}}$ of one variable holomorphic functions.

\smallskip

\subsection*{Schr\"{o}dinger cocycles.} Our main results apply to and extend the statements in \cite{Schlag} which establish H\"{o}lder continuity in the ``energy''  parameter near points where the Lyapunov spectrum is simple. 

Indeed, consider as in \cite{Schlag} a compact metric space $(\E, \dist)$ and a continuous function $A = A (x, E) \colon \T \times \E \to \GLmR$ such that  $x \mapsto A (x, E)$ is analytic uniformly in $E \in \E$ and $E \mapsto A (x, E)$ is H\"{o}lder uniformly in $x \in \T$. 

Then there is $r > 0$ such that $A (\cdot, E) \in \cocycles$ (i.e. with the same width of analyticity for all $E \in \E$) and the map 
$$ \E \ni E \mapsto    A (\cdot, E) \in \cocycles$$
is H\"{o}lder  continuous.

Therefore, if $\tau$ is any signature (including for instance the signature $\tau = (1, 2, \ldots, m-1)$, which encodes simple Lyapunov spectrum) and if $E_0 \in \E$ is such that the cocycle $A (\cdot, E_0) \in \cocycles$ has a $\taugp$, then using Theorem~\ref{Holder-cont-thm} and the fact that a composition of (locally) H\"{o}lder  continuous functions  is also (locally) H\"{o}lder  continuous, we conclude that 
$$\Lap (E) := \Lap (A (\cdot, E))$$ is H\"{o}lder  continuous near $E = E_0$ for any $\tausvp$  $\pi$. In particular, if $\tau$ were the signature  $\tau = (1, 2, \ldots, m-1)$, then each Lyapunov exponent $L_i (E)$ would be H\"{o}lder  continuous near $E=E_0$, or everywhere on $\E$, if the Lyapunov spectrum were simple at each point $E \in \E$ (this last statement is due to the lower semicontinuity of the gaps shown in Corollary~\ref{lscgaps}, and on the compacity of $\E$.
Moreover, regardless of any gap pattern in the spectrum, from Theorem~\ref{cont-all-Lyap-thm} we get that for $1 \le i \le m$
$$\E \ni E \mapsto L_i (E) := L_i ( A(\cdot, E)) \in \R$$
are all continuous functions. 

\smallskip

We may specialize the cocycles above to ones associated to ``weighted" band lattice Schr\"odinger operators (also called Jacobi operators) .  This model includes all finite range hopping Schr\"odinger operators, on integer and on band-integer lattices. They act on the space of square summable sequences  $l^2 (\Z, \R^d)$ by:
\begin{equation*}
[H_\la (x) \, \vpsi]_n := - (W_{n+1} (x) \, \vpsi_{n+1} + W_n^T (x) \, \vpsi_{n-1}  + R_n (x) \, \vpsi_n)  + \la \, D_n (x) \,  \vpsi_n
\end{equation*}
where $x \in \T$ is a phase parameter,  $\la > 0$ is a coupling constant,  $\om \in \R \setminus \Q$ is a frequency defining the base dynamics, $W_n (x) := W (x + n \om)$ and $R_n (x) := R (x + n \om)$ are the weights encoding the hopping amplitude and $D_n (x) := D (x + n \om)$ is the potential matrix. We assume that the matrix-valued functions $W (x), R(x)$ and  $ D (x)$ defining the weights and the potential are analytic. Moreover, we denote by $W^T (x)$  the transpose of $W (x)$ and we assume that $R(x)$ and $D (x)$ are symmetric matrices, which ensure that the above operator is self-adjoint.

The associated Schr\"odinger equation
\begin{equation}\label{schr-eq}
H_\la (x) \, \vpsi = E \, \vpsi \quad \text{for } E \in \R \ \text{ and } \vpsi = [ \vpsi_n ]_{n \ge 1}
\end{equation}
gives rise to a family $A_{\la, E} (x)$ of linear cocyles indexed by the coupling constant $\la > 0$ and the energy $E \in \R$.   

In \cite{pDsK1} we have shown that under generic conditions on the matrix-valued functions $W (x)$ and $D (x)$,  and for all irrational frequencies $\om$, if $\la \ge \la_0 (R, W, D)$, then all Lyapunov exponents associated with the Schr\"odinger equation \eqref{schr-eq}
are of order $\log \la$.

\smallskip

Assuming that the matrix-valued function $W (x)$ is invertible for every $x \in \T$, and that the frequency $\om$ is Diophantine, the results in this paper imply continuity of the Lyapunov exponents as functions of $(\la, E)$ with no restrictions on these parameters. Moreover, they also imply  joint local H\"older continuity of the Lyapunov exponents (or of the Lyapunov spectrum blocks) near points $(\la, E)$ where the Lyapunov spectrum is simple (or it has a gap pattern, respectively). 
In a forthcoming paper we will give sufficient conditions for  such general Schr\"odinger cocycles to have simple Lyapunov spectrum (or any gap pattern).

\smallskip

\subsection*{Complex valued cocycles.} The {\em realification} of a matrix $g\in\GL(m,\C)$ 
is the matrix $\widetilde{g}\in\GL(2m,\R)$ obtained from $g$
replacing each complex entry $a+i\,b$ by the $2\times 2$-bock
$$ \left[\begin{array}{rr}
a & -b \\b & a 
\end{array}\right] \;.$$
By definition, the singular values of a complex matrix $g$
are the square roots of the conjugate positive definite hermitian matrices
$g^\ast g$ and $g g^ \ast$. The associated eigenspaces of $g^\ast g$ and $g g^ \ast$ are complex spaces. Hence each $g\in\GL(m,\C)$ has 
$m$ singular values, still denoted by $s_1(g)\geq \ldots \geq s_m(g)>0$,
 possibly repeated. It is clear that the realification $\widetilde{g}$ has exactly the same singular values but with doubled  multiplicity,
 i.e.
 $$ s_{2i-1}(\widetilde{g}) = s_{2i}(\widetilde{g}) =s_i(g),\;
 \text { for } \; 1\leq i\leq m\;. $$
 Now, given a complex cocycle $A\in \cocycle{\T}{m}{\C}$,
 we define its realification $\widetilde{A}\in \cocycle{\T}{2m}{\R}$
 to be the function $\widetilde{A}:\T\to\GL(2m,\R)$ where for each
 $x\in\T$, $\widetilde{A}(x)$ is the realification of $A(x)\in\GL(m,\C)$.
 Because of~(\ref{formula-lyap-i}), the Lyapunov exponents of $A$ and
 $\widetilde{A}$ are the same, i.e.,
 $$ L_{2i-1}(\widetilde{A}) = L_{2i}(\widetilde{A}) =L_i(A),\;
 \text { for } \; 1\leq i\leq m\;. $$
As before, a complex matrix $g\in\GL(m,\C)$ is said to have a 
{\em $\tau$-gap pattern} \, iff\, $s_{\tau_j}(g)/s_{\tau_j+1}(g)>1$,
for every $j=1,\ldots, k$. An analogous definition is adopted for
a complex cocycle having a $\tau$-gap pattern. 
Then a matrix $g\in\GL(m,\C)$, resp. a cocycle
$A\in\cocycle{\T}{m}{\C}$, has a $\tau$-gap pattern \, iff\, 
$\widetilde{g}$, resp. $\widetilde{A}$, has a $2\tau$-gap pattern.
 
Given a signature $\tau=(\tau_1,\ldots, \tau_k)$,
with $1\leq \tau_1<\tau_2<\ldots <\tau_k<m$,
a {\em complex $\tau$-flag} is a family
$\underline{F}=(F_1,\ldots, F_k)$ of complex linear subspaces
$F_1\subset F_2\subset \ldots \subset F_k\subset \C^m$ such that
${\rm dim}_\C(F_j)=\tau_j$, for $1\leq j\leq k$.
We shall denote by $\mathscr{F}^m_\tau(\C)$ the manifold
 of complex $\tau$-flags, and write $\mathscr{F}^m_\tau(\R)$
 to emphasize the real character of a flag manifold.
Identifying $\C^m\equiv\R^{2m}$,
each complex linear subspace $V\subset \C^m$ can be viewed as
a real  linear subspace $V\subset \R^{2m}$ of twice the dimension.
This identification induces a natural embedding
 $\mathscr{F}^m_\tau(\C) \hookrightarrow \mathscr{F}^{2m}_{2\tau}(\R)$.
 When  $\tau=(k)$, $\tau$-flags are $k$-dimensional complex subspaces
 and the flag manifold  $\mathscr{F}^m_\tau(\C)$ is called a
 {\em complex Grassmann manifold}, and denoted by $\Gr^m_k(\C)$.
Again there is a natural embedding
  $\Gr^m_k(\C) \hookrightarrow \Gr^{2m}_{2k}(\R)$.
The {\em most expanding $\tau$-flags}  $\hatv_{\tau,\pm}(g)\in\mathscr{F}^m_\tau(\C)$ of a complex matrix $g\in\GL(m,\C)$
are defined  as in the real case. They are $\tau$-flags
generated by the $g$ singular eigen-basis of $\C^m$.
Using the previous embedding we  identify
  $\hatv_{\tau,\pm}(g) \equiv \hatv_{2\tau,\pm}(\widetilde{g})$.
Having established this sort of  `dictionary' for the realification
of complex matrices and cocycles, it should be clear that theorems~\ref{main-thm1},~\ref{main-thm2} and~\ref{oseledets:continuity} apply to complex cocyles as well.
The assumptions on  the complex cocyles translate to the corresponding
assumptions on their realified counterparts.
 Hence we have:

 \begin{theorem} 
 All Lyapunov exponents are continuous functions on $\cocycle{\T}{m}{\C}$.
 
Given a cocycle  
$A \in \cocycle{\T}{m}{\C}$ and a signature $\tau$,
if $A$ has a $\taugp$, then the corresponding $\taublockL$s are H\"older continuous functions in a neighborhood of $A$.

Moreover, the Oseledets $\tau$-decomposition and the 
Oseledets $\tau$-filtration are continuous functions on
the open set of cocycles in $\cocycle{\T}{m}{\C}$ with a $\tau$-gap pattern.
\end{theorem}


\section{Final remarks}\label{final_section}
There are several different directions in which the results in this paper might be extended.  

To our knowledge, any {\em quantitative} continuity result for Lyapunov exponents of quasiperiodic cocycles requires some arithmetic assumptions on the frequency. Here we assume a Diophantine condition.  In a recent paper (see \cite{YouZhang-Holder}), J. You and S. Zhang obtain H\"{o}lder continuity of $\SL (2, \R)$ Schr\"{o}dinger cocycles under much weaker arithmetic assumptions on the frequency (the base dynamics is the translation on $\T$). It is conceivable that their argument can be abstracted and used to prove the LDT~\ref{LDT-thm}  for more general frequencies than Diophantine.

For $\SL (2, \R)$ Schr\"{o}dinger cocycles, the approach to proving continuity and positivity of the Lyapunov exponent via an AP and a LDT has proven robust enough to work for potential functions more general than analytic. More specifically (see \cite{sK2}), by  using successive polynomial approximations adapted to each scale, these types of results were extended to potential functions in a Gevrey class. This approach might prove successful for general, higher dimensional cocycles as well.

The arguments  in this paper depend on having uniform measurements on the size of the cocycle and of its inverse. It would be interesting to see whether with the help of a finer analysis, anything survives  when the cocycle is allowed to have singularities (i.e. it is not necessarily invertible at every point).

Throughout this paper, the base dynamics is fixed and given by a Diophantine frequency, and we are only concerned with (quantitative) continuity in the cocycle. An important related problem is that of continuity for any irrational frequency as well as joint continuity in the frequency and cocycle.

These last two problems (more general cocycles, joint continuity) were elegantly treated in \cite{AJS} for base dynamics given by translations on the {\em one-dimensional} torus $\T$. The method in \cite{AJS} is based on the complexification of the phase variable, and it does  not have an extension to other types of dynamics, such as translations on higher dimensional tori. However, in the case of  $\SL (2, \R)$ Schr\"{o}dinger cocycles, this type of result is available for translations on the higher dimensional torus (see \cite{B-contpos-Td}), which encourages us to believe they might hold for general, higher dimensional cocycles as well.

Moreover, it would be interesting to adapt the techniques in this paper to other types of base dynamics, such as the skew-translation on the two dimensional torus or possibly some {\em non} quasiperiodic transformations.  

\smallskip

We plan to investigate all of these directions in our future projects.


\section*{Acknowledgments}

The second author would like to thank Christian Sadel for a conversation they had regarding his own version of a higher dimensional Avalanche Principle,  as well as for sharing an earlier draft of \cite{AJS}. 

Both authors were partially supported by Funda\c c\~ao para a Ci\^encia e a Tecnologia through the Program POCI 2010 and the Project ``Randomness in Deterministic Dynamical Systems and Applications'' (PTDC-MAT-105448-2008).

\end{document}